\documentclass [10pt] {article}
\usepackage{amsfonts}
\usepackage{amsmath}
\usepackage{amssymb}
\usepackage{stmaryrd}
\usepackage{hyperref}
\usepackage{bm}
\usepackage[alphabetic]{amsrefs}
\usepackage{fullpage,theorem}
\usepackage{multirow}
\usepackage{tensor}
\usepackage{caption	}
\usepackage{xcolor}
\usepackage{color, colortbl}
\usepackage[first=0,last=9]{lcg}

\usepackage{etex}
\usepackage[all]{xy}

\definecolor{Gray}{gray}{0.9}
\newcolumntype{g}{>{\columncolor{Gray}}c}

\newtheorem{thm}{Theorem}[section]
\newtheorem{cor}[thm]{Corollary}
\newtheorem{lem}[thm]{Lemma}
\newtheorem{prop}[thm]{Proposition}

\theorembodyfont{\rmfamily}
\newtheorem{defn}[thm]{Definition}
\newtheorem{exa}[thm]{Example}

\newtheorem{nota}[thm]{Notation}
\newtheorem{rem}[thm]{Remark}

\numberwithin{equation}{section}
\newenvironment{proof}{\noindent \emph{Proof.}}{\hspace{\stretch{1}}$\Box$}

\newcommand{\parderv}[2] {\frac{\partial#1}{\partial#2}}

\newcommand{\mcB} {\mathcal{B}}

\newcommand{\mcD} {\mathcal{D}}

\newcommand{\mcF} {\mathcal{F}}

\newcommand{\mcH} {\mathcal{H}}

\newcommand{\mcK} {\mathcal{K}}
\newcommand{\mcL} {\mathcal{L}}
\newcommand{\mcM} {\mathcal{M}}
\newcommand{\mcN} {\mathcal{N}}

\newcommand{\mcU} {\mathcal{U}}

\newcommand{\mcY} {\mathcal{Y}}

\newcommand{\dd} {\mathrm{d}}

\newcommand{\ii} {\mathrm{i}}

\newcommand{\dbl} {[\![}
\newcommand{\dbr} {]\!]}

\newcommand{\ind} {\indices}

\newcommand{\lp} [1] {{\left( #1 \right. }}
\newcommand{\rp} [1] {{\left. #1 \right) }}
\newcommand{\lb} [1] {{\left[ #1 \right. }}
\newcommand{\rb} [1] {{\left. #1 \right] }}

\newcommand{\im} {\mathop{\mathrm{im}}}
\newcommand{\End} {\mathrm{End}}

\newcommand{\gr} {\mathrm{gr}}

\newcommand{\SO} {\mathrm{SO}}

\newcommand{\OO} {\mathrm{O}}
\newcommand{\U} {\mathrm{U}}

\newcommand{\SU} {\mathrm{SU}}
\newcommand{\Gr} {\mathrm{Gr}}

\newcommand{\Tgt} {\mathrm{T}}

\newcommand{\Sim} {\mathrm{Sim}}

\newcommand{\so} {\mathfrak{so}}

\newcommand{\g} {\mathfrak{g}}

\newcommand{\prb} {\mathfrak{p}}

\newcommand{\uu} {\mathfrak{u}}
\newcommand{\su} {\mathfrak{su}}

\newcommand{\simalg} {\mathfrak{sim}}
\newcommand{\mfz} {\mathfrak{z}}
\newcommand{\mfF} {\mathfrak{F}}

\newcommand{\mfC} {\mathfrak{C}}
\newcommand{\co} {\mathfrak{co}}
\newcommand{\cu} {\mathfrak{cu}}
\newcommand{\mfK} {\mathfrak{K}}
\newcommand{\mfN} {\mathfrak{N}}
\newcommand{\mfV} {\mathfrak{V}}
\newcommand{\mfU} {\mathfrak{U}}

\newcommand{\mfL} {\mathfrak{L}}

\newcommand{\mfS} {\mathfrak{S}}
\newcommand{\mfD} {\mathfrak{D}}

\newcommand{\mfA} {\mathfrak{A}}
\newcommand{\mfR} {\mathfrak{R}}

\newcommand{\CP} {\mathbb{CP}}

\newcommand{\R} {\mathbb{R}}
\newcommand{\C} {\mathbb{C}}
\newcommand{\Z} {\mathbb{Z}}

\newcommand{\Cl} {\mathcal{C}\ell}

\newcounter{mnotecount}[section]
\renewcommand{\themnotecount}{\thesection.\arabic{mnotecount}}

\newcommand{\mnote}[1]
{\protect{\stepcounter{mnotecount}}$^{\mbox{\footnotesize
$
\bullet$\themnotecount}}$ \marginpar{
\raggedright\tiny\em
$\!\!\!\!\!\!\,\bullet$\themnotecount: #1} }

\begin{document}
\title{The curvature of almost Robinson manifolds}
 \author{Arman Taghavi-Chabert\\
 {\small Masaryk University, Faculty of Science, Department of Mathematics and Statistics,}\\
  {\small Kotl\'{a}\v{r}sk\'{a} 2, 611 37 Brno, Czech Republic } }
\date{}

\maketitle

\begin{abstract}
An almost Robinson structure on an $n$-dimensional Lorentzian manifold $(\mcM,g)$, where $n=2m+\epsilon$, $\epsilon \in \{ 0 ,1 \}$, is a complex $m$-plane distribution $\mcN$ that is totally null with respect to the complexified metric, and intersects its complex conjugate in a real null line distribution $\mcK$, say. When $\mcN$ and its orthogonal complement $\mcN^\perp$ are in involution, the line distribution $\mcK$ is tangent to a congruence of null geodesics, and the quotient of $\mcM$ by this flow acquires the structure of a CR manifold. In four dimensions, such a congruence is shearfree.

We give classifications of the tracefree Ricci tensor, the Cotton-York tensor and the Weyl tensor, invariant under i) the stabiliser of a null line, and ii) the stabiliser of an almost Robinson structure. For the Weyl tensor, these are generalisations of the Petrov classification to higher dimensions. Since an almost Robinson structure is equivalent to a projective pure spinor field of real index $1$, the present work can also be viewed as spinorial classifications of curvature tensors.

We illustrate these algebraic classifications by a number of examples of higher-dimensional general relativity that admit integrable almost Robinson structures, emphasising the degeneracy type of the Weyl tensor in each case.
\end{abstract}

\section{Introduction}
Let $(\mcM,g)$ be a Lorentzian manifold of dimension $n=2m+\epsilon$, $\epsilon \in \{0,1\}$. There are two natural geometric structures we can endow $(\mcM,g)$ with:
\begin{enumerate}
\item \label{item-line} either a preferred real null line distribution, whereby the structure group of the frame bundle is reduced to $\Sim(n-2) \subset \SO(n-1,1)$,
\item \label{item-plane} or a preferred totally null \emph{complex} distribution of rank $m$ intersecting its complex conjugate in a real null line distribution,  whereby the structure group is reduced to $\Sim(m-1,\C) \subset \Sim(n-2) \subset \SO(n-1,1)$.
\end{enumerate}
This latter structure will be referred to as an \emph{almost Robinson structure} \cites{Trautman2002,Nurowski2002}. It has also been referred to as an \emph{optical structure} in other places e.g. \cites{Nurowski1996,Taghavi-Chabert2011}. Clearly, structure \ref{item-plane} implies structure \ref{item-line}, and the additional geometric datum here can be seen to be a \emph{Hermitian structure} on the fibers of the screenspace bundle of the real null line distribution.

The main aims of the article are
\begin{itemize}
\item  to give $\Sim(n-2)$-invariant classifications of the Weyl tensors and other curvature tensors;
\item  to give $\Sim(m-1,\C)$-invariant classifications of the Weyl tensors and other curvature tensors;
\item to apply the classifications to known solutions of Einstein's equations.
\end{itemize}

\subsection{Motivation}
The motivation for the study of such geometrical structures comes from four-dimensional general relativity, where they are equivalent and intrinsically related to Petrov's classification of the Weyl tensor \cite{Petrov2000}.

In four-dimensional general relativity, there is a fundamental relation between the existence of \emph{shearfree congruences of null geodesics (SCNG)} on a spacetime $(\mcM,g)$ and \emph{algebraically degenerate or special} solutions to certain field equations. For instance, a $2$-form $F \ind{_{ab}}$  that is a solution of the vacuum Maxwell equations $\nabla \ind{_{[a}} F \ind{_{bc]}} = 0$, and $\nabla \ind{^b} F \ind{_{ab}} = 0$ is algebraically special if and only if $(\mcM,g)$ admits a SCNG -- this is the content of the Robinson theorem \cite{Robinson1961}. Another important example comes from the study of Einstein's field equations: the Goldberg-Sachs theorem \cites{Goldberg2009} tells us that the Weyl tensor $C \ind{_{abcd}}$ of an Einstein spacetime $(\mcM,g)$ is algebraically special, i.e. of \emph{Petrov type II or more degenerate} \cite{Petrov2000} if and only if $(\mcM,g)$ admits a SCNG. In both cases, the algebraic degeneracy of $F \ind{_{ab}}$ and $C \ind{_{abcd}}$ can be characterised by the existence of a repeated \emph{principal null direction (PND)}, i.e. a null vector $k^a$ such that $F \ind{_{ab}} k \ind{^a} = 0$ and $k \ind{^d} k \ind{^e} C \ind{_{ade[b}} k \ind{_{c]}} = 0$, respectively, and such a vector field generates the SCNG by virtue of the field equations.

A further development in the theory came from the introduction of spinors, which simplified much of the formalism of general relativity, and stems essentially from the remark that in four dimensions, a null vector field is equivalent to a chiral spinor field (up to scale): such a spinor defines a complex null $2$-plane distribution, which intersects its complex conjugate in a line bundle spanned by this real null vector field. Quite naturally, to the notion of principal null direction corresponds a notion of \emph{principal spinor}, and the Petrov classification can then be expressed more transparently in this context \cites{Witten1959,Penrose1960}. Similarly, the geometric properties of a PND can be translated into this formalism: a SCNG is equivalent to its corresponding principal spinor being \emph{foliating} in the sense that its associated complex null $2$-plane distribution is integrable. Many of the classical results of general relativity such as the Robinson and Goldberg-Sachs theorems can then be viewed from this spinorial angle, as explained in the monographs \cites{Penrose1984,Penrose1986}.

How does all this generalise to higher dimensions? Unless $n=4$, the correspondence between null lines and totally null complex $m$-planes is not one-to-one: each null line lifts to a compact subset of the grassmannian of totally null complex $m$-planes of dimension $(m-1)(m-2(1-\epsilon))$. This implies that, for spacetimes of dimensions higher than four, one can define two distinct, yet related, algebraic classifications of the Weyl tensor, generalising the Petrov classification:
\begin{itemize}
\item one, put forward in \cites{Coley2004,Pravda2004}, is based on the notion of principal null directions, and is now known as the \emph{null alignment formalism}. It has given rise to a large amount of literature, generalising, or at least partially, results from four to higher general relativity - for a survey see \cite{Ortaggio2013}. Our classification essentially emphasises the $\Sim(n-2)$-invariance of the null alignment formalism.
\item the other, already advocated in \cites{Hughston1995,Jeffryes1995,Taghavi-Chabert2012a,Taghavi-Chabert2013}, is based on the notion of principal spinors, and provides a convenient setting in which generalisations of the Robinson and Kerr theorems \cite{Hughston1988}, and to some extent, the Goldberg-Sachs theorem \cites{Taghavi-Chabert2011,Taghavi-Chabert2012} can be formulated. Since an almost Robinson structure is also equivalent to the line spanned by a \emph{pure} spinor field \emph{of real index $1$}, the $\Sim(m-1,\C)$-invariant classification of the Weyl tensor is essentially a spinorial classification.
\end{itemize}
Furthermore, we can view the $\Sim(m-1,\C)$-invariant classification of the Weyl tensor as a refinement of the $\Sim(n-2)$-invariant classification, which also partakes of the $\U(m-1)$-invariant classifications of the Weyl and other curvature tensors of almost Hermitian manifolds found in \cites{Tricerri1981,Falcitelli1994}.

Finally, as shown in \cites{Mason2010,Taghavi-Chabert2011}, \emph{integrable} almost Robinson structures occur in important higher-dimensional solutions to Einstein's field equations such as the Kerr black hole \cites{Myers1986,Gibbons2005,Hawking1999,Chen2006} or the black ring \cite{Emparan2002}. This is unlike shearfree congruences of null geodesics, which have not been as ubiquitous as in dimension four.

\subsection{Structure of the paper}
The structure of the paper is as follows. Section \ref{sec-algebra} lays the algebraic foundation of the paper. We first review the properties of the stabiliser $\simalg(n-2)$ of a (real) null line in $n$-dimensional Minkowski space. We broadly follow the theory of parabolic Lie algebras of \cite{vCap2009} by making use of invariant filtrations and gradings on vector spaces. We then show how a Robinson structure, i.e. a conjugate pair of totally null complex $m$-planes intersecting in a real null line, induces a Hermitian structure on the orthogonal complement of this null line. Drawing from standard results of Hermitian geometry, e.g. \cite{Salamon1989}, we describe the stabiliser $\simalg(m-1,\C)$ in the Lie algebra $\so(n-1,1)$ of a Robinson structure and its irreducible representations. In addition, Proposition \ref{prop-twistor} describes the space of all Robinson structures on Minkowski space, while Propositions \ref{prop-null2Rob} the space of all Robinson structures intersecting a given real null line.

This calculus is then extended in section \ref{sec-curvature} to the classifications of curvature tensors. We first focus on the $\Sim(n-2)$-invariant case: Propositions \ref{prop-Ricci-sim}, \ref{prop-CY-sim}, \ref{prop-CY-sim6}, \ref{prop-Weyl-sim} and \ref{prop-Weyl-sim6} are $\Sim(n-2)$-invariant classifications of the tracefree Ricci tensor, the Cotton-York tensor and the Weyl tensor -- the relation to the work of \cites{Coley2004,Pravda2004,Ortaggio2009b} is explained in section \ref{rem-PW-types}. We then go a step further and classify these same curvature tensors now with respect to a Robinson structure, leading to Propositions \ref{prop-Ricci-rob}, \ref{prop-CY-rob} and \ref{prop-Weyl-rob}. All these classifications are expressed in terms of diagrams, which encode the invariance of the irreducible decompositions of the curvature tensors under $\Sim(n-2)$ or $\Sim(m-1,\C)$.

Section \ref{sec-geomexa} is concerned with the geometric applications of the classifications of section \ref{sec-curvature}. After a brief review of the geometry associated to a null line distribution, we shift our attention to Robinson manifolds, i.e. Lorentzian manifolds equipped with an integrable almost Robinson structure. We distinguish various degrees of integrability in odd dimensions in Definition \ref{def-Robinson-manifold}. We recast some of the previous results of the author \cites{Taghavi-Chabert2011,Taghavi-Chabert2012,Taghavi-Chabert2012a,Taghavi-Chabert2013} within the framework of the present $\Sim(m-1,\C)$-invariant classification of the Weyl tensor. Proposition \ref{prop-int-cond-Robinson} gives necessary Weyl curvature conditions for the existence of a Robinson structure. This yields a natural definition of an almost Robinson structure \emph{aligned} with the Weyl tensor - Definition \ref{def-aligned-Rob}. We exhibit aligned Robinson structures in the Petrov type G static KK bubble -- Example \ref{exa-static-KK5}. Proposition \ref{prop-multi-Rob} and Corollary \ref{cor-multi-int-Rob} establish a relation between a $\Sim(n-2)$-invariant algebraic degeneracy of the Weyl tensor (a `hybrid of Petrov types I and II(c)' in the null alignment formalism) and the existence of infinitely many aligned almost Robinson structures.

Based on the notion of algebraically special spacetimes introduced in \cites{Taghavi-Chabert2011,Taghavi-Chabert2012}, a definition of \emph{repeated} aligned almost Robinson structures is given in Definition \ref{defn-alg-sp}. The algebraically special condition is then reformulated in the present context in Proposition \ref{prop-alg-sp}. Under additional curvature assumptions, this forms the basis of a higher-dimensional version of the Goldberg-Sachs theorem given in the same references, and restated as Theorem \ref{thm-GS}. Proposition \ref{prop-multi-special-Rob} 
establishes a relation between a $\Sim(n-2)$-invariant algebraic degeneracy of the Weyl tensor (a `hybrid of Petrov types II and III(b)' in the null alignment formalism) and the existence of infinitely many repeated aligned almost Robinson structures.

We then give a number of examples of metrics locally admitting repeated aligned (almost) Robinson structures. While many of these examples are not new, we emphasise the relation between the algebraic degeneracy of the Weyl tensor with respect to (almost) Robinson structures. We start with manifolds equipped with a conformal Killing-Yano (CKY) $2$-form. While the Iwasawa manifold -- Example \ref{exa-Iwasawa} -- is a proper Riemannian manifold, it features a number of geometric properties that find parallels in Lorentzian geometry: Cotton-York obstruction to the integrability of an almost Robinson structure, and families of Hermitian structures. Next, we state Proposition \ref{prop-KS-Robinson} on Kerr-Schild metrics admitting a Robinson structure, which is illustrated by the Myers-Perry black hole -- Example \ref{exa-Kerr-MP} -- and as a limiting case, we show that the Schwarzdchild spacetime -- Example \ref{exa-Schwarzschild} -- admits infinitely many repeated aligned almost Robinson structures, both integrable and non-integrable -- Proposition \ref{prop-alg-sp-NInt}. We carry on with a study of Robinson-Trautman spacetimes leading to Proposition \ref{prop-Robinson-Trautman-6}. Taub-NUT spacetimes are examined in Example \ref{exa-Taub-NUT-(A)dS}. We conclude our collection of examples by considering Lorentzian manifolds that admit a parallel Robinson structure in Proposition \ref{prop-int-parallel_Rob}, and those that admit a parallel pure spinor of real index $1$ in Proposition \ref{prop-parallel-pure-spinor}.

The paper ends with three appendices, which contain more lengthy and descriptive material referred to in the core sections of the article. Appendix \ref{sec-spinors} contains a spinor approach to Robinson structures. In appendix \ref{sec-spinor-descript}, we give generalisations of the Bel-Debever criterion for the $\Sim(n-2)$-invariant classifications of section \ref{sec-curvature}, and `basis decompositions' of irreducible curvature tensors according to both the $\Sim(n-2)$-invariant and $\Sim(m-1,\C)$-invariant classifications of section \ref{sec-curvature}. The short appendix \ref{sec-low-dim} highlights some of the special features occuring in low dimensions, where a number of simplifications can be made.

\section{Algebraic background}\label{sec-algebra}
\subsection{Null lines and their stabilisers}\label{sec-sim-alg}
\paragraph{Null lines}
Let $(\mfV,g)$ be oriented $n$-dimensional Minkowski space, i.e. a vector space $\mfV$ equipped with a non-degenerate symmetric bilinear form $g_{ab}$ of signature $(n-1,1)$ (i.e. $(+,\ldots ,+,-)$). We adopt the abstract index convention of \cite{Penrose1984}, elements of $\mfV$ and $\mfV^*$ will carry upstairs and downstairs lower case Roman indices respectively. Indices will be lowered and raised freely be means of $g_{ab}$ and its inverse $g^{ab}$ respectively. Symmetrisation and skew-symmetrisation will be denoted by round and square brackets around groups of indices respectively, e.g. $g \ind{_{ab}} = g \ind{_{(ab)}} \in \odot^2 \mfV^*$ and $\alpha \ind{_{abc}} =  \alpha \ind{_{[abc]}} \in \wedge^3 \mfV^*$.

Let $\mfK$ be a null line in $\mfV$, i.e. for any $V^a \in \mfK$, $V^a V_a = 0$. For convenience, we fix a generator $k^a$ in $\mfK$. Since $\mfK$ is null, it is contained in its orthogonal complement $\mfK^\perp = \left\{ V^a \in \mfV : g_{ab} k^a V^b = 0 \right\}$. Thus, a null line in $\mfV$ induces a filtration
\begin{align}\label{eq-sim-filtration}
 \{ 0 \} =: \mfV^2 \subset \mfV^1 \subset \mfV^0 \subset \mfV^{-1} := \mfV \, .
\end{align}
of vector subspaces on $\mfV$, where $\mfV^1 := \mfK$ and $\mfV^0 := \mfK^\perp$. By the metric isomorphism $\mfV^* \cong \mfV$, it is clear that $\mfK$ also defines a filtration $\{ (\mfV^*)^i\} $ on the dual vector space $\mfV^*$ where $(\mfV^*)^i \cong \mfV^i$ for each $i$.

One can fix a vector subspace $\mfV_{-1}$ dual to $\mfV_1$, so that
\begin{align}\label{eq-sim-grading}
 \mfV & = \mfV_1 \oplus \mfV_0 \oplus \mfV_{-1} \, ,
\end{align}
where $\mfV_1 \cong \mfV^1$, $\mfV_0 \oplus \mfV_1 \cong \mfV^0$.  In this case, $\mfV_0$ is the orthogonal complement of both $\mfV_1$ and $\mfV_{-1}$. Fix a generator $\ell^a$ of $\mfV_{-1}$ satisfying the normalisation $\ell_a k^a = 1$.
The metric tensor $g_{ab}$ induces a non-degenerate positive definite symmetric bilinear form
\begin{align*}
 h \ind{_{ab}} & = g \ind{_{ab}} - S \ind{_{ab}} \, , & \mbox{where} & & S \ind{_{ab}} & := 2 \, k \ind{_{\lp{a}}} \ell \ind{_{\rp{b}}} \, , 
\end{align*}
on $\mfV_0$.

\paragraph{The Lie subalgebra $\simalg(n-2)$}
The filtration \eqref{eq-sim-filtration} and grading \eqref{eq-sim-filtration} on $\mfV$ induce filtrations and gradings on any tensor product of $\mfV$. In particular, the Lie algebra $\g := \so(n-1,1)$, which will be identified with $\wedge^2 \mfV$, admits the filtration
\begin{align}\label{eq-filtration-sim}
 \{ 0 \} =: \g^2 \subset \g^1 \subset \g^0 \subset \g^{-1} := \g \, ,
\end{align}
of vector subspaces $\g^i$, and a grading
\begin{align}\label{eq-grading-sim}
 \g & = \g_{-1} \oplus \g_0 \oplus \g_1 \, ,
\end{align}
where $\g^i = \g_i \oplus \g^{i+1} $ for $i=-1,0,1$, and $\g_{\pm1} \cong \mfV_{\pm1} \otimes \mfV_0$ and $\g_0 \cong \wedge^2 \mfV_0 \oplus \left( \mfV_1 \otimes \mfV_{-1} \right)$. It is straightforward to see that the vector subspace $\g^0$ is the stabiliser of $\mfV^1$, i.e. $k \ind{^c} \phi \ind{_c^{[a}} k \ind{^{b]}} = 0$ for any $\phi \ind{_{ab}} \in \g^0$. Such a Lie subalgebra is known  as a \emph{parabolic Lie subalgebra} of $\g$ \cites{Fulton1991,Baston1989,vCap2009}. It is also known as the Lie algebra $\simalg(n-2)$ of the group $\Sim(n-2)$ of similarities of $\R^{n-2}$. As for any parabolic Lie subalgebra, it consists of a nilpotent part $\g_1 \cong \R^{n-2}$, and a reductive part $\g_0 = \co(n-2) := \so(n-2) \oplus \R$. For convenience, we shall denote the center of $\g_0$ and its complement in $\g_0$ by $\mfz_0$ and $\so_0$ respectively. The center contains a unique \emph{grading element}, which, given our choice of $k^a$ and $\ell^a$, can be expressed as $E_{ab} := - 2 k_{[a} \ell_{b]}$, and has eigenvalues $i$ on $\mfV_i$, i.e. $V \ind{^b} E \ind{_b^a} = \pm V \ind{^a}$ if $V \ind{^a} \in \mfV_{\pm1}$, and $V \ind{^b} E \ind{_b^a} = 0$ if $V \ind{^a} \in \mfV_0$, and so on for tensor products of $\mfV$.

At this stage, we note that $\g_1 + \mfz_0$ and $\g_1 + \so_0$ are $\simalg(n-2)$-invariant subspaces\footnote{In fact, their invariance can be made more explicit by noting that they can be defined in terms of kernels of the maps ${}^\g_\mfK \Pi_i^j$ defined in appendix \ref{ref-spinor-descript-proj}:
\begin{align*}
 \g^0 & = \{ \phi_{ab} \in \g : {}^\g_\mfK \Pi_{-1}^0 (\phi) =  0 \} \, , &
 \g_1 + \mfz_0 & := \{ \phi_{ab} \in \g : {}^\g_\mfK \Pi_0^1 (\phi) =  0 \}   \, , &
  \g_1 + \so_0 & := \{ \phi_{ab} \in \g : {}^\g_\mfK \Pi_0^0 (\phi) =  0 \} \, .
\end{align*}}  of $\g$.
Their quotients by $\g_1$ are linear isomorphic to the irreducible $\g_0$-modules $\mfz_0$ and $\so_0$. In particular, they are irreducible $\simalg(n-2)$-modules. Introduce the graded module $\gr(\g)$ associated to the filtration \eqref{eq-filtration-sim}, i.e.
\begin{align*}
\gr(\g) & = \gr_{-1} (\g) \oplus \gr_0 (\g) \oplus \gr_1 (\g) \, , & \mbox{where} & &\gr_i (\g) & := \g^i /\g^{i+1} \, ,
\end{align*}
Each summand $\gr_i (\g)$ is a totally reducible $\simalg(n-2)$-module, linearly isomorphic to a $\co(n-2)$-module $\g_i$ that depends on the splitting \eqref{eq-grading-sim} chosen. When $n \neq 6$, each irreducible $\simalg(n-2)$-submodule $\g_i^j$, say, of $\gr_i(\g)$ is linearly isomorphic to a $\co(n-2)$-submodule $\breve{\g}_i^j$, say, of $\g_i$ -- here, $\breve{\g}_{\pm1}^0 := \g_{\pm1}$, $\breve{\g}_0^0 := \mfz_0$ and $\breve{\g}_0^1 := \so (n-2)$.
When $n=6$, the $\simalg(4)$-module $\g_0^1$ splits further into a self-dual part $\g_0^{1,+}$ and an anti-self-dual part $\g_0^{1,+}$, i.e. $\g_0^1 \cong \g_0^{1,+} \oplus \g_0^{1,-}$. In both cases, the nilpotent part $\g_1$ of $\simalg(n-2)$ acts trivially on each $\g_i^j$. The splitting of $\gr(\g)$ into irreducibles together with the effect of the action of $\g_1$ on the corresponding $\co(n-2)$-modules $\breve{\g}_i^j$ can be conveniently represented in terms of the graph
\begin{align}
\begin{aligned}
 \xymatrix@R=1mm{
 & \g_0^1 \ar[dr] & \\
 \g_1^0 \ar[dr] \ar[ur] & & \g_{-1}^0 \\
 & \g_0^0 \ar[ur] & }
 \end{aligned}\label{eq-Pendiag-sim}
\end{align}
where we draw an arrow from $\g_i^j$ to $\g_{i-1}^k$ whenever $\breve{\g}_i^j \subset \g_1 \cdot \breve{\g}_{i-1}^k$ -- here, the $\cdot$ denotes the action of $\g$ on a $\g$-module, and in this case, corresponds with the Lie bracket.

\paragraph{Lie group level and the null grassmanian of null lines}
The stabiliser of an (unoriented) null line $\mfK \subset \mfV$ in the orthogonal group $\OO(\mfV,g) \cong \OO(n-1,1)$ is isomorphic to the group $\Sim(n-2) \cong \left( \R^* \times \OO(n-2) \right) \ltimes \R^{n-2}$ of similarities of $\R^{n-2}$, with Lie algebra $\simalg(n-2)$ as described above.\footnote{In the extant literature, $\Sim(n-2)$ may be taken to stabilise an \emph{oriented} null line.} We can identify the space $\Gr_1(\mfV,g)$ of all unoriented null lines in $\mfV$ with the homogeneous space
\begin{align*}
\Gr_1 (\mfV,g) & = \OO (n-1,1)/ \Sim(n-2) \, ,
\end{align*}
which is simply diffeomorphic to the $(n-2)$-dimensional sphere $S^{n-2}$.

We note in passing that $\Sim(n-2)$ contains the Lie subgroup $\Sim_+(n-2) \cong \left( \R^+ \times \OO(n-2) \right) \ltimes \R^{n-2}$ stabilising an \emph{oriented} null line. Thus, the space of all oriented null lines is given by the homogeneous space $\OO(n-1,1)/\Sim_+(n-2)$ and splits into two connected components, the space of all future-pointing null lines and the space of all past-pointing null lines, each being isomorphic to the $(n-2)$-dimensional sphere $S^{n-2}$.

\subsection{Robinson structure}\label{sec-rob-alg}
As before, $(\mfV, g)$ will denote oriented $n$-dimensional Minkowski space. Denote by $({}^\C \mfV , g)$ the complexification of $(\mfV,g)$, where it is now understood that $g$ becomes complex-valued. Let $\bar{}: {}^\C \mfV \rightarrow {}^\C \mfV: V^a \mapsto \bar{V}^a$ be the real structure, or complex conjugation, on ${}^\C \mfV$ preserving $\mfV$, i.e. $\bar{V}^a = V^a$ if and only if $V^a \in \mfV$.

Set $n=2m+\epsilon$ where $\epsilon \in \{ 0 ,1 \}$. Let $\mfN$ be a totally null complex $m$-plane, i.e. an $m$-dimensional vector subspace of ${}^\C \mfV$, on which the complexified bilinear form $g_{ab}$ is totally degenerate. 
Denote by $\bar{\mfN}$ the complex conjugate of $\mfN$, i.e. $\bar{\mfN} := \{ V^a \in {}^\C \mfV : \bar{V}^a \in \mfN \}$.

\begin{defn}
We call the dimension of the intersection of $\mfN$ and $\bar{\mfN}$ the \emph{real index} of $\mfN$.
\end{defn}

\begin{lem}[\cites{Kopczy'nski1992,Kopczy'nski1997}]
When $\epsilon=0$, the real index of $\mfN$ is $1$. When $\epsilon=1$, the real index of $\mfN$ is either $0$ or $1$.
\end{lem}

\begin{defn}
A \emph{Robinson structure} on $(\mfV,g)$ is a totally null complex $m$-plane of $(\mfV,g)$ of real index $1$.
\end{defn}

Since a Robinson structure $\mfN$ determines a null line $\mfK$ in $\mfV$, we can henceforth assume the setting of the previous section \ref{sec-sim-alg} with $\mfV^1 := \mfK$ and $\mfV^0 := \mfK^\perp$ together with the filtration \eqref{eq-sim-filtration} and grading \eqref{eq-sim-grading} on $\mfV$. Clearly both $\mfN$ and $\bar{\mfN}$ are contained in the complexification of $\mfV^0$. For clarity, we first deal with the even- and odd-dimensional cases separately.

\subsubsection{Even dimensions ($\epsilon=0$)}
In this case, it is a standard result \cite{Penrose1986} that the totally null complex $m$-plane $\mfN$ of the Robinson structure can be either self-dual or anti-self-dual according to a chosen orientation on $\mfV$, i.e. $* \omega = \pm ( - \ii )^{(m-1)(m-3)} \omega$ for $\omega \in \wedge^m \mfN$, where $* : \wedge^k \mfV \rightarrow \wedge^{n-k} \mfV$ is the Hodge duality operator. The orientation of the complex conjugate $\bar{\mfN}$ relative to that of $\mfN$ depends on the dimension of $\mfV$. The following lemma is standard (see eg. \cites{Cartan1981,Kopczy'nski1992}).

\begin{lem}
When $m$ is even, a Robinson structure is self-dual, respectively anti-self-dual, if and only if its complex conjugate is anti-self-dual, respectively self-dual.

When $m$ is odd, a Robinson structure is self-dual, respectively anti-self-dual, if and only if its complex conjugate is self-dual, respectively anti-self-dual.
\end{lem}

Referring to the grading \eqref{eq-sim-grading} adapted to the filtration \eqref{eq-sim-filtration}, we can define the $(m-1)$-dimensional vector subspaces
\begin{align}\label{eq+-eigenspaces}
 \mfV_0^{(1,0)} & := {}^\C \mfV_0 \cap \mfN \, , & \mfV_0^{(0,1)} & := {}^\C \mfV_0 \cap \bar{\mfN} \, . 
\end{align}
of ${}^\C \mfV_0 := \C \otimes_\R \mfV_0$, so that
\begin{align}\label{eq-splitting-V0-even}
{}^\C \mfV_0 & = \mfV_0^{(1,0)} \oplus \mfV_0^{(0,1)} \, .
\end{align}
Being subspaces of $\mfN$ and $\bar{\mfN}$ respectively, each of $\mfV_0^{(1,0)}$ and $\mfV_0^{(0,1)}$ are totally null with respect to the complexification of $g_{ab}$, and in fact, with respect to the complexification of $h_{ab}$ since they are also orthogonal to both  $\mfV_{\pm1}$. By a standard argument of linear algebra \cite{Nurowski2002}, they can therefore be identified with the $+\ii$ and $-\ii$-eigenspaces of a complex structure $J \ind{_a^b}$ on $\mfV_0$ compatible with $H_{ab} := h_{ab}$, i.e. $J \ind{_a^c} J \ind{_c^b} = - H \ind{_a^b}$. The reason for the relabelling of $h_{ab}$ as $H_{ab}$ is for notational consistency with the odd-dimensional case. We take the convention that $V^a J \ind{_a^b} = \ii V^a$ for any $V^a \in \mfV_0^{(1,0)}$. In fact, this complex structure has the associated hermitian $2$-form $\omega_{ab} := J \ind{_a^c} H \ind{_{cb}}$ on $\mfV_0$.

To make the description independent of the splitting \eqref{eq-sim-grading}, we extend the vector subspaces $\mfV_0^{(1,0)}$ and $\mfV_0^{(0,1)}$ of ${}^\C\mfV_0$ to vector subspaces
\begin{align}\label{eq-eigenspaces-gr0}
\gr_0^{(1,0)} (\mfV) & := \left(\mfV_0^{(1,0)} + {}^\C \mfV_1 \right)/ {}^\C \mfV_1 \, , & 
\gr_0^{(0,1)} (\mfV) & := \left(\mfV_0^{(0,1)} + {}^\C \mfV_1 \right)/ {}^\C \mfV_1 \, ,
\end{align}
respectively, of the complexification of $\gr_0( \mfV) = \mfV^0/\mfV^1$. These can be identified with the $+\ii$ and $-\ii$ eigenspaces of a complex structure on the screenspace $\gr_0( \mfV)$, that is an equivalence class of complex structures $[ J \ind{_a^b}]$ on $\gr_0(\mfV)$ where 
\begin{align*}
  J \ind{_a^b} \, , \hat{J} \ind{_a^b} \in  [J \ind{_a^b}] & & \Leftrightarrow & & \hat{J} \ind{_a^b} & = J \ind{_a^b} + \alpha \ind{_a} k \ind{^b} - k \ind{_a} \alpha \ind{^b} \, , & & \mbox{for some $\alpha_a \in \mfV^0$.}
\end{align*}
Similarly, one can also define a Hermitian $2$-form on $\gr_0(\mfV)$ as an equivalence class of Hermitian $2$-forms  $[\omega \ind{_{ab}}]$ on $\gr_0(\mfV)$.

\subsubsection{Odd dimensions ($\epsilon=1$)}
The odd-dimensional description of the Robinson structure $\mfN$ is almost identical to the even-dimensional case, except that now the totally null complex $m$-plane $\mfN$ of the Robinson structure is strictly contained in its $(m+1)$-dimensional orthogonal complement
\begin{align*}
	\mfN^\perp & := \{ V^a \in {}^\C \mfV : V^a W_a =0 \, , \mbox{for all $W^a \in \mfN$} \} \, ,
\end{align*}
and $\mfN^\perp/\mfN$ is one-dimensional. In the splitting \eqref{eq-sim-grading}, we can then distinguish a one-dimensional subspace
\begin{align*}
 \mfV_0^{(0,0)} & := {}^\C \mfV_0 \cap \mfN^\perp \cap \bar{\mfN}^\perp
\end{align*}
of ${}^\C \mfV_0$, which can be seen to be the complexification of a real vector subspace $\mfU$ of $\mfV_0$, on which $g_{ab}$ (and thus $h_{ab}$) is non-degenerate. We thus have a splitting of the complexification of $\mfV_0$
\begin{align}\label{eq-splitting-V0-odd}
 {}^\C \mfV_0 = \mfV_0^{(1,0)} \oplus \mfV_0^{(0,1)} \oplus \mfV_0^{(0,0)} \, , 
\end{align}
where $\mfV_0^{(1,0)}$ and $\mfV_0^{(0,1)}$ are defined by \eqref{eq+-eigenspaces} as in the even-dimensional case. 
Choosing a unit vector $u^a$ in $\mfU$, we can define a positive definite symmetric bilinear form
\begin{align*}
 H \ind{_{ab}} & :=
 h \ind{_{ab}} - u \ind{_a} u \ind{_b} \, ,
\end{align*}
on $\mfV_0 \cap \mfU^\perp$. Then each of $\mfV_0^{(1,0)}$ and $\mfV_0^{(0,1)}$ are totally null $(m-1)$-planes corresponding to the $+\ii$ and $-\ii$-eigenspaces of a $H_{ab}$-compatible complex structure $J \ind{_a^b}$ on $\mfV_0 \cap \mfU^\perp$, with corresponding hermitian $2$-form $\omega_{ab} := J \ind{_a^c} H \ind{_{cb}}$. The vector space $\mfU$ is the kernel of $J \ind{_a^b}$. Thus, the odd-dimensional case differs from the even-dimensional case only by an additional one-dimensional vector subspace.

\begin{rem}
In odd dimensions, the endomorphism $J \ind{_a^b}$ on $\mfV_0$ could be called an $h_{ab}$-compatible CR structure on $\mfV_0$, but it is usually referred to as a \emph{contact Riemannian structure}. In this article, we shall referred to $J \ind{_a^b}$ as a Hermitian structure regardless of the dimension with the understanding that it applies to $\mfV_0 \cap \mfU^\perp$.
\end{rem}

Again, this complex structure defines an equivalence class of complex structures on $\gr_0(\mfV)$ with $\pm\ii$-eigenspaces $\gr_0^{(1,0)} (\mfV)$ and $\gr_0^{(1,0)} (\mfV)$ as defined by \eqref{eq-eigenspaces-gr0}, and kernel
\begin{align*}
\gr_0^{(0,0)} (\mfV) & := \left(\mfV_0^{(0,0)} + {}^\C \mfV_1 \right)/ {}^\C \mfV_1 \, .
\end{align*}
The $\simalg(n-2)$-module $\gr_0 (\mfV)$ is equipped with an equivalence class of unit vectors  $[ u^a ]$ where
\begin{align*}
u^a \, , \hat{u} ^a \in  [ u^a ] & & \Leftrightarrow & & \hat{u}^a & = u^a  + \lambda k^a \, , & & \mbox{for some $\lambda \in \R$.}
\end{align*}

\begin{nota}
For future use, we define the \emph{real} vector spaces
\begin{align*}
 \breve{\mfV}_0^{0,0} & := \dbl \mfV_0^{(1,0)} \dbr \, , & \mfV_0^{0,0} & := \left( \breve{\mfV}_0^{0,0} + \mfV_1 \right) / \mfV_1 \, , \\
  \breve{\mfV}_0^{0,1} & := [ \mfV_0^{(0,0)} ] \, , & \mfV_0^{0,1} & := \left( \breve{\mfV}_0^{0,1} + \mfV_1 \right) / \mfV_1 \, ,
\end{align*}
where the `bracket' notation of \cite{Salamon1989} denotes the real span of the bracketed complex vector space, i.e.
\begin{align*}
\mfV_0^{(1,0)} & = \C \otimes_\R \dbl \mfV_0^{(1,0)} \dbr \, , & \mfV_0^{(0,0)} & = \C \otimes_\R [ \mfV_0^{(0,0)} ] \, .
\end{align*}
Here, we have been careful to distinguish $\breve{\mfV}_0^{0,0}$ and $\mfV_0^{0,0}$ as vector subspaces of $\mfV_0$ and $\mfV^0/\mfV^1$ respectively.
\end{nota}

\subsubsection{The stabiliser of a Robinson structure}
Assume as before $n=2m+\epsilon$ where $\epsilon \in \{0,1\}$. Henceforth, we shall combine the even- and odd-dimensional cases, the distinction between which will be made by the use of $\epsilon \in \{0,1\}$. Thus, the splittings \eqref{eq-splitting-V0-even} and \eqref{eq-splitting-V0-odd} will be merged by writing
\begin{align*}
 {}^\C \mfV_0 = \mfV_0^{(1,0)} \oplus \mfV_0^{(0,1)} \oplus \epsilon \, \mfV_0^{(0,0)} \, .
\end{align*}
The stabiliser of $\mfN$ in $\so(n-1,1)$ also stabilises its complex conjugate and thus their intersection. Hence, it must be contained in $\simalg(n-2)$. For clarity, we shall fix a splitting \eqref{eq-sim-grading} adapted to the null line $\mfK$. Then, using our previous notation, the complexifications ${}^\C \g_i$ of $\g_i$ admit direct sum decompositions
\begin{align*}
 {}^\C \g_{\pm1} & =  \g_{\pm1}^{(1,0)} \oplus \g_{\pm1}^{(0,1)} \oplus \epsilon \, \g_{\pm1}^{(0,0)} \, , &  {}^\C \g_0 & = {}^\C \mfz_0 \oplus \g_0^{(2,0)} \oplus \g_0^{(0,2)} \oplus \g_0^{(1,1)} \oplus \epsilon \left( \g_0^{(1,0)} \oplus \g_0^{(0,1)} \right) \, ,
\end{align*}
where
\begin{align*}
\begin{aligned}
 \g_{\pm1}^{(0,1)} & \cong \mfV_{\pm1} \otimes \mfV_0^{(0,1)}  \, , &
 \g_{\pm1}^{(1,0)} & \cong \mfV_{\pm1} \otimes \mfV_0^{(1,0)}  \, , \\
 \g_0^{(2,0)}  & \cong \wedge^2 \mfV_0^{(1,0)} \, , &
 \g_0^{(0,2)}  & \cong \wedge^2 \mfV_0^{(0,1)} \, , &
 \g_0^{(1,1)}  & \cong \mfV_0^{(1,0)} \otimes \mfV_0^{(0,1)} \, ,
\end{aligned}
 \end{align*}
and in addition, when $\epsilon=1$,
\begin{align*}
 \g_{\pm1}^{(0,0)} & \cong \mfV_{\pm1} \otimes \mfV_0^{(0,0)}  \, , &
 \g_0^{(1,0)}  & \cong \mfV_0^{(1,0)} \otimes \mfV_0^{(0,0)} \, .
\end{align*} 
Here, we have made use of the standard identifications $\so(n,\C) \cong \wedge^2 ( {}^\C \mfV )$ and $\so(n-1,1) \cong \wedge^2 \mfV$.

Further, the complex structure $J \ind{_a^b}$ on $\mfV_0$ is contained in $\g_0^{(1,1)}$ and spans a $1$-dimensional invariant subspace, which we shall denote $\g_0^\omega$. The complement of $\g_0^\omega$ in $\g_0^{(1,1)}$ will be denoted by $\g_0^{(1,1)_\circ}$, i.e.
\begin{align*}
 \g_0^{(1,1)} & \cong \g_0^\omega \oplus \g_0^{(1,1)_\circ} \, .
\end{align*}

\paragraph{Real representations}
Since we shall essentially be interested in real representations, we again draw from the notation of \cite{Salamon1989} and define
\begin{align*}
\C \otimes_\R  \dbl \g_i^{(j,k)} \dbr & := \g_i^{(j,k)} & k \neq j  \, , \\
\C \otimes_\R  [ \g_0^{(i,i)} ] & := \g_0^{(i,i)}  \, .
\end{align*}
Putting things together, we see that the Lie algebra $\so(n-1,1)$ admits the direct sum decomposition
\begin{align*}
 \g_{\pm1} & = \dbl \g_{\pm1}^{(1,0)} \dbr \oplus \epsilon \, [ \g_{\pm1}^{(0,0)} ] \, , &  
 \g_0 & = \mfz_0 \oplus \left( [ \g_0^\omega ] \oplus [ \g_0^{(1,1)_\circ} ] \oplus \dbl \g_0^{(2,0)} \dbr \oplus \epsilon \, \dbl \g_0^{(1,0)} \dbr \right) \, .
\end{align*}
By standard results of Hermitian geometry, we have $[ \g_0^{(1,1)} ] \cong \uu(m-1)$, $[\g_0^\omega] \cong \R$ and $[ \g_0^{(1,1)_\circ} ] \cong \su(m-1)$, where $\uu(m-1)$ and $\su(m-1)$ are the Lie algebras of the unitary group $\U(m-1)$ and the special unitary group $\SU(m-1)$ respectively. Clearly, $( \mfz_0 \oplus [ \g_0^{(1,1)} ] ) \oplus \g_1$ preserves $\mfV_1 + \mfV_0^{(1,0)}$ and $\mfV_1 \oplus \mfV_0^{(0,0)}$. Thus,

\begin{lem}
The stabilizer of a Robinson structure in $\so(n-1,1)$ is the Lie subalgebra
\begin{align}\label{eq-cu}
 \simalg(m-1,\C) & := \cu(m-1) ) \oplus \R^{n-2} \, , & \mbox{where} & & \cu(m-1) & := \R \oplus \uu (m-1) \, ,
\end{align}
of the Lie algebra $\simalg(n-2)$.
\end{lem}

\paragraph{Associated graded module}
The associated graded $\simalg(n-2)$-module $\gr(\g)$ of the filtration \eqref{eq-filtration-sim} is also a $\simalg(m-1,\C)$-module, and each of the irreducible $\simalg(n-2)$-modules $\g_i^j$ decomposes further into irreducible $\simalg(m-1,\C)$-modules, i.e. for each $i,j$,
\begin{align*}
\gr_i^j (\g) & = \bigoplus_{k} \g_i^{j,k} \, ,
\end{align*}
where each $\g_i^{j,k}$ is linearly isomorphic to  a $\cu(m-1)$-module $\breve{\g}_i^{j,k}$ given by
\begin{align*}
\breve{\g}_{\pm1}^{0,0} & := \dbl \g_{\pm1}^{(1,0)} \dbr \, , &
\breve{\g}_{\pm1}^{0,1} &  := [ \g_{\pm1}^{(0,0)} ] \, , \\
 \breve{\g}_{\pm1}^{0,0} & := \dbl \g_{\pm1}^{(1,0)} \dbr \, , &  
 \breve{\g}_0^{0,0} & := [ \g_0^\omega ] \, , &
 \breve{\g}_0^{0,1} & := [ \g_0^{(1,1)_\circ} ] \, , &
 \breve{\g}_0^{0,2} & := \dbl \g_0^{(2,0)} \dbr \, .
\end{align*}

Since the modules $\breve{\g}_i^{j,k}$ are not invariant under $\g_1$, we can encode the action of the nilpotent part $\g_1$ of $\g$ on $\breve{\g}_i^{j,k}$ by means of a $\simalg(m-1,\C)$-invariant graph, which refines the graph \eqref{eq-Pendiag-sim}
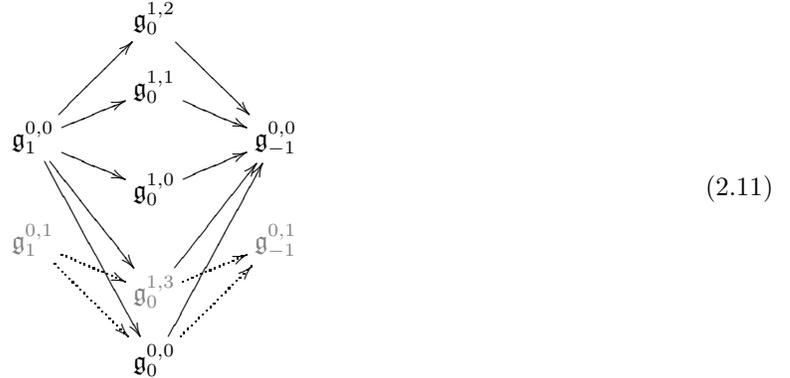
\begin{align}
\begin{aligned}\label{eq-Penrose-diag-g}
 \xymatrix@R=.1em{ & \g_0^{1,2} \ar[dddr] & \\
        & & \\
        & \g_0^{1,1} \ar[dr] & \\
        \g_1^{0,0} \ar[dddddr] \ar[dddr] \ar[dr] \ar[ur] \ar[uuur]& & \g_{-1}^{0,0} \\
        & \g_0^{1,0} \ar[ur] & \\
        \color{gray} \g_1^{0,1} \ar@{.>}[dddr] \ar@{.>}[dr]  & & \color{gray} \g_{-1}^{0,1} \\
        & \color{gray} \g_0^{1,3} \ar@{.>}[ur] \ar[uuur] & \\
        & & \\
        & \g_0^{0,0} \ar[uuuuur] \ar@{.>}[uuur] & 
}
\end{aligned}
\end{align}
where an arrow from $\g_i^{j,k}$ to $\g_{i-1}^{m,n}$ for some $i,j,k,m,n$ means that $\breve{\g}_i^{j,k} \subset \g_1 \cdot \breve{\g}_{i-1}^{m,n}$. Here, the modules occuring only in odd dimensions are in grey, and dotted arrows apply only to odd dimensions.\footnote{Alternatively, the irreducible $\simalg(m-1,\C)$-modules can be invariantly defined by means of the projection maps ${}^\g \Pi_i^{j,k}$ to be found in appendix \ref{sec-spinor-descript}, i.e. $\g_0^{j,k} := \{ \phi \in \g^0 : {}^\g \Pi_i^{m,n} (\phi) = 0 \, , m \neq j \, , n \neq k \} / \g^{i+1}$, and so on.}

\subsubsection{The space of Robinson structures}
Quite naturally, the stabiliser of $\mfN$ and $\mfN^\perp$ in the orthogonal group $\OO(n-1,1)$ is the group $\Sim(m-1,\C)$ of similarities of an $(m-1)$-dimensional Hermitian space. These are isomorphic to \cites{Kopczy'nski1992,Kopczy'nski1997}
\begin{align*}
\left( \R^* \times \U(m-1) \right) \ltimes \R^{n-2} \, , & & \mbox{when $n$ is even ($\epsilon=0$),} \\
\left( \R^* \times \U(m-1) \times \Z_2 \right) \ltimes \R^{n-2} \, , & & \mbox{when $n$ is odd ($\epsilon=1$),}
\end{align*}
with Lie algebras $\simalg(m-1,\C)$ as described above, and have two, respectively four, connected components when $n$ is even, respectively odd. Both $\OO(n-1,1)$ and $\SO(n-1,1)$ act transitively on the space of all totally null complex $m$-planes of ${}^\C \mfV$ of real index $1$. In particular,

\begin{prop}\label{prop-twistor}
The space $\Gr_m^1 ({}^\C \mfV,g)$ of all totally null complex $m$-planes of real index $1$ in ${}^\C \mfV$ is isomorphic to the $(m(m-1)+ \epsilon (2m-1))$-dimensional homogeneous space $\OO(n-1,1)/\Sim(m-1,\C)$. In even dimensions $(\epsilon=0)$, this space splits into two connected components, the space $\Gr_m^{+,1} ({}^\C \mfV,g)$ of all self-dual Robinson structures in ${}^\C \mfV$, and the space $\Gr_m^{-,1} ({}^\C \mfV,g)$ of all anti-self-dual Robinson structures in ${}^\C \mfV$, each isomorphic to $\SO(n-1,1)/\Sim(m-1,\C)$.
\end{prop}

\begin{rem}
When $n=2m$ and $m$ is even, points in $\Gr_m^{+,1} ({}^\C \mfV,g)$ and points in $\Gr_m^{-,1} ({}^\C \mfV,g)$ can be canonically identified by sending a self-dual Robinson structure to its complex conjugate, which is necessarily anti-self-dual.
\end{rem}
Clearly, the group $\Sim(m-1,\C)$ is a subgroup of $\Sim(n-2)$. We shall consider the quotient of the latter by the former in the next section.

\subsubsection{From null lines to Robinson structures}
In four dimensions, the correspondence between null lines and Robinson structures is one-to-one. In higher dimensions, this is no longer case. Instead, while a fixed Robinson structure $\mfN$ determines a real null line $\mfK$ as the real span of its intersection $\mfN \cap \bar{\mfN}$, a given null line $\mfK$ yields a family of totally null complex $m$-planes, those that intersect $\mfK$.

More specifically, there is a projection $\pi : \Gr_m^1 ({}^\C \mfV,g) \rightarrow \Gr_1 (\mfV,g)$, which sends a totally null complex $m$-plane $\mfN$ of real index $1$ to the real span of $\mfN \cap \bar{\mfN}$. The inverse image of a point $\mfK$ is then
$\pi^{-1} (\mfK) = \{ \mfN \in \Gr_m^1 (\mfV,g) : \dim (\mfN \cap \mfK ) > 0 \}$, which must be isomorphic to the quotient $\Sim(n-2)/\Sim(m-1,\C)$, or equivalently $\OO(n-2)/\U(m-1)$. In even dimensions, this space consists of the disjoint union of two connected homogeneous spaces, each isomorphic to $\SO(2m-2)/\U(m-1)$. The real dimensions of these spaces can readily be computed to be $(m-1)(m-2)$ and $m(m-1)$ in even and odd dimensions respectively.\footnote{This can also be seen at the Lie algebra level by considering the quotient $\simalg(n-2)$ by $\simalg(m-1,\C)$, or equivalently the complement of $\cu(m-1)$ (see \eqref{eq-cu}) in $\co(n-2)$, which is isomorphic to $\dbl \g_0^{(2,0)} \dbr \oplus \epsilon \dbl \g_0^{(1,0)} \dbr$.}

In summary,
\begin{prop}\label{prop-null2Rob}
Let $(\mfV,g)$ be $n$-dimensional Minkowski space with $n=2m+\epsilon$, $\epsilon \in \{0,1\}$. Then every null line in $\mfV$ corresponds to a family of Robinson structures in ${}^\C \mfV$ parametrised by points of the homogeneous space $\OO(n-2)/\U(m-1)$ of dimensions $(m-1)(m+2(\epsilon-1))$. Further, when $n$ is even, this family splits into two disjoint connected families of self-dual and anti-self-dual Robinson structures, each isomorphic to $\SO(2m-2)/\U(m-1)$.
\end{prop}

\section{Classifications of curvature tensors}\label{sec-curvature}
The algebraic setting of section \ref{sec-algebra} extends naturally to tensor products of $\mfV$, the standard representation of $\so(n-1,1)$. We shall now be concerned with the invariant classification of curvature tensors, more particularly, we classify the irreducible $\so(n-1,1)$-modules
\begin{align}
 \mfF & := \{ \Phi_{abc} \in \otimes^2 \mfV : \Phi \ind{_{ab}} = \Phi \ind{_{(ab)}} \, , \Phi \ind{^a_a} = 0 \} \, , \label{eq-rep-F} \\
 \mfA & := \{ A_{abc} \in \otimes^3 \mfV : A \ind{_{abc}} = A \ind{_{a[bc]}} \, , A \ind{_{[abc]}} = 0 \, , A \ind{^a_{ab}} = 0 \} \, , \label{eq-rep-A} \\
 \mfC & := \{ C_{abcd} \in \otimes^4 \mfV : C \ind{_{abcd}} = C \ind{_{[ab][cd]}} \, , C \ind{_{[abc]d}} = 0 \, , C \ind{_{abc}^b} = 0 \} \, , \label{eq-rep-C}
\end{align}
of tracefree Ricci tensors, Cotton-York tensors, and of Weyl tensors respectively. We recall the dimensions of these modules in the following table
\vspace{-6.5mm}
\begin{center}
\begin{displaymath}
{\renewcommand{\arraystretch}{1.5}
\begin{array}{||c|c|g||}
\hline
 \mfV & 2m & 2m+1 \\
\hline
 \g \cong \wedge^2 \mfV & (2m-1)m & (2m+1)m \\
\hline
 \mfF & (2m-1)(m+1) & m(2m+3) \\
\hline
 \mfA & \frac{8}{3}m(m+1)(m-1) & \frac{1}{3}(2m-1)(2m+1)(2m+3) \\
\hline
 \mfC & \frac{1}{3}m(m+1)(2m+1)(2m-3) & \frac{1}{3}(m-1)(m+1)(2m+1)(2m+3) \\
 \hline
\end{array}}
\end{displaymath}
\end{center}
where, for convenience, we distinguish the even- and odd-dimensional cases by shading the latter in gray.

The $\simalg(n-2)$-invariant classifications are given first, and are then refined in terms of $\simalg(m-1,\C)$-modules. Needless to say that these classifications are also invariant under the Lie groups $\Sim(n-2)$ and $\Sim(m-1,\C)$. In the tables, we have abbreviated $\so(n-1,1)$-modules, $\simalg(n-2)$-modules, $\co(n-2)$-modules, $\simalg(m-1,\C)$-modules and $\cu(m-1)$-modules by $\so(n-1,1)$-mod, $\simalg$-mod, $\co$-mod, $\simalg_\C$-mod and $\cu$-mod respectively. Again, we treat both the even- and odd-dimensional cases at once: the rows of the additional modules occurring in odd dimensions only are shaded in gray. The symbol $\circledcirc$ denotes the Cartan product, which corresponds to the tracefree symmetric product for $\co(n-2)$-modules, and to the symmetric product or tracefree tensor product for $\cu(m-1)$-modules.

\subsection{$\simalg(n-2)$-invariant classifications}\label{sec-sim-class}
As in section \ref{sec-sim-alg}, we single out a null line $\mfK$ on $n$-dimensional Minkowski space $(\mfV,g)$, so that $\mfV$ admits a filtration \eqref{eq-sim-filtration} of $\simalg(n-2)$-modules. Then this filtration induces $\simalg(n-2)$-invariant filtrations on any tensor product of $\mfV$. This applies in particular to $\mfF$, $\mfA$ and $\mfC$. Each of the graded $\simalg(n-2)$-modules associated to these filtrations splits into a direct sum of irreducible $\simalg(n-2)$-modules. The effect of the action of the nilpotent part of $\simalg(n-2)$ on any associated $\co(n-2)$-invariant decompositions is encoded in the form of a diagram. We also discuss the relation with the classification of \cite{Coley2004} based on the null alignment formalism generalisating the Petrov classification of the Weyl tensor in section \ref{rem-PW-types}.

Before we start, it is worth mentioning that the irreducible $\co(n-2)$-modules occuring in the classifications of  $\mfF$, $\mfA$ and $\mfC$ can all be expressed as the Cartan product of the $\co(n-2)$-modules of $\mfV$ and $\g$, which we presently recall:
\vspace{-6.5mm}
\begin{center}
\begin{displaymath}
{\renewcommand{\arraystretch}{1.5}
\begin{array}{||c|c|c|g||}
\hline
\text{$\simalg$-mod} & \text{$\co$-mod} & \text{Dimension ($n=2m$)} & \text{Dimension ($n=2m+1$)}\\
\hline
\mfV_{\pm1}^0 & \mfV_{\pm1} & 1 & 1\\
\hline
\mfV_0^0 & \mfV_0 & 2m-2 & 2m-1 \\
\hline
\end{array}}
\end{displaymath}
\end{center}
\begin{center}
\begin{displaymath}
{\renewcommand{\arraystretch}{1.5}
\begin{array}{||c|c|c|g||}
\hline
\text{$\simalg$-mod} & \text{$\co$-mod} & \text{Dimension ($n=2m$)} & \text{Dimension ($n=2m+1$)}\\
\hline
\g_{\pm1}^0 & \mfV_{\pm1} & 2m-2 & 2m-1 \\
\hline
 \g_0^0 & \mfz_0 & 1 & 1 \\
 \g_0^1 & \so_0 & (m-1)(2m-3) & (m-1)(2m-1) \\
\hline
\end{array}}
\end{displaymath}
\end{center}
\begin{center}
\begin{minipage}[b]{0.45\linewidth}\centering
\begin{displaymath}
{\renewcommand{\arraystretch}{1.5}
\begin{array}{||c|c|c||}
\hline
\text{$\simalg$-mod} & \text{$\co$-mod} & \text{Dimension ($n=6$)} \\
\hline
\g_{\pm1}^0 & \mfV_{\pm1} & 4 \\
\hline
\end{array}}
\end{displaymath}
\end{minipage}
\begin{minipage}[b]{0.45\linewidth}\centering
\begin{displaymath}
{\renewcommand{\arraystretch}{1.5}
\begin{array}{||c|c|c||}
\hline
\text{$\simalg$-mod} & \text{$\co$-mod} & \text{Dimension ($n=6$)} \\
\hline
 \g_0^0 & \mfz_0 & 1 \\
 \g_0^{\pm1} & \so^\pm_0 & 3 \\
\hline
\end{array}}
\end{displaymath}
\end{minipage}
\end{center}

\subsubsection{The tracefree Ricci tensor}
\begin{prop}\label{prop-Ricci-sim}
The filtration \eqref{eq-sim-filtration} on $\mfV$ induces a filtration
\begin{align}\label{eq-filtration-F}
 \{ 0 \} =: \mfF^3 \subset \mfF^2 \subset \mfF^1 \subset \mfF^0 \subset \mfF^{-1} \subset \mfF^{-2} := \mfF \, ,
\end{align}
of $\simalg(n-2)$-modules on the space $\mfF$ defined by \eqref{eq-rep-F}.

Further, the associated graded $\simalg(n-2)$-module $\gr(\mfF)=\bigoplus_{i=-2}^2 \gr_i(\mfF)$ where $\gr_i(\mfF):= \mfF^i/\mfF^{i+1}$ splits into a direct sum
\begin{align*}
 \gr_{\pm2} (\mfF) & = \mfF_{\pm2}^0 \, , & \gr_{\pm1} (\mfF) & = \mfF_{\pm1}^0 \, , & \gr_0 (\mfF) & = \mfF_0^0 \oplus \mfF_0^1 \, ,
\end{align*}
of irreducible $\simalg(n-2)$-modules, where
\vspace{-6.5mm}
\begin{center}
\begin{displaymath}
{\renewcommand{\arraystretch}{1.5}
\begin{array}{||c|c|c|g||}
\hline
\text{$\simalg$-mod} & \text{$\co$-mod} & \text{Dimension ($n=2m$)} & \text{Dimension ($n=2m+1$)} \\
\hline
 \mfF_{\pm2}^0 & \mfV_{\pm1} \circledcirc \mfV_{\pm1} & 1 & 1 \\
\hline
 \mfF_{\pm1}^0 & \mfV_{\pm1} \circledcirc \mfV_0 & 2m-2 & 2m-1 \\
\hline
 \mfF_0^0 & \mfV_{-1} \circledcirc \mfV_1 & 1 & 1 \\
 \mfF_0^1 & \mfV_0 \circledcirc \mfV_0 & m(2m-3)  & (m-1)(2m+1) \\
 \hline
\end{array}}
\end{displaymath}
\end{center}

Finally, the $\simalg(n-2)$-module $\gr(\mfF)$ can be expressed by means of a $\simalg(n-2)$-invariant graph
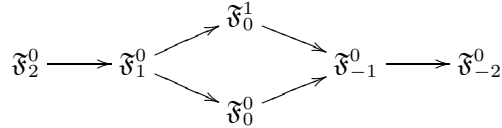
\begin{equation*}
 \xymatrix@R=.1em{
	& & \mfF_0^1 \ar[dr] & &  \\
\mfF_2^0 \ar[r] & \mfF_1^0 \ar[dr] \ar[ur] & & \mfF_{-1}^0 \ar[r] & \mfF_{-2}^0 \\
	  & & \mfF_0^0  \ar[ur] & & }
\end{equation*}
where an arrow from $\mfF_i^j$ to $\mfF_{i-1}^k$ for some $i,j,k$ implies that $\breve{\mfF}_i^j \subset \g_1 \cdot \breve{\mfF}_{i-1}^k$ for any choice of irreducible $\co(n-2)$-modules  $\breve{\mfF}_i^j$ and $\breve{\mfF}_{i-1}^k$ linearly isomorphic to $\mfF_i^j$ and $\mfF_{i-1}^k$ respectively.
\end{prop}

\subsubsection{The Cotton-York tensor}

\begin{prop}\label{prop-CY-sim}
The filtration \eqref{eq-sim-filtration} on $\mfV$ induces a filtration
\begin{align}\label{eq-filtration-A}
 \{ 0 \} =: \mfA^3 \subset \mfA^2 \subset \mfA^1 \subset \mfA^0 \subset \mfA^{-1} \subset \mfA^{-2} := \mfA \, ,
\end{align}
of $\simalg(n-2)$-modules on the space $\mfA$ defined by \eqref{eq-rep-A}

Further, when $n\neq6$, the associated graded $\prb$-module $\gr(\mfA)=\bigoplus_{i=-2}^2 \gr_i(\mfA)$ where $\gr_i(\mfA):= \mfA^i/\mfA^{i+1}$ splits into a direct sum
\begin{align*}
 \gr_{\pm2} (\mfA) & = \mfA_{\pm2}^0 \, , & \gr_{\pm1} (\mfA) & = \mfA_{\pm1}^0 \oplus \mfA_{\pm1}^1 \oplus \mfA_{\pm1}^2 \, , & \gr_{\pm1} (\mfA) & = \mfA_0^0 \oplus \mfA_0^1 \oplus \mfA_0^2  \, .
\end{align*}
of irreducible $\simalg(n-2)$-modules, where
\vspace{-6.5mm}
\begin{center}
\begin{displaymath}
{\renewcommand{\arraystretch}{1.5}
\begin{array}{||c|c|c|g||}
\hline
\text{$\simalg$-mod} & \text{$\co$-mod} & \text{Dimension ($n=2m > 6$)} & \text{Dimension ($n=2m+1$)} \\
\hline
 \mfA_{\pm2}^0 & \mfV_{\pm1} \circledcirc \g_{\pm1} & 2m-2 & 2m-1 \\
\hline
 \mfA_{\pm1}^0 & \mfV_{\pm1} \circledcirc \mfz_0 & 1 & 1  \\
 \mfA_{\pm1}^1 & \mfV_{\pm1} \circledcirc \so_0 & (m-1)(2m-3) & (m-1)(2m-1)  \\
 \mfA_{\pm1}^2 & \mfV_0 \circledcirc \g_{\pm1} & m(2m-3) & (m-1)(2m+1)  \\
\hline
 \mfA_0^0 & \mfV_0 \circledcirc \mfz_0 & 2m-2  & 2m-1  \\
 \mfA_0^1 & \mfV_{\mp1} \circledcirc \g_{\pm1} & 2m-2 & 2m-1 \\
 \mfA_0^2 & \mfV_0 \circledcirc \so_0 & \frac{8}{3}m(m-1)(m-2)  & \frac{1}{3}(2m-3)(2m-1)(2m+1)  \\
 \hline
\end{array}}
\end{displaymath}
\end{center}
with the proviso that when $n\leq4$, $\mfA_{\pm1}^2$ do not occur.

Finally, when $n \neq 6$, the $\simalg(n-2)$-module $\gr(\mfA)$ can be expressed by means of a $\simalg(n-2)$-invariant graph
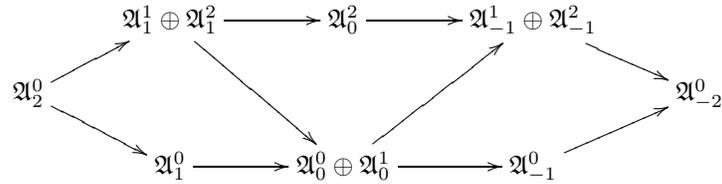
\begin{equation*}
 \xymatrix@R=1em{
& \mfA_1^1 \oplus \mfA_1^2 \ar[r] \ar[ddr] & \mfA_0^2 \ar[r] & \mfA_{-1}^1 \oplus \mfA_{-1}^2 \ar[dr]  & \\
\mfA_2^0 \ar[ur] \ar[dr] & & & & \mfA_{-2}^0 \\
& \mfA_1^0 \ar[r] & \mfA_0^0 \oplus \mfA_0^1 \ar[r] \ar[uur] & \mfA_{-1}^0 \ar[ur] & }
\end{equation*}
where an arrow from $\mfA_i^j$ to $\mfA_{i-1}^k$ for some $i,j,k$ implies that $\breve{\mfA}_i^j \subset \g_1 \cdot \breve{\mfA}_{i-1}^k$ for any choice of irreducible $\co(n-2)$-modules  $\breve{\mfA}_i^j$ and $\breve{\mfA}_{i-1}^k$ isomorphic to $\mfA_i^j$ and $\mfA_{i-1}^k$ respectively.
\end{prop}

\begin{prop}\label{prop-CY-sim6}
Assume $n=6$. The filtration \eqref{eq-sim-filtration} on $\mfV$ induces a filtration \eqref{eq-filtration-A} of $\simalg(4)$-modules on the space $\mfA$ defined by \eqref{eq-rep-A}.

Further, the associated graded $\prb$-module $\gr(\mfA)=\bigoplus_{i=-2}^2 \gr_i(\mfA)$ where $\gr_i(\mfA):= \mfA^i/\mfA^{i+1}$ splits into a direct sum
\begin{align*}
 \gr_{\pm2} (\mfA) & = \mfA_{\pm2}^0 \, , & \gr_{\pm1} (\mfA) & = \mfA_{\pm1}^0 \oplus \mfA_{\pm1}^{1,+} \oplus \mfA_{\pm1}^{1,-} \oplus \mfA_{\pm1}^2 \, , & \gr_{\pm1} (\mfA) & = \mfA_0^0 \oplus \mfA_0^1 \oplus \mfA_0^{2,+} \oplus \mfA_0^{2,-}   \, .
\end{align*}
of irreducible $\simalg(4)$-modules, where
\vspace{-6.5mm}
\begin{center}
\begin{minipage}[b]{0.45\linewidth}\centering
\begin{displaymath}
{\renewcommand{\arraystretch}{1.5}
\begin{array}{||c|c|c||}
\hline
\text{$\simalg$-mod} & \text{$\co$-mod} & \text{Dimension} \\
\hline
 \mfA_{\pm2}^0 & \mfV_{\pm1} \circledcirc \g_{\pm1} & 4 \\
\hline
 \mfA_{\pm1}^0 & \mfV_{\pm1} \circledcirc \mfz_0 & 1 \\
 \mfA_{\pm1}^{1,+} & \mfV_{\pm1} \circledcirc \so_0^+ & 3 \\
  \mfA_{\pm1}^{1,-} & \mfV_{\pm1} \circledcirc \so_0^- & 3 \\
 \mfA_{\pm1}^2 & \mfV_0 \circledcirc \g_{\pm1} & 9 \\
 \hline
\end{array}}
\end{displaymath}
\end{minipage}
\begin{minipage}[b]{0.45\linewidth}\centering
\begin{displaymath}
{\renewcommand{\arraystretch}{1.5}
\begin{array}{||c|c|c||}
\hline
\text{$\simalg$-mod} & \text{$\co$-mod} & \text{Dimension} \\
\hline
 \mfA_0^0 & \mfV_0 \circledcirc \mfz_0 & 4 \\
 \mfA_0^1 & \mfV_{\mp1} \circledcirc \g_{\pm1} & 4 \\
 \mfA_0^{2,+} & \mfV_0 \circledcirc \so_0^- & 8 \\
 \mfA_0^{2,-} & \mfV_0 \circledcirc \so_0^+ & 8 \\
 \hline
\end{array}}
\end{displaymath}
\end{minipage}
\end{center}

Finally, the $\simalg(4)$-module $\gr(\mfA)$ can be expressed by means of a $\simalg(4)$-invariant graph
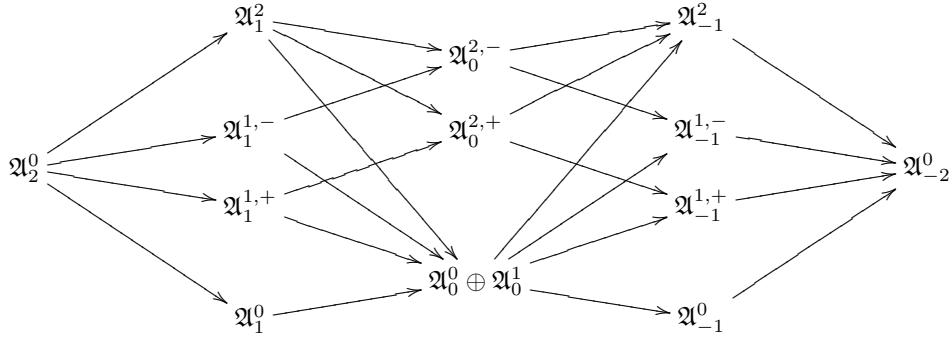
\begin{align*}
\xy
(-60,0)*+{\mfA_2^0}="s33",
(-30,20)*+{\mfA_1^2}="s30",
(-30,5)*+{\mfA_1^{1,-}}="s27",
(-30,-5)*+{\mfA_1^{1,+}}="s26",
(-30,-20)*+{\mfA_1^0}="s23",
(0,15)*+{\mfA_0^{2,-}}="s22",
(0,5)*+{\mfA_0^{2,+}}="s15",
(0,-15)*+{\mfA_0^0 \oplus \mfA_0^1}="s12",
(30,20)*+{\mfA_{-1}^2}="s9",
(30,5)*+{\mfA_{-1}^{1,-}}="s6",
(30,-5)*+{\mfA_{-1}^{1,+}}="s5",
(30,-20)*+{\mfA_{-1}^0}="s2",
(60,0)*+{\mfA_{-2}^0}="s1",
"s1"; "s2" ; **@{-} ?>*\dir{>}; "s1"; "s5" ; **@{-} ?>*\dir{>}; "s1"; "s6" ; **@{-} ?>*\dir{>}; "s1"; "s9" ; **@{-} ?>*\dir{>};
"s2"; "s12" ; **@{-} ?>*\dir{>};
"s5"; "s12" ; **@{-} ?>*\dir{>}; "s5"; "s15" ; **@{-} ?>*\dir{>};
"s6"; "s22" ; **@{-} ?>*\dir{>}; "s6"; "s12" ; **@{-} ?>*\dir{>}; 
"s9"; "s12" ; **@{-} ?>*\dir{>}; "s9"; "s15" ; **@{-} ?>*\dir{>}; "s9"; "s22" ; **@{-} ?>*\dir{>};
"s12"; "s23" ; **@{-} ?>*\dir{>}; "s12"; "s26" ; **@{-} ?>*\dir{>}; "s12"; "s27" ; **@{-} ?>*\dir{>}; "s12"; "s30" ; **@{-} ?>*\dir{>}; 
"s15"; "s30" ; **@{-} ?>*\dir{>}; 
"s22"; "s27" ; **@{-} ?>*\dir{>}; "s22"; "s30" ; **@{-} ?>*\dir{>};
"s23"; "s33" ; **@{-} ?>*\dir{>};
"s26"; "s33" ; **@{-} ?>*\dir{>};
"s15"; "s26" ; **@{-} ?>*\dir{>};
"s27"; "s33" ; **@{-} ?>*\dir{>};
"s30"; "s33" ; **@{-} ?>*\dir{>}
\endxy  
\end{align*}
where an arrow from $\mfA_i^j$ to $\mfA_{i-1}^k$ for some $i,j,k$ implies that $\breve{\mfA}_i^j \subset \g_1 \cdot \breve{\mfA}_{i-1}^k$ for any choice of irreducible $\co(n-2)$-modules  $\breve{\mfA}_i^j$ and $\breve{\mfA}_{i-1}^k$ isomorphic to $\mfA_i^j$ and $\mfA_{i-1}^k$ respectively.
\end{prop}

\subsubsection{The Weyl tensor}

\begin{prop}\label{prop-Weyl-sim}
The filtration \eqref{eq-sim-filtration} on $\mfV$ induces a filtration
\begin{align}\label{eq-filtration-C}
 \{ 0 \} =: \mfC^3 \subset \mfC^2 \subset \mfC^1 \subset \mfC^0 \subset \mfC^{-1} \subset \mfC^{-2} := \mfC \, .
 \end{align}
of $\simalg(n-2)$-modules on the space $\mfC$ defined by \eqref{eq-rep-C}.

Further, when $n \neq 6$, the associated graded $\simalg(n-2)$-module $\gr(\mfC)=\bigoplus_{i=-2}^2 \gr_i(\mfC)$ where $\gr_i(\mfC):= \mfC^i/\mfC^{i+1}$ splits into a direct sum
\begin{align*}
 \gr_{\pm2} (\mfC) & = \mfC_{\pm2}^0 \, , & \gr_{\pm1} (\mfC) & = \mfC_{\pm1}^0 \oplus \mfC_{\pm1}^1 \, , & \gr_{\pm1} (\mfC) & = \mfC_0^0 \oplus \mfC_0^1 \oplus \mfC_0^2 \oplus \mfC_0^3 \, .
\end{align*}
of irreducible $\simalg(n-2)$-modules, where
\vspace{-6.5mm}
\begin{center}
\begin{displaymath}
{\renewcommand{\arraystretch}{1.5}
\begin{array}{||c|c|c|g||}
\hline
\text{$\simalg$-mod} & \text{$\co$-mod} & \text{Dimension ($n=2m \neq 6$)} & \text{Dimension ($n=2m+1$)} \\
\hline
 \mfC_{\pm2}^0 & \g_{\pm1} \circledcirc \g_{\pm1} & m(2m-3) & (m-1)(2m+1) \\
\hline
 \mfC_{\pm1}^0 & \g_{\pm1} \circledcirc \mfz_0 & 2m-2 & 2m-1  \\
 \mfC_{\pm1}^1 & \g_{\pm1} \circledcirc \so_0 & \frac{8}{3}m(m-1)(m-2) & \frac{1}{3}(2m-1)(2m+1)(2m-3)  \\
\hline
 \mfC_0^0 & \mfz_0 \circledcirc \mfz_0 & 1 & 1 \\
 \mfC_0^1 & \mfz_0 \circledcirc \so_0 & (m-1)(2m-3)  & (m-1)(2m-1)  \\
 \mfC_0^2 & \g_1 \circledcirc \g_{-1} & m (2m-3) &  (m-1)(2m+1) \\
 \mfC_0^3 & \so_0 \circledcirc \so_0 & \frac{1}{3}m(m-1)(2m-1)(2m-5) & \frac{1}{3}m(m-2)(2m-1)(2m+1)  \\
 \hline
\end{array}}
\end{displaymath}
\end{center}
with the proviso that when $n = 4$, $\mfC_0^2$ does not occur.

Finally, when $n \neq 6$, the $\simalg(n-2)$-module $\gr(\mfA)$ can be expressed by means of a $\simalg(n-2)$-invariant graph
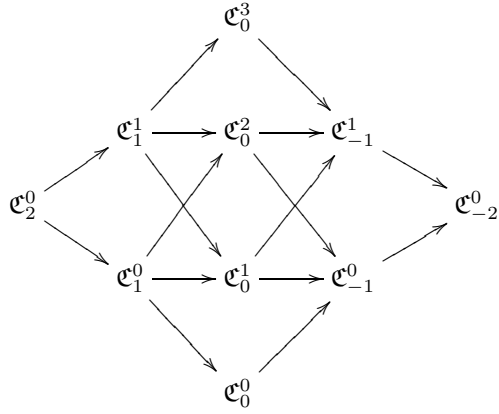
\begin{equation*}
 \xymatrix@R=1em{
	&	& \mfC_0^3 \ar[ddr] & & \\
       & 	& & &  \\
	& \mfC_1^1 \ar[uur] \ar[r] \ar[ddr]  & \mfC_0^2 \ar[r] \ar[ddr] & \mfC_{-1}^1 \ar[dr] &  \\
\mfC_2^0 \ar[ur] \ar[dr] & & & & \mfC_{-2}^0 \\
	  & \mfC_1^0 \ar[ddr] \ar[r] \ar[uur] & \mfC_0^1  \ar[uur] \ar[r] & \mfC_{-1}^0 \ar[ur] & \\
       & & & &  \\
&	& \mfC_0^0  \ar[uur] & & }
\end{equation*}
where an arrow from $\mfC_i^j$ to $\mfC_{i-1}^k$ for some $i,j,k$ implies that $\breve{\mfC}_i^j \subset \g_1 \cdot \breve{\mfC}_{i-1}^k$ for any choice of irreducible $\co(n-2)$-modules  $\breve{\mfC}_i^j$ and $\breve{\mfC}_{i-1}^k$ isomorphic to $\mfC_i^j$ and $\mfC_{i-1}^k$ respectively.
\end{prop}

\begin{prop}\label{prop-Weyl-sim6}
Assume $n=6$. The filtration \eqref{eq-sim-filtration} on $\mfV$ induces a filtration \eqref{eq-filtration-C} of $\simalg(4)$-modules on the space $\mfC$ defined by \eqref{eq-rep-C}.

Further, the associated graded $\simalg(4)$-module $\gr(\mfC)=\bigoplus_{i=-2}^2 \gr_i(\mfC)$ where $\gr_i(\mfC):= \mfC^i/\mfC^{i+1}$ splits into a direct sum
\begin{align*}
 \gr_{\pm2} (\mfC) & = \mfC_{\pm2}^0 \, , & \gr_{\pm1} (\mfC) & = \mfC_{\pm1}^0 \oplus \mfC_{\pm1}^{1,+} \oplus \mfC_{\pm1}^{1,-} \, , & \gr_{\pm1} (\mfC) & = \mfC_0^0 \oplus \mfC_0^{1,+} \oplus \mfC_0^{1,-} \oplus \mfC_0^2 \oplus \mfC_0^{3,+} \oplus \mfC_0^{3,-}  \, .
\end{align*}
of irreducible $\simalg(4)$-modules, where
\vspace{-6.5mm}
\begin{center}
\begin{minipage}[b]{0.45\linewidth}\centering
\begin{displaymath}
{\renewcommand{\arraystretch}{1.5}
\begin{array}{||c|c|c||}
\hline
\text{$\simalg$-mod} & \text{$\g_0$-mod} & \text{Dimension} \\
\hline
 \mfC_{\pm2}^0 & \g_{\pm1} \circledcirc \g_{\pm1} & 9 \\
\hline
 \mfC_{\pm1}^0 & \g_{\pm1} \circledcirc \mfz_0 & 4 \\
 \mfC_{\pm1}^{1,+} & \g_{\pm1} \circledcirc \so_0^+ & 8\\
 \mfC_{\pm1}^{1,-} & \g_{\pm1} \circledcirc \so_0^- & 8\\
 \hline
\end{array}}
\end{displaymath}
\end{minipage}
\begin{minipage}[b]{0.45\linewidth}\centering
\begin{displaymath}
{\renewcommand{\arraystretch}{1.5}
\begin{array}{||c|c|c||}
\hline
\text{$\simalg$-mod} & \text{$\g_0$-mod} & \text{Dimension} \\
\hline
 \mfC_0^0 & \mfz_0 \circledcirc \mfz_0 & 1 \\
 \mfC_0^{1,\pm} & \mfz_0 \circledcirc \so_0^\pm & 3 \\
 \mfC_0^2 & \g_1 \circledcirc \g_{-1} & 9 \\
 \mfC_0^{3,\pm} & \so_0^\pm \circledcirc \so_0^\pm & 5 \\
 \hline
\end{array}}
\end{displaymath}
\end{minipage}
\end{center}

Finally, the $\simalg(4)$-module $\gr(\mfA)$ can be expressed by means of a $\simalg(4)$-invariant graph
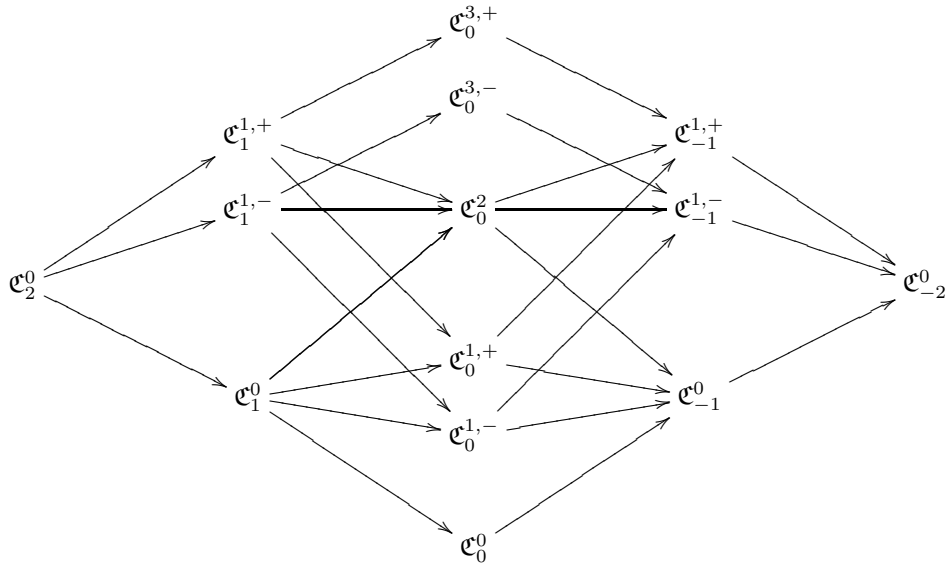
\begin{align*}
\xy
(-60,0)*+{\mfC_2^0}="s13",
(-30,20)*+{\mfC_1^{1,+}}="s12",
(-30,10)*+{\mfC_1^{1,-}}="s11",
(-30,-15)*+{\mfC_1^0}="s10",
(0,35)*+{\mfC_0^{3,+}}="s9",
(0,25)*+{\mfC_0^{3,-}}="s8",
(0,10)*+{\mfC_0^2}="s7",
(0,-10)*+{\mfC_0^{1,+}}="s6",
(0,-20)*+{\mfC_0^{1,-}}="s5",
(0,-35)*+{\mfC_0^0}="s4",
(30,20)*+{\mfC_{-1}^{1,+}}="s3",
(30,10)*+{\mfC_{-1}^{1,-}}="s2",
(30,-15)*+{\mfC_{-1}^0}="s1",
(60,0)*+{\mfC_{-2}^0}="s0",
"s0"; "s1" ; **@{-} ?>*\dir{>};
"s0"; "s2" ; **@{-} ?>*\dir{>};
"s0"; "s3" ; **@{-} ?>*\dir{>};
"s1"; "s6" ; **@{-} ?>*\dir{>};
"s1"; "s5" ; **@{-} ?>*\dir{>};
"s1"; "s4" ; **@{-} ?>*\dir{>};
"s1"; "s7" ; **@{-} ?>*\dir{>};
"s2"; "s8" ; **@{-} ?>*\dir{>};
"s2"; "s5" ; **@{-} ?>*\dir{>};
"s2"; "s7" ; **@{-} ?>*\dir{>};
"s3"; "s6" ; **@{-} ?>*\dir{>};
"s3"; "s9" ; **@{-} ?>*\dir{>};
"s3"; "s7" ; **@{-} ?>*\dir{>};
"s4"; "s10" ; **@{-} ?>*\dir{>};
"s5"; "s11" ; **@{-} ?>*\dir{>};
"s6"; "s12" ; **@{-} ?>*\dir{>};
"s7"; "s10" ; **@{-} ?>*\dir{>};
"s7"; "s11" ; **@{-} ?>*\dir{>};
"s7"; "s12" ; **@{-} ?>*\dir{>};
"s8"; "s11" ; **@{-} ?>*\dir{>};
"s9"; "s12" ; **@{-} ?>*\dir{>};
"s5"; "s10" ; **@{-} ?>*\dir{>};
"s6"; "s10" ; **@{-} ?>*\dir{>};
"s7"; "s10" ; **@{-} ?>*\dir{>};
"s10"; "s13" ; **@{-} ?>*\dir{>};
"s11"; "s13" ; **@{-} ?>*\dir{>};
"s12"; "s13" ; **@{-} ?>*\dir{>};
\endxy  
\end{align*}
where an arrow from $\mfC_i^j$ to $\mfC_{i-1}^k$ for some $i,j,k$ implies that $\breve{\mfC}_i^j \subset \g_1 \cdot \breve{\mfC}_{i-1}^k$ for any choice of irreducible $\co(n-2)$-modules  $\breve{\mfC}_i^j$ and $\breve{\mfC}_{i-1}^k$ linearly isomorphic to $\mfC_i^j$ and $\mfC_{i-1}^k$ respectively.
\end{prop}

\subsubsection{Projections}\label{sec-proj-sim}
Let $\mathfrak{D}$ denote any of $\mfF$, $\mfA$ or $\mfC$ (or any irreducible $\so(n-1,1)$-module for that matter), and $\{ \mfD^i \}$ the $\simalg(n-2)$-invariant filtration induced by a null line $\mfK$, and $\{ \gr_i(\mfD)\}$ its associated graded module with irreducibles $\mfD_i^j$. Given a choice of splitting, irreducible $\co(n-2)$-modules will be denoted $\breve{\mfD}_i^j$ as before.

In explicit computations, it is often easier to characterise elements of $\mfD$ that do not lie in a particular $\simalg(n-2)$-submodule. For this purpose, we introduce
the natural projections
\begin{align}\label{eq-proj-sim}
{}^\mfD \Pi_i & : \mfD \rightarrow \mfD /\mfD^{i+1} \, , &
{}^\mfD \breve{\Pi}_i^j & : \mfD \rightarrow \breve{\mfD}_i^j \, .
\end{align}
Clearly, the kernel of ${}^\mfD \Pi_i$ is the $\simalg(n-2)$-submodule $\mfD^{i+1}$. However, the kernel of ${}^\mfD \breve{\Pi}_i^j$ is not $\simalg(n-2)$-invariant, but depends on a choice of a splitting $\mfD_i$ and $\breve{\mfD}_i^j$ since it consists of $\mfD^{i+1}$, $\mfD_{i-1}$ and $\breve{\mfD}_i^k$ for all $k \neq j$. Since we are essentially concerned with $\simalg(n-2)$-invariant conditions, for any $D \in \mfD$, we shall write
\begin{align}\label{eq-proj-sim2}
{}^\mfD \Pi_i^j (D) & = 0 \, , & & \Longleftrightarrow & {}^\mfD \breve{\Pi}_i^j (D) & = 0 \, , & \mbox{for any splitting \eqref{eq-sim-grading}.}
\end{align}

\begin{rem}\label{rem-proj-sim}
One way of verifying the RHS of \eqref{eq-proj-sim2} is to trace down the arrows in the $\simalg(n-2)$-invariant graphs above starting from the module $\mfD_i^j$, in effect computing the action of $\g_{-1}$ on a given $\breve{\mfD}_i^j$, thereby produces a subgraph.\footnote{This subgraph is however not $\simalg(n-2)$-invariant, but invariant under $\g_{-1} \oplus \g_0$.} If, for a homogeneity $k,i$, say, the $\co(n-2)$-modules $\{ \breve{\mfD}_k^\ell \}$ lying in this subgraph are of distinct dimensions, then ${}^\mfD \breve{\Pi}_k^\ell (D) = 0$. However, for isotypic $\co(n-2)$-modules $\{\breve{\mfD}_k^\ell\}$ (i.e. of the same dimensions), one cannot in general expect ${}^\mfD \breve{\Pi}_k^\ell (D) = 0$, but algebraic conditions among $\{ {}^\mfD  \breve{\Pi}_k^\ell (D) \}$.
\end{rem}

\begin{rem}\label{rem-proj-sim-exp}
The maps ${}^\mfD \Pi_i^j$ occurring in \eqref{eq-proj-sim2} are expressed explicitly in appendix \ref{ref-spinor-descript-proj}, where they are interpreted as generalisations of the Bel-Debever criteria of \cite{Ortaggio2009b}.
\end{rem}

\subsubsection{Relation to the Petrov types in the null alignment formalism}
\label{rem-PW-types}
To make contact with reference \cite{Coley2004}, we recast a number of their definitions in our language:
\begin{itemize}
\item Let $\mfK$ be a null line in $(\mfV,g)$, i.e. an element of $\Gr_1 (\mfV,g)$ so that $\mfK$ induces a $\so(n-1,1)_\mfK$-invariant filtration \eqref{eq-filtration-C} on $\mfC$ where $\so(n-1,1)_\mfK \cong \simalg(n-2)$ is the stabiliser of $\mfK$. Then $\mfK$ is a \emph{Weyl aligned null direction (WAND)} of a Weyl tensor $C \ind{_{abcd}}$ if $C \ind{_{abcd}} \in \mfC^{-1}$.
\item A Weyl tensor $C_{abcd} \in \mfC$ is said to be (at least) of \emph{Petrov type I, II, III, or N,} respectively, if there exists a WAND $\mfK \in \Gr_1 (\mfV,g)$ with induced $\simalg(n-2)$-invariant filtration \eqref{eq-filtration-C} on $\mfC$, such that $C \ind{_{abcd}} \in \mfC^{-1}, \mfC^0, \mfC^1$, or $\mfC^2$, respectively.
\end{itemize}
In the null alignment formalism, one generally insists on choosing a null line with respect to which the Weyl tensor degenerates `most'. Thus, once such a null line has been fixed, the null alignment classification coincides with the present $\simalg(n-2)$-invariant classification, at least in broad terms.\footnote{Our approach emphasises the $\simalg(n-2)$-invariance of the classification of the Weyl tensor. For this reason, we have deliberately excluded from the discussion the Petrov subtypes $\mathrm{II_i}$, D and $\mathrm{III_i}$ defined in \cite{Coley2004}. Such subtypes break the $\simalg(n-2)$-invariance, but it is straightforward to relate these to our approach by fixing a splitting of the associated graded module $\gr(\mfC)$, in which case the $\simalg(n-2)$-invariant classification reduces to a $\co(n-2)$-invariant classification.} We refer the reader to the literature for details. We summarise the comparison between the terminology and notation of \cite{Coley2004}, and ours in the following table:
\begin{center}
{\renewcommand{\arraystretch}{1.1}
\renewcommand{\tabcolsep}{0.2cm}
\begin{tabular}{|c|c|c|c|}
\hline
Petrov types & $\Sim(n-2)$ condition & Petrov subtypes & $\Sim(n-2)$ condition \\
\hline
\hline
G & --  & -- & -- \\
\hline
\multirow{2}{*}{I} & \multirow{2}{*}{${}^\mfC \Pi_{-2}(C) =0$}
 & I(a) & ${}^\mfC \Pi_{-1}^0(C) = 0$  \\
  &  & I(b) & ${}^\mfC \Pi_{-1}^1(C) = 0$  \\
   \hline
\multirow{4}{*}{II} & \multirow{4}{*}{${}^\mfC \Pi_{-1}(C) =0$} 
  & II(a) & ${}^\mfC \Pi_0^0(C) = 0$ \\
  & & II(d) & ${}^\mfC \Pi_0^1(C) = 0$  \\
  & & II(b) & ${}^\mfC \Pi_0^2(C) = 0$  \\
  & & II(c) & ${}^\mfC \Pi_0^3(C) = 0$  \\
\hline
\multirow{2}{*}{III} & \multirow{2}{*}{${}^\mfC \Pi_0(C) =0$}
   & III(a) & ${}^\mfC \Pi_1^0(C) = 0$  \\
   & & III(b) & ${}^\mfC \Pi_1^1(C) = 0$  \\
  \hline
 N & ${}^\mfC \Pi_1 (C) =0$ & -- & -- \\
 \hline
 O & ${}^\mfC \Pi_2 (C) =0$ & -- & -- \\
\hline
\end{tabular}}
\end{center}
Here the maps ${}^\mfC \Pi_i^j$ are defined in section \ref{sec-proj-sim}. We must emphasise that the above subtypes are only the buildling blocks of other types given in \cites{Coley2004,Ortaggio2009b}. Thus, for instance, the Petrov subtype type II(abc) is characerised by the Weyl tensor belonging to the module $\mfC^0$ and satisfying ${}^\mfC \Pi_0^0 (C) = {}^\mfC \Pi_0^2 (C) = {}^\mfC \Pi_0^3 (C) = 0$, and so on. We finally note that the $\simalg(n-2)$-submodule characterised by ${}^\mfC \Pi_1^0 (C) = 0$ (since it implies ${}^\mfC \Pi_0^i (C) = 0$ for $i=0,1,2$ and ${}^\mfC \Pi_{-1} (C) = 0$) corresponds to the Petrov type denoted II'(abd) in \cite{Ortaggio2013}.

All together, when $n>5$, we count $28$ distinct $\Sim(n-2)$-submodules of $\mfC$. In terms of the Petrov types, these define
\begin{itemize}
\item each of types G, N and O;
\item $5$ subtypes of type I;
\item $17$ of type II;
\item $3$ of type III.
\end{itemize}
When $n=5$, we count $19$ distinct $\Sim(n-2)$-submodules of $\mfC$, $8$ of which being subtypes of type II.

\subsection{$\simalg(m-1,\C)$-invariant classifications}\label{sec-rob-class}
We now introduce a Robinson structure $\mfN$ on $(\mfV,g)$ with associated null line $\mfK$, i.e. $\mfN \cap \bar{\mfN} = {}^\C \mfK$. Assume $n=2m+\epsilon$ with $\epsilon \in \{ 0,1\}$. The null line $\mfK$ gives rise to the $\simalg(n-2)$-classifications of $\mfF$, $\mfA$ and $\mfC$ of the previous section, which we shall henceforth assume. There is a further splitting of the associated graded $\simalg(n-2)$-modules $\gr(\mfF)$, $\gr(\mfA)$ and $\gr(\mfC)$ into irreducible $\simalg(m-1,\C)$-modules, where $\simalg(m-1,\C)$ is the stabiliser of $\mfN$. In effect, this splitting is the result of the branching rule arising from the inclusion $\uu(m-1) \subset \so(m-1)$ of the semi-simple part $\g_0$ of $\simalg(n-2)$. The classifications again include $\simalg(m-1,\C)$-invariant graphs, reflecting the nilpotent action of $\simalg(m-1,\C)$. Unlike in the $\simalg(n-2)$ case, one does not need to distinguish between the cases $n=6$ and $n\neq6$ -- see appendix \ref{sec-low-dim}.

\begin{nota}
The even-dimensional and odd-dimensional cases differ only by the existence of additional irreducible $\simalg(m-1,\C)$-modules in the latter. The distinction between these two cases will be made by the appropriate value of $\epsilon$ and the use of gray fonts and dotted arrows in the $\simalg(m-1,\C)$-invariant graphs.
\end{nota}

As in the $\simalg(n-2)$ case, the irreducible $\cu(n-2)$-modules occuring in the classifications of  $\mfF$, $\mfA$ and $\mfC$ can all be expressed as the Cartan product of the $\cu(n-2)$-modules of $\mfV$ and $\g$, which we recall below:
\vspace{-6.5mm}
\begin{center}
\begin{minipage}[b]{0.45\linewidth}\centering
\begin{displaymath}
{\renewcommand{\arraystretch}{1.5}
\begin{array}{||c|c|c||}
\hline
\text{$\simalg_\C$-mod} & \text{$\cu$-mod} & \text{Dimension}  \\
 \hline
 \mfV_{\pm1}^0 & \mfV_{\pm1} & 1 \\
 \hline
\end{array}}
\end{displaymath}
\end{minipage}
\begin{minipage}[b]{0.45\linewidth}\centering
\begin{displaymath}
{\renewcommand{\arraystretch}{1.5}
\begin{array}{||c|c|c||}
\hline
\text{$\simalg_\C$-mod} & \text{$\cu$-mod} & \text{Dimension}  \\
 \hline
 \mfV_0^{0,0} & \dbl \mfV_0^{(1,0)} \dbr & 2m-2 \\
 \rowcolor{Gray} \mfV_0^{1,0} & [ \mfV_0^{(1,0)} ] & 1 \\
 \hline
\end{array}}
\end{displaymath}
\end{minipage}
\end{center}
\begin{center}
\begin{minipage}[b]{0.45\linewidth}\centering
\begin{displaymath}
{\renewcommand{\arraystretch}{1.5}
\begin{array}{||c|c|c||}
\hline
\text{$\simalg_\C$-mod} & \text{$\cu$-mod} & \text{Dimension}  \\
\hline
 \g_{\pm1}^{0,0} & \dbl \g_{\pm1}^{(1,0)} \dbr & 2m-2 \\
 \rowcolor{Gray}  \g_{\pm1}^{0,1} & [ \g_{\pm1}^{(0,0)} ] & 1 \\
\hline
\end{array}}
\end{displaymath}
\end{minipage}
\begin{minipage}[b]{0.45\linewidth}\centering
\begin{displaymath}
{\renewcommand{\arraystretch}{1.5}
\begin{array}{||c|c|c||}
\hline
\text{$\simalg_\C$-mod} & \text{$\cu$-mod} & \text{Dimension}  \\
\hline
 \g_0^0 & \mfz_0 & 1 \\
 \hline
 \g_0^{1,0} & [ \g_0^\omega ] & 1 \\
 \g_0^{1,1} & \dbl \g_0^{(2,0)} \dbr & (m-1)(m-2) \\
 \g_0^{1,2} & [ \g_0^{(1,1)_\circ} ] & m(m-2)\\
 \rowcolor{Gray} \g_0^{1,3} & \dbl \g_0^{(1,0)} \dbr & 2m-2 \\
 \hline
\end{array}}
\end{displaymath}
\end{minipage}
\end{center}

\subsubsection{The tracefree Ricci tensor}
\begin{prop}\label{prop-Ricci-rob}
The associated graded $\simalg(n-2)$-module $\gr(\mfF)$ of the filtration \eqref{eq-filtration-F} on the space $\mfF$ defined by \eqref{eq-rep-F} splits into a direct sum
\begin{align*}
 \mfF_{\pm2}^0 & = \mfF_{\pm2}^{0,0} \, , &
 \mfF_{\pm1}^0 & = \mfF_{\pm1}^{0,0} \oplus \epsilon \, \mfF_{\pm1}^{0,1} \, , \\
 \mfF_0^0 & = \mfF_0^{0,0} \, , & 
 \mfF_0^1 & = \left( \mfF_0^{1,0} \oplus \mfF_0^{1,1} \right) \oplus \epsilon \left( \mfF_0^{1,2} \oplus \mfF_0^{1,3} \right) \, .
\end{align*}
of irreducible $\simalg(m-1,\C)$-modules, where
\vspace{-6.5mm}
\begin{center}
\begin{minipage}[b]{0.45\linewidth}\centering
\begin{displaymath}
{\renewcommand{\arraystretch}{1.5}
\begin{array}{||c|c|c||}
\hline
\text{$\simalg_\C$-mod} & \text{$\cu$-mod} & \text{Dimension}  \\
\hline
 \mfF_{\pm2}^{0,0} & \mfV_{\pm1} \circledcirc \mfV_{\pm1} & 1 \\
\hline
\hline
 \mfF_{\pm1}^{0,0} & \mfV_{\pm1} \circledcirc \dbl\mfV_0^{(1,0)}\dbr & 2m-2 \\
 \rowcolor{Gray}  \mfF_{\pm1}^{0,1} & \mfV_{\pm1} \circledcirc \dbl\mfV_0^{(0,0)}\dbr & 1 \\
\hline
\end{array}}
\end{displaymath}
\end{minipage}
\begin{minipage}[b]{0.45\linewidth}\centering
\begin{displaymath}
{\renewcommand{\arraystretch}{1.5}
\begin{array}{||c|c|c||}
\hline
\text{$\simalg_\C$-mod} & \text{$\cu$-mod} & \text{Dimension}  \\
\hline
 \mfF_0^0 & \mfV_{-1} \circledcirc \mfV_1 & 1 \\
 \hline
 \mfF_0^{1,0} & [ \mfV_0^{(1,0)} \circledcirc \mfV_0^{(0,1)} ] & m(m-2) \\
 \mfF_0^{1,1} & \dbl \mfV_0^{(1,0)} \circledcirc \mfV_0^{(1,0)} \dbr & m(m-1)\\
 \rowcolor{Gray}  \mfF_0^{1,2} & [ \mfV_0^{(0,0)} \circledcirc \mfV_0^{(0,0)} ] & 1 \\
 \rowcolor{Gray}  \mfF_0^{1,3} & \dbl \mfV_0^{(0,0)} \circledcirc \mfV_0^{(1,0)} \dbr & 2m-2 \\
 \hline
\end{array}}
\end{displaymath}
\end{minipage}
\end{center}

Further, the $\simalg(m-1,\C)$-module $\gr(\mfF)$ can be expressed in terms of the $\simalg(m-1,\C)$-invariant graph
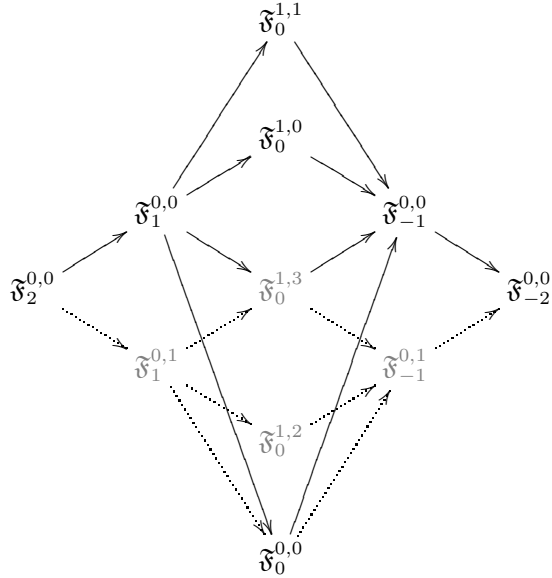
\begin{align*}
\xymatrix@R=1em{
& & \mfF_0^{1,1} \ar[dddr] & & \\
& & & & \\
& & \mfF_0^{1,0} \ar[dr] & & \\
& \mfF_1^{0,0} \ar[dddddr] \ar[dr] \ar[ur] \ar[uuur] & & \mfF_{-1}^{0,0} \ar[dr] & \\
\mfF_2^{0,0} \ar[ur] \ar@{.>}[dr] & & \color{gray} \mfF_0^{1,3} \ar[ur] \ar@{.>}[dr] & & \mfF_{-2}^{0,0} \\
& \color{gray} \mfF_1^{0,1} \ar@{.>}[dddr] \ar@{.>}[dr] \ar@{.>}[ur] & & \color{gray} \mfF_{-1}^{0,1} \ar@{.>}[ur] & \\
& & \color{gray} \mfF_0^{1,2} \ar@{.>}[ur] & & \\
& & & & \\
& & \mfF_0^{0,0} \ar@{.>}[uuur] \ar[uuuuur]& & 
}
\end{align*}
where an arrow from $\mfF_i^{j,k}$ to $\mfF_{i-1}^{p,q}$ for some $i,j,k,p,q$ implies that $\breve{\mfF}_i^{j,k} \subset \g_1 \cdot \breve{\mfF}_{i-1}^{p,q}$ for any choice of irreducible $\cu(m-1)$-modules  $\breve{\mfF}_i^{j,k}$ and $\breve{\mfF}_{i-1}^{p,q}$ linearly isomorphic to $\mfF_i^{j,k}$ and $\mfF_{i-1}^{p,q}$ respectively.
\end{prop}

\subsubsection{The Cotton-York tensor}
\begin{prop}\label{prop-CY-rob}
The associated graded $\simalg(n-2)$-module $\gr(\mfA)$ of the filtration \eqref{eq-filtration-A} on the space $\mfA$ defined by \eqref{eq-rep-A} splits into a direct sum
\begin{align*}
 \mfA_{\pm2}^0 & = \mfA_{\pm2}^{0,0} \oplus \epsilon \, \mfA_{\pm2}^{0,1} \, , \, , \\
 \mfA_{\pm1}^0 & = \mfA_{\pm1}^{0,0} \, , \qquad 
 \mfA_{\pm1}^1 = \left( \mfA_{\pm1}^{1,0} \oplus \mfA_{\pm1}^{1,1} \oplus \mfA_{\pm1}^{1,2} \right) \oplus \epsilon \,  \mfA_{\pm1}^{1,3} \, , \qquad
 \mfA_{\pm1}^2 = \left( \mfA_{\pm1}^{2,0} \oplus \mfA_{\pm1}^{2,1} \right) \oplus  \epsilon \left( \mfA_{\pm1}^{2,2}\oplus \mfA_{\pm1}^{2,3} \right) \, , \\
 \mfA_0^0 & = \mfA_0^{0,0} \oplus \epsilon \, \mfA_0^{0,0} \, , \qquad
 \mfA_0^1 = \mfA_0^{1,0} \oplus \epsilon \, \mfA_0^{1,1} \, , \\
 \mfA_0^2 & = \left( \mfA_0^{2,0} \oplus \mfA_0^{2,1} \oplus \mfA_0^{2,2} \oplus \mfA_0^{2,3} \right) \oplus \epsilon \left( \mfA_0^{2,4} \oplus \mfA_0^{2,5} \oplus \mfA_0^{2,6} \oplus \mfA_0^{2,7} \oplus \mfA_0^{2,8} \oplus \mfA_0^{2,9} \right) \, .
\end{align*}
of irreducible $\simalg(m-1,\C)$-modules, where
\vspace{-6.5mm}
\begin{center}
\begin{minipage}[b]{0.45\linewidth}\centering
\begin{displaymath}
{\renewcommand{\arraystretch}{1.5}
\begin{array}{||c|c|c||}
\hline
\text{$\simalg_\C$-} & \text{$\cu$-mod} & \text{Dimension}  \\
\hline
 \mfA_{\pm2}^{0,0} & \dbl \mfV_{\pm1} \circledcirc \g_{\pm1}^{(1,0)} \dbr & 2m-2 \\
 \rowcolor{Gray} \mfA_{\pm2}^{0,1} & \dbl \mfV_{\pm1} \circledcirc \g_{\pm1}^{(0,0)} \dbr & 1 \\
\hline
\hline
 \mfA_{\pm1}^0 & \mfV_{\pm1} \circledcirc \mfz_0 & 1 \\
 \mfA_{\pm1}^{1,0} & \dbl \mfV_{\pm1} \circledcirc \g_0^{(2,0)} \dbr & (m-1)(m-2)  \\
 \mfA_{\pm1}^{1,1} & [ \mfV_{\pm1} \circledcirc \g_0^{(1,1)_\circ} ] & m(m-2)  \\
 \mfA_{\pm1}^{1,2} & [ \mfV_{\pm1} \circledcirc \g_0^\omega ] & 1  \\
 \rowcolor{Gray} \mfA_{\pm1}^{1,3} & \dbl \mfV_{\pm1} \circledcirc \g_0^{(1,0)} \dbr & 2m-2  \\
 \hline
 \mfA_{\pm1}^{2,0} & \dbl \mfV_0^{(1,0)} \circledcirc \g_{\pm1}^{(1,0)} \dbr & m(m-1)\\
 \mfA_{\pm1}^{2,1} & \dbl \mfV_0^{(1,0)} \circledcirc \g_{\pm1}^{(0,1)} \dbr & m(m-2) \\
 \rowcolor{Gray} \mfA_{\pm1}^{2,2} & \dbl \mfV_0^{(0,0)} \circledcirc \g_{\pm1}^{(0,0)} \dbr & 1 \\
 \rowcolor{Gray} \mfA_{\pm1}^{2,3} & \dbl \mfV_0^{(0,0)} \circledcirc \g_{\pm1}^{(1,0)} \dbr & 2m-2 \\
 \hline
\end{array}}
\end{displaymath}
\end{minipage}
\begin{minipage}[b]{0.45\linewidth}\centering
\begin{displaymath}
{\renewcommand{\arraystretch}{1.5}
\begin{array}{||c|c|c||}
\hline
\text{$\simalg_\C$-} & \text{$\cu$-mod} & \text{Dimension}  \\
\hline
 \mfA_0^{0,0} & \dbl \mfV_0^{(1,0)} \dbr \circledcirc \mfz_0 & 2m-2  \\
 \rowcolor{Gray} \mfA_0^{0,1} & \mfV_0^{(0,0)} \circledcirc \mfz_0 & 1  \\
 \hline
 \mfA_0^{1,0} & \mfV_{\mp1} \circledcirc \dbl \g_{\pm1}^{(1,0)} \dbr & 2m-2 \\
 \rowcolor{Gray} \mfA_0^{1,1} & \mfV_{\mp1} \circledcirc \g_{\pm1}^{(0,0)} & 1 \\
 \hline
 \mfA_0^{2,0} & \dbl \mfV_0^{(1,0)} \circledcirc \g_0^\omega \dbr & 2m-2 \\
 \mfA_0^{2,1} & \dbl \mfV_0^{(1,0)} \circledcirc \g_0^{(2,0)} \dbr & \frac{2}{3}m(m-1)(m-2) \\ 
 \mfA_0^{2,2} & \dbl \mfV_0^{(0,1)} \circledcirc \g_0^{(2,0)} \dbr & m(m-1)(m-3) \\
 \mfA_0^{2,3} & \dbl \mfV_0^{(1,0)} \circledcirc \g_0^{(1,1)_\circ} \dbr & (m+1)(m-1)(m-2) \\
 \rowcolor{Gray} \mfA_0^{2,4} & \dbl \mfV_0^{(0,0)} \circledcirc \g_0^\omega \dbr & 1 \\
 \rowcolor{Gray} \mfA_0^{2,5} & \dbl \mfV_0^{(0,0)} \circledcirc \g_0^{(1,0)} \dbr & 2m-2 \\
 \rowcolor{Gray} \mfA_0^{2,6} & \dbl \mfV_0^{(0,0)} \circledcirc \g_0^{(2,0)} \dbr & (m-1)(m-2) \\
 \rowcolor{Gray} \mfA_0^{2,7} & \dbl \mfV_0^{(0,0)} \circledcirc \g_0^{(1,1)_\circ} \dbr & m(m-2) \\
 \rowcolor{Gray} \mfA_0^{2,8} & \dbl \mfV_0^{(0,1)} \circledcirc \g_0^{(1,0)} \dbr & m(m-2) \\
 \rowcolor{Gray} \mfA_0^{2,9} & \dbl \mfV_0^{(1,0)} \circledcirc \g_0^{(1,0)} \dbr & m(m-1) \\
 \hline
\end{array}}
\end{displaymath}
\end{minipage}
\end{center}

Further, the $\simalg(m-1,\C)$-module $\gr(\mfA)$ can be expressed in terms of the $\simalg(m-1,\C)$-invariant graph \ref{diagram-Penrose-A-rob} where an arrow\footnote{Arrow heads have been omitted for clarity.} from $\mfA_i^{j,k}$ to $\mfA_{i-1}^{p,q}$ for some $i,j,k,p,q$ implies that $\breve{\mfA}_i^{j,k} \subset \g_1 \cdot \breve{\mfA}_{i-1}^{p,q}$ for any choice of irreducible $\cu(m-1)$-modules  $\breve{\mfA}_i^{j,k}$ and $\breve{\mfA}_{i-1}^{p,q}$ linearly isomorphic to $\mfA_i^{j,k}$ and $\mfA_{i-1}^{p,q}$ respectively.
\end{prop}

\begin{table}
\begin{align*}
\xy
(-80,20)*+{\mfA_2^{0,0}}="s33",
(-80,-20)*+{\color{gray} \mfA_2^{0,1}}="s32",
(-50,40)*+{\mfA_1^{2,1}}="s31",
(-50,30)*+{\mfA_1^{2,0}}="s30",
(-50,20)*+{\color{gray} \mfA_1^{2,3}}="s29",
(-50,10)*+{\color{gray} \mfA_1^{2,2}}="s28",
(-50,0)*+{\mfA_1^{1,1}}="s27",
(-50,-10)*+{\mfA_1^{1,0}}="s26",
(-50,-20)*+{\color{gray} \mfA_1^{1,3}}="s25",
(-50,-30)*+{\mfA_1^{1,2}}="s24",
(-50,-40)*+{\mfA_1^{0,0}}="s23",
(0,82.5)*+{\mfA_0^{2,3}}="s22",
(0,67.5)*+{\mfA_0^{2,2}}="s21",
(0,52.5)*+{\mfA_0^{2,1}}="s20",
(0,37.5)*+{\color{gray} \mfA_0^{2,9}}="s19",
(0,22.5)*+{\color{gray} \mfA_0^{2,8}}="s18",
(0,7.5)*+{\color{gray} \mfA_0^{2,7}}="s17",
(0,-7.55)*+{\color{gray} \mfA_0^{2,6}}="s16",
(0,-22.5)*+{\mfA_0^{2,0}}="s15",
(0,-37.5)*+{\color{gray} \mfA_0^{2,5}}="s14",
(0,-52.5)*+{\color{gray} \mfA_0^{2,4}}="s13",
(0,-67.5)*+{\mfA_0^{0,0} \oplus \mfA_0^{1,0}}="s12",
(0,-82.5)*+{\color{gray} \mfA_0^{0,1}\oplus\mfA_0^{1,1}}="s11",
(50,40)*+{\mfA_{-1}^{2,1}}="s10",
(50,30)*+{\mfA_{-1}^{2,0}}="s9",
(50,20)*+{\color{gray} \mfA_{-1}^{2,3}}="s8",
(50,10)*+{\color{gray} \mfA_{-1}^{2,2}}="s7",
(50,0)*+{\mfA_{-1}^{1,1}}="s6",
(50,-10)*+{\mfA_{-1}^{1,0}}="s5",
(50,-20)*+{\color{gray} \mfA_{-1}^{1,3}}="s4",
(50,-30)*+{\mfA_{-1}^{1,2}}="s3",
(50,-40)*+{\mfA_{-1}^{0,0}}="s2",
(80,20)*+{\mfA_{-2}^{0,0}}="s1",
(80,-20)*+{\color{gray} \mfA_{-2}^{0,1}}="s0",
"s0"; "s2" ; **@{.} ; "s0"; "s4" ; **@{.} ; "s0"; "s7" ; **@{.} ; "s0"; "s8" ; **@{.} ; 
"s1"; "s2" ; **@{-} ; "s1"; "s5" ; **@{-} ; "s1"; "s6" ; **@{-} ; "s1"; "s3" ; **@{-} ; "s1"; "s4" ; **@{.} ; "s1"; "s9" ; **@{-} ; "s1"; "s10" ; **@{-} ; "s1"; "s8" ; **@{.} ; 
"s2"; "s12" ; **@{-} ; "s2"; "s11" ; **@{.} ;
"s3"; "s12" ; **@{-} ; "s3"; "s13" ; **@{.} ; "s3"; "s15" ; **@{-} ; 
"s4"; "s11" ; **@{.} ; "s4"; "s12" ; **@{.} ; "s4"; "s14" ; **@{.} ; "s4"; "s18" ; **@{.} ; "s4"; "s19" ; **@{.} ; 
"s5"; "s20" ; **@{-} ; "s5"; "s21" ; **@{-} ; "s5"; "s16" ; **@{.} ; "s5"; "s12" ; **@{-} ; 
"s6"; "s22" ; **@{-} ; "s6"; "s17" ; **@{.} ; "s6"; "s12" ; **@{-} ; 
"s7"; "s11" ; **@{.} ; "s7"; "s14" ; **@{.} ;
"s8"; "s11" ; **@{.} ; "s8"; "s12" ; **@{.} ; "s8"; "s13" ; **@{.} ; "s8"; "s14" ; **@{.} ; "s8"; "s16" ; **@{.} ; "s8"; "s17" ; **@{.} ; "s8"; "s18" ; **@{.} ; "s8"; "s19" ; **@{.} ;
"s9"; "s12" ; **@{-} ; "s9"; "s15" ; **@{-} ; "s9"; "s22" ; **@{-} ; "s9"; "s20" ; **@{-} ; "s9"; "s19" ; **@{.} ;
"s10"; "s12" ; **@{-} ; "s10"; "s15" ; **@{-} ; "s10"; "s21" ; **@{-} ; "s10"; "s22" ; **@{-} ; "s10"; "s18" ; **@{.} ;
"s11"; "s28" ; **@{.} ; "s11"; "s29" ; **@{.} ; "s11"; "s23" ; **@{.} ; "s11"; "s25" ; **@{.} ; 
"s12"; "s31" ; **@{-} ; "s12"; "s29" ; **@{.} ; "s12"; "s23" ; **@{-} ; "s12"; "s25" ; **@{.} ; "s12"; "s26" ; **@{-} ; "s12"; "s27" ; **@{-} ; "s12"; "s24" ; **@{-} ; "s12"; "s30" ; **@{-} ; 
"s13"; "s24" ; **@{.} ; "s13"; "s29" ; **@{.} ; 
"s14"; "s25" ; **@{.} ; "s14"; "s28" ; **@{.} ; "s14"; "s29" ; **@{.} ; 
"s15"; "s24" ; **@{-} ; "s15"; "s30" ; **@{-} ; "s15"; "s31" ; **@{-} ; 
"s16"; "s26" ; **@{.} ; "s16"; "s29" ; **@{.} ;
"s17"; "s27" ; **@{.} ; "s17"; "s29" ; **@{.} ;
"s18"; "s25" ; **@{.} ; "s18"; "s31" ; **@{.} ; "s18"; "s29" ; **@{.} ; 
"s19"; "s25" ; **@{.} ; "s19"; "s30" ; **@{.} ; "s19"; "s29" ; **@{.} ; 
"s20"; "s26" ; **@{-} ; "s20"; "s30" ; **@{-} ; 
"s21"; "s26" ; **@{-} ; "s21"; "s31" ; **@{-} ; 
"s22"; "s27" ; **@{-} ; "s22"; "s30" ; **@{-} ; "s22"; "s31" ; **@{-} ; 
"s23"; "s32" ; **@{.} ; "s23"; "s33" ; **@{-} ;
"s24"; "s33" ; **@{-} ;
"s25"; "s32" ; **@{.} ; "s25"; "s33" ; **@{.} ;
"s26"; "s33" ; **@{-} ;
"s27"; "s33" ; **@{-} ;
"s28"; "s32" ; **@{.} ;
"s29"; "s32" ; **@{.} ;
"s29"; "s33" ; **@{.} ;
"s30"; "s33" ; **@{-} ;
"s31"; "s33" ; **@{-} 
\endxy  
\end{align*}
\caption{$\simalg(m-1,\C)$-invariant graph for $\mfA$} \label{diagram-Penrose-A-rob}
\end{table}

\subsubsection{The Weyl tensor}
\begin{prop}\label{prop-Weyl-rob}
The associated graded $\simalg(n-2)$-module $\gr(\mfC)$ of the filtration \eqref{eq-filtration-C} on the space $\mfC$ defined by \eqref{eq-rep-C} splits into a direct sum
\begin{align*}
 \mfC_{\pm2}^0 & = \left( \mfC_{\pm2}^{0,0} \oplus \mfC_{\pm2}^{0,1} \right) \oplus \epsilon \left( \mfC_{\pm2}^{0,2} \oplus \mfC_{\pm2}^{0,3} \right) \, , \\
 \mfC_{\pm1}^0 & = \mfC_{\pm1}^{0,0} \oplus \epsilon \, \mfC_{\pm1}^{0,1} \, , \\
 \mfC_{\pm1}^1 & = \left( \mfC_{\pm1}^{1,0} \oplus \mfC_{\pm1}^{1,1} \oplus \mfC_{\pm1}^{1,2} \oplus \mfC_{\pm1}^{1,3} \right) \oplus \epsilon \left( \mfC_{\pm1}^{1,4} \oplus \mfC_{\pm1}^{1,5} \oplus \mfC_{\pm1}^{1,6} \oplus \mfC_{\pm1}^{1,7} \oplus \mfC_{\pm1}^{1,8} \oplus \mfC_{\pm1}^{1,9} \right) \, , \\
 \mfC_0^0 & = \mfC_0^0 \, , \\
 \mfC_0^1 & = \left( \mfC_0^{1,0} \oplus \mfC_0^{1,1} \oplus \mfC_0^{1,2} \right) \oplus \epsilon \, \mfC_0^{1,3} \, , \\
 \mfC_0^2 & = \left( \mfC_0^{2,0} \oplus \mfC_0^{2,1} \right) \oplus \epsilon \left( \mfC_0^{2,2} \oplus \mfC_0^{2,3} \right) \, , \\
 \mfC_0^3 & = \left( \mfC_0^{3,0} \oplus \mfC_0^{3,1} \oplus \mfC_0^{3,2} \oplus \mfC_0^{3,3} \oplus \mfC_0^{3,4} \oplus \mfC_0^{3,5} \oplus \mfC_0^{3,6} \oplus \mfC_0^{3,7} \right) \oplus \epsilon \left( \mfC_0^{3,8} \oplus \mfC_0^{3,9} \oplus \mfC_0^{3,10} \oplus \mfC_0^{3,11} \oplus \mfC_0^{3,12} \right) \, .
\end{align*}
of irreducible $\simalg(m-1,\C)$-modules, where
\vspace{-6.5mm}
\begin{center}
\begin{minipage}[b]{0.45\linewidth}\centering
\begin{displaymath}
{\renewcommand{\arraystretch}{1.5}
\begin{array}{||c|c|c||}
\hline
\text{$\simalg_\C$-} & \text{$\cu$-mod} & \text{Dimension} \\
\hline
  \mfC_{\pm2}^{0,0} & [ \g_{\pm1}^{(1,0)} \circledcirc \g_{\pm1}^{(0,1)} ] & m(m-2) \\
 \mfC_{\pm2}^{0,1} & \dbl \g_{\pm1}^{(1,0)} \circledcirc \g_{\pm1}^{(1,0)} \dbr & m(m-1) \\
 \rowcolor{Gray} \mfC_{\pm2}^{0,2} & [ \g_{\pm1}^{(0,0)} \circledcirc \g_{\pm1}^{(0,0)} ] & 1 \\
 \rowcolor{Gray} \mfC_{\pm2}^{0,3} & \dbl \g_{\pm1}^{(1,0)} \circledcirc \g_{\pm1}^{(0,0)} \dbr & 2m-2 \\
\hline \hline
 \mfC_{\pm1}^{0,0} & \dbl \g_{\pm1}^{(1,0)} \dbr \circledcirc \mfz_0 & 2m-2 \\
 \rowcolor{Gray} \mfC_{\pm1}^{0,1} & [ \g_{\pm1}^{(0,0)} ] \circledcirc \mfz_0  & 1\\
\hline
 \mfC_{\pm1}^{1,0} & \dbl \g_{\pm1}^{(1,0)} \circledcirc \g_0^{\omega} \dbr & 2m-2 \\
 \mfC_{\pm1}^{1,1} & \dbl \g_{\pm1}^{(1,0)} \circledcirc \g_0^{(2,0)} \dbr & {\scriptstyle \frac{2}{3}m(m-1)(m-2) }\\
 \mfC_{\pm1}^{1,2} & \dbl \g_{\pm1}^{(1,0)} \circledcirc \g_0^{(0,2)} \dbr & {\scriptstyle m(m-1)(m-3)} \\
 \mfC_{\pm1}^{1,3} & \dbl \g_{\pm1}^{(1,0)} \circledcirc \g_0^{(1,1)_\circ} \dbr & {\scriptstyle (m+1)(m-1)(m-2)} \\
 \rowcolor{Gray} \mfC_{\pm1}^{1,4} & [ \g_{\pm1}^{(0,0)} \circledcirc \g_0^{\omega} ] & 1 \\
 \rowcolor{Gray} \mfC_{\pm1}^{1,5} & \dbl \g_{\pm1}^{(0,0)} \circledcirc \g_0^{(1,0)} \dbr & 2m-2 \\
 \rowcolor{Gray} \mfC_{\pm1}^{1,6} & \dbl \g_{\pm1}^{(0,0)} \circledcirc \g_0^{(2,0)} \dbr & (m-1)(m-2) \\
 \rowcolor{Gray} \mfC_{\pm1}^{1,7} & \dbl \g_{\pm1}^{(0,0)} \circledcirc \g_0^{(1,1)_\circ} \dbr & m(m-2) \\
 \rowcolor{Gray} \mfC_{\pm1}^{1,8} & \dbl \g_{\pm1}^{(1,0)} \circledcirc \g_0^{(0,1)} \dbr & m(m-2) \\
 \rowcolor{Gray} \mfC_{\pm1}^{1,9} & \dbl \g_{\pm1}^{(1,0)} \circledcirc \g_0^{(1,0)} \dbr & m(m-1) \\
\hline
\end{array}}
\end{displaymath}
\end{minipage}
\begin{minipage}[b]{0.45\linewidth}\centering
\begin{displaymath}
{\renewcommand{\arraystretch}{1.5}
\begin{array}{||c|c|c||}
\hline
\text{$\simalg_\C$-} & \text{$\cu$-mod} & \text{Dimension} \\
\hline
\hline
 \mfC_0^{0,0}  & \mfz_0 \circledcirc \mfz_0 & 1 \\
\hline
 \mfC_0^{1,0} & \mfz_0 \circledcirc [ \g_0^\omega ] & 1 \\
 \mfC_0^{1,1} & \mfz_0 \circledcirc \dbl \g_0^{(2,0)} \dbr & (m-1)(m-2) \\
 \mfC_0^{1,2} & \mfz_0 \circledcirc [ \g_0^{(1,1)_\circ} ] & m(m-2) \\
 \rowcolor{Gray} \mfC_0^{1,3} & \mfz_0 \circledcirc \dbl \g_0^{(1,0)} \dbr & 2m-2 \\
\hline
 \mfC_0^{2,0} & [ \g_1^{(1,0)} \circledcirc \g_{-1}^{(0,1)} ] & m(m-2) \\
 \mfC_0^{2,1} & \dbl \g_1^{(1,0)} \circledcirc \g_{-1}^{(1,0)} \dbr & m(m-1) \\
 \rowcolor{Gray} \mfC_0^{2,2} & [ \g_1^{(0,0)} \circledcirc \g_{-1}^{(0,0)} ] & 1 \\
 \rowcolor{Gray} \mfC_0^{2,3} & \dbl \g_{\pm1}^{(1,0)} \circledcirc \g_{\mp1}^{(0,0)} \dbr & 2m-2 \\
\hline
 \mfC_0^{3,0}& [ \g_0^\omega \circledcirc \g_0^\omega ] & 1 \\
 \mfC_0^{3,1} & \dbl \g_0^{(2,0)} \circledcirc \g_0^\omega \dbr & (m-1)(m-2) \\
 \mfC_0^{3,2} & [ \g_0^{(1,1)_\circ}  \circledcirc \g_0^\omega ] & m(m-2) \\
 \mfC_0^{3,3} & \dbl \g_0^{(2,0)} \circledcirc \g_0^{(2,0)} \dbr & {\scriptstyle \frac{1}{6}m(m-1)^2(m-2)} \\
 \mfC_0^{3,4} & [ \g_0^{(2,0)} \circledcirc \g_0^{(0,2)} ] & {\scriptstyle \frac{1}{4}m(m-1)^2(m-4)} \\
 \mfC_0^{3,5} & [ \g_0^{(1,1)_\circ} \circledcirc \g_0^{(1,1)_\circ} ] & {\scriptstyle \frac{1}{4}(m+2)(m-1)^2(m-2)} \\
 \mfC_0^{3,6} & \dbl \g_0^{(2,0)} \circledcirc \g_0^{(1,1)_\circ} \dbr & {\scriptstyle \frac{2}{3}(m+1)(m-1)^2(m-3)} \\
 \rowcolor{Gray} \mfC_0^{3,7} & \dbl \g_0^{(1,0)} \circledcirc \g_0^\omega \dbr & 2m-2 \\
 \rowcolor{Gray} \mfC_0^{3,8} & [ \g_0^{(1,0)} \circledcirc \g_0^{(0,1)} ] & m(m-2) \\
 \rowcolor{Gray} \mfC_0^{3,9} & \dbl \g_0^{(1,0)} \circledcirc \g_0^{(1,0)} \dbr & m(m-1) \\
 \rowcolor{Gray} \mfC_0^{3,10} & \dbl \g_0^{(2,0)} \circledcirc \g_0^{(1,0)} \dbr & {\scriptstyle \frac{2}{3}m(m-1)(m-2)} \\
 \rowcolor{Gray} \mfC_0^{3,11} & \dbl \g_0^{(2,0)} \circledcirc \g_0^{(0,1)} \dbr & {\scriptstyle m(m-1)(m-3) }\\
 \rowcolor{Gray} \mfC_0^{3,12} & \dbl \g_0^{(1,0)} \circledcirc \g_0^{(1,1)_\circ} \dbr & {\scriptstyle (m+1)(m-1)(m-2)} \\
 \hline
\end{array}}
\end{displaymath}
\end{minipage}
\end{center}
with the proviso that when $n\leq4$, $\mfC_{\pm1}^{1,0}$ do not occur, and when $n \leq 5$, $\mfC_0^{3,2}$ does not occur.
 
Further, the $\simalg(m-1,\C)$-module $\gr(\mfC)$ can be expressed in terms of the $\simalg(m-1,\C)$-invariant graphs \ref{diagram-Penrose-C-rob} \& \ref{diagram-Penrose-C-rob_contd} where an arrow\footnote{Arrow heads have been omitted for clarity.} from $\mfC_i^{j,k}$ to $\mfC_{i-1}^{p,q}$ for some $i,j,k,p,q$ implies that $\breve{\mfC}_i^{j,k} \subset \g_1 \cdot \breve{\mfC}_{i-1}^{p,q}$ for any choice of irreducible $\cu(m-1)$-modules  $\breve{\mfC}_i^{j,k}$ and $\breve{\mfC}_{i-1}^{p,q}$ linearly isomorphic to $\mfC_i^{j,k}$ and $\mfC_{i-1}^{p,q}$ respectively.
\end{prop}

\begin{table}
\begin{align*}
\xy
(-160,30)*+{\mfC_2^{0,1}}="s51",
(-160,10)*+{\mfC_2^{0,0}}="s50",
(-160,-10)*+{\color{gray} \mfC_2^{0,3}}="s53",
(-160,-30)*+{\color{gray} \mfC_2^{0,2}}="s52",
(-100,82.5)*+{\mfC_1^{1,3}}="s43",
(-100,67.5)*+{\mfC_1^{1,2}}="s42",
(-100,52.5)*+{\mfC_1^{1,1}}="s41",
(-100,37.5)*+{\color{gray} \mfC_1^{1,9}}="s49",
(-100,22.5)*+{\color{gray} \mfC_1^{1,8}}="s48",
(-100,7.5)*+{\color{gray} \mfC_1^{1,7}}="s47",
(-100,-7.5)*+{\color{gray} \mfC_1^{1,6}}="s46",
(-100,-22.5)*+{\color{gray} \mfC_1^{1,5}}="s45",
(-100,-37.5)*+{\mfC_1^{1,0}}="s40",
(-100,-52.5)*+{\color{gray} \mfC_1^{1,4}}="s44",
(-100,-67.5)*+{\mfC_1^{0,0}}="s38",
(-100,-82.5)*+{\color{gray} \mfC_1^{0,1}}="s39",
(0,110)*+{\mfC_0^{3,6}}="s31",
(0,100)*+{\mfC_0^{3,5}}="s30",
(0,90)*+{\mfC_0^{3,4}}="s29",
(0,80)*+{\mfC_0^{3,3}}="s28",
(0,70)*+{\color{gray} \mfC_0^{3,12}}="s37",
(0,60)*+{\color{gray} \mfC_0^{3,11}}="s36",
(0,50)*+{\color{gray} \mfC_0^{3,10}}="s35",
(0,40)*+{\color{gray} \mfC_0^{3,9}}="s34",
(0,30)*+{\color{gray} \mfC_0^{3,8}}="s33",
(0,20)*+{\mfC_0^{3,2}}="s27",
(0,10)*+{\mfC_0^{3,1}}="s26",
(0,0)*+{\color{gray} \mfC_0^{3,7}}="s32",
(0,-10)*+{\mfC_0^{3,0}}="s25",
(0,-20)*+{\mfC_0^{2,1}}="s22",
(0,-30)*+{\mfC_0^{2,0}}="s21",
(0,-40)*+{\color{gray} \mfC_0^{2,3}}="s24",
(0,-50)*+{\color{gray} \mfC_0^{2,2}}="s23",
(0,-60)*+{\mfC_0^{1,2}}="s19",
(0,-70)*+{\mfC_0^{1,1}}="s18",
(0,-80)*+{\color{gray} \mfC_0^{1,3}}="s20",
(0,-90)*+{\mfC_0^{1,0}}="s17",
(0,-100)*+{\mfC_0^{0,0}}="s16",
"s16"; "s38" ; **@{-} ; "s16"; "s39" ; **@{.} ;
"s17"; "s38" ; **@{-} ; "s17"; "s40" ; **@{-} ; "s17"; "s44" ; **@{.} ;
"s18"; "s38" ; **@{-} ; "s18"; "s41" ; **@{-} ; "s18"; "s42" ; **@{-} ; "s18"; "s46" ; **@{.} ;
"s19"; "s38" ; **@{-} ; "s19"; "s43" ; **@{-} ; "s19"; "s47" ; **@{.} ;
"s20"; "s38" ; **@{.} ; "s20"; "s39" ; **@{.} ; "s20"; "s45" ; **@{.} ; "s20"; "s48" ; **@{.} ; "s20"; "s49" ; **@{.} ;
"s21"; "s38" ; **@{-} ; "s21"; "s40" ; **@{-} ; "s21"; "s42" ; **@{-} ; "s21"; "s43" ; **@{-} ; "s21"; "s48" ; **@{.} ;
"s22"; "s38" ; **@{-} ; "s22"; "s40" ; **@{-} ; "s22"; "s41" ; **@{-} ; "s22"; "s43" ; **@{-} ; "s22"; "s49" ; **@{.} ;
"s23"; "s39" ; **@{.} ; "s23"; "s45" ; **@{.} ;
"s24"; "s38" ; **@{.} ; "s24"; "s48" ; **@{.} ; "s24"; "s49" ; **@{.} ;
"s24"; "s39" ; **@{.} ; "s24"; "s44" ; **@{.} ; "s24"; "s45" ; **@{.} ; "s24"; "s46" ; **@{.} ; "s24"; "s47" ; **@{.} ;
"s25"; "s40" ; **@{-} ;
"s27"; "s40" ; **@{-} ; "s27"; "s43" ; **@{-} ;
"s26"; "s40" ; **@{-} ; "s26"; "s41" ; **@{-} ; "s26"; "s42" ; **@{-} ;
"s28"; "s41" ; **@{-} ;
"s29"; "s42" ; **@{-} ;
"s30"; "s43" ; **@{-} ;
"s31"; "s41" ; **@{-} ; "s31"; "s42" ; **@{-} ; "s31"; "s43" ; **@{-} ;
"s32"; "s40" ; **@{.} ; "s32"; "s44" ; **@{.} ; "s32"; "s48" ; **@{.} ; "s32"; "s49" ; **@{.} ;
"s33"; "s45" ; **@{.} ; "s33"; "s48" ; **@{.} ;
"s34"; "s45" ; **@{.} ; "s34"; "s49" ; **@{.} ;
"s35"; "s41" ; **@{.} ; "s35"; "s46" ; **@{.} ; "s35"; "s49" ; **@{.} ;
"s36"; "s42" ; **@{.} ; "s36"; "s46" ; **@{.} ; "s36"; "s48" ; **@{.} ;
"s37"; "s43" ; **@{.} ; "s37"; "s47" ; **@{.} ; "s37"; "s48" ; **@{.} ; "s37"; "s49" ; **@{.} ;
"s38"; "s50" ; **@{-} ; "s38"; "s51" ; **@{-} ; "s38"; "s53" ; **@{.} ;
"s39"; "s52" ; **@{.} ; "s39"; "s53" ; **@{.} ;
"s40"; "s50" ; **@{-} ; "s40"; "s51" ; **@{-} ;
"s41"; "s51" ; **@{-} ;
"s42"; "s50" ; **@{-} ;
"s43"; "s50" ; **@{-} ; "s43"; "s51" ; **@{-} ;
"s44"; "s53" ; **@{.} ;
"s45"; "s52" ; **@{.} ; "s45"; "s53" ; **@{.} ;
"s46"; "s53" ; **@{.} ;
"s47"; "s53" ; **@{.} ;
"s48"; "s50" ; **@{.} ; "s48"; "s53" ; **@{.} ;
"s49"; "s51" ; **@{.} ; "s49"; "s53" ; **@{.} ;
\endxy  
\end{align*}
\caption{$\simalg(m-1,\C)$-invariant graph for $\mfC$} \label{diagram-Penrose-C-rob}
\end{table}

\begin{table}
\begin{align*}
\xy
(0,110)*+{\mfC_0^{3,6}}="s31",
(0,100)*+{\mfC_0^{3,5}}="s30",
(0,90)*+{\mfC_0^{3,4}}="s29",
(0,80)*+{\mfC_0^{3,3}}="s28",
(0,70)*+{\color{gray} \mfC_0^{3,12}}="s37",
(0,60)*+{\color{gray} \mfC_0^{3,11}}="s36",
(0,50)*+{\color{gray} \mfC_0^{3,10}}="s35",
(0,40)*+{\color{gray} \mfC_0^{3,9}}="s34",
(0,30)*+{\color{gray} \mfC_0^{3,8}}="s33",
(0,20)*+{\mfC_0^{3,2}}="s27",
(0,10)*+{\mfC_0^{3,1}}="s26",
(0,0)*+{\color{gray} \mfC_0^{3,7}}="s32",
(0,-10)*+{\mfC_0^{3,0}}="s25",
(0,-20)*+{\mfC_0^{2,1}}="s22",
(0,-30)*+{\mfC_0^{2,0}}="s21",
(0,-40)*+{\color{gray} \mfC_0^{2,3}}="s24",
(0,-50)*+{\color{gray} \mfC_0^{2,2}}="s23",
(0,-60)*+{\mfC_0^{1,2}}="s19",
(0,-70)*+{\mfC_0^{1,1}}="s18",
(0,-80)*+{\color{gray} \mfC_0^{1,3}}="s20",
(0,-90)*+{\mfC_0^{1,0}}="s17",
(0,-100)*+{\mfC_0^{0,0}}="s16",
(100,82.5)*+{\mfC_{-1}^{1,3}}="s9",
(100,67.5)*+{\mfC_{-1}^{1,2}}="s8",
(100,52.5)*+{\mfC_{-1}^{1,1}}="s7",
(100,37.5)*+{\color{gray} \mfC_{-1}^{1,9}}="s15",
(100,22.5)*+{\color{gray} \mfC_{-1}^{1,8}}="s14",
(100,7.5)*+{\color{gray} \mfC_{-1}^{1,7}}="s13",
(100,-7.5)*+{\color{gray} \mfC_{-1}^{1,6}}="s12",
(100,-22.5)*+{\color{gray} \mfC_{-1}^{1,5}}="s11",
(100,-37.5)*+{\mfC_{-1}^{1,0}}="s6",
(100,-52.5)*+{\color{gray} \mfC_{-1}^{1,4}}="s10",
(100,-67.5)*+{\mfC_{-1}^{0,0}}="s4",
(100,-82.5)*+{\color{gray} \mfC_{-1}^{0,1}}="s5",
(160,30)*+{\mfC_{-2}^{0,1}}="s1",
(160,10)*+{\mfC_{-2}^{0,0}}="s0";
(160,-10)*+{\color{gray} \mfC_{-2}^{0,3}}="s3",
(160,-30)*+{\color{gray} \mfC_{-2}^{0,2}}="s2",
"s0"; "s4" ; **@{-} ; "s0"; "s6" ; **@{-} ; "s0"; "s8" ; **@{-} ; "s0"; "s9" ; **@{-} ; "s0"; "s14" ; **@{.} ;
"s1"; "s4" ; **@{-} ; "s1"; "s6" ; **@{-} ; "s1"; "s7" ; **@{-} ; "s1"; "s9" ; **@{-} ; "s1"; "s15" ; **@{.} ;
"s2"; "s5" ; **@{.} ; "s2"; "s11" ; **@{.} ;
"s3"; "s4" ; **@{.} ; "s3"; "s5" ; **@{.} ; "s3"; "s10" ; **@{.} ; "s3"; "s11" ; **@{.} ; "s3"; "s12" ; **@{.} ; "s3"; "s13" ; **@{.} ; "s3"; "s14" ; **@{.} ; "s3"; "s15" ; **@{.} ;
"s4"; "s16" ; **@{-} ; "s4"; "s17" ; **@{-} ; "s4"; "s18" ; **@{-} ; "s4"; "s19" ; **@{-} ; "s4"; "s21" ; **@{-} ; "s4"; "s22" ; **@{-} ; "s4"; "s20" ; **@{-} ; "s4"; "s24" ; **@{.} ;
"s5"; "s16" ; **@{.} ; "s5"; "s20" ; **@{.} ; "s5"; "s23" ; **@{.} ; "s5"; "s24" ; **@{.} ;  
"s6"; "s17" ; **@{-} ; "s6"; "s21" ; **@{-} ; "s6"; "s22" ; **@{-} ; "s6"; "s25" ; **@{-} ; "s6"; "s27" ; **@{-} ; "s6"; "s26" ; **@{-} ; "s6"; "s32" ; **@{.} ;
"s7"; "s18" ; **@{-} ; "s7"; "s22" ; **@{-} ; "s7"; "s26" ; **@{-} ; "s7"; "s28" ; **@{-} ; "s7"; "s31" ; **@{-} ; "s7"; "s35" ; **@{.} ;
"s8"; "s18" ; **@{-} ; "s8"; "s21" ; **@{-} ; "s8"; "s26" ; **@{-} ; "s8"; "s29" ; **@{-} ; "s8"; "s31" ; **@{-} ; "s8"; "s36" ; **@{.} ;
"s9"; "s19" ; **@{-} ; "s9"; "s21" ; **@{-} ; "s9"; "s22" ; **@{-} ; "s9"; "s27" ; **@{-} ; "s9"; "s30" ; **@{-} ; "s9"; "s31" ; **@{-} ; "s9"; "s37" ; **@{.} ;
"s10"; "s17" ; **@{.} ; "s10"; "s24" ; **@{.} ; "s10"; "s32" ; **@{.} ;
"s11"; "s20" ; **@{.} ; "s11"; "s23" ; **@{.} ; "s11"; "s24" ; **@{.} ; "s11"; "s33" ; **@{.} ; "s11"; "s34" ; **@{.} ;
"s12"; "s18" ; **@{.} ; "s12"; "s24" ; **@{.} ; "s12"; "s35" ; **@{.} ; "s12"; "s36" ; **@{.} ;
"s13"; "s19" ; **@{.} ; "s13"; "s24" ; **@{.} ; "s13"; "s37" ; **@{.} ;
"s14"; "s20" ; **@{.} ; "s14"; "s21" ; **@{.} ; "s14"; "s24" ; **@{.} ; "s14"; "s32" ; **@{.} ; "s14"; "s33" ; **@{.} ; "s14"; "s36" ; **@{.} ; "s14"; "s37" ; **@{.} ;
"s15"; "s20" ; **@{.} ; "s15"; "s22" ; **@{.} ; "s15"; "s24" ; **@{.} ; "s15"; "s32" ; **@{.} ; "s15"; "s34" ; **@{.} ; "s15"; "s35" ; **@{.} ; "s15"; "s37" ; **@{.} ;
\endxy  
\end{align*}
\caption{$\simalg(m-1,\C)$-invariant graph for $\mfC$ (continued)} \label{diagram-Penrose-C-rob_contd}
\end{table}

\subsubsection{Projections}\label{sec-proj-rob}
Continuing on from section \ref{sec-proj-sim} with $\mathfrak{D}=$ $\mfF$, $\mfA$ or $\mfC$, we denote by $\mfD_i^{j,k}$ the irreducible $\simalg(m-1,\C)$-modules of $\gr(\mfD)$, and by $\breve{\mfD}_i^{j,k}$ the corresponding $\cu(m-1)$-modules for a given splitting. We can then introduce the projections
\begin{align*}
{}^\mfD \breve{\Pi}_i^{j,k} & : \mfD \rightarrow \breve{\mfD}_i^{j,k} \subset \breve{\mfD}_i^j  \, , 
\end{align*}
for each $i,j,k$. This time, the kernel of ${}^\mfD \breve{\Pi}_i^{j,k}$ is not $\simalg(m-1,\C)$-invariant, as it depends on a choice of a splitting $\breve{\mfD}_i^{j,k}$. Since we are essentially concerned with $\simalg(m-1,\C)$-invariant conditions, for any $D \in \mfD$, we shall write
\begin{align}\label{eq-proj-rob2}
{}^\mfD \Pi_i^{j,k} (D) & = 0 \, , & & \Longleftrightarrow & {}^\mfD \breve{\Pi}_i^{j,k} (D) & = 0 \, , & \mbox{for any splitting \eqref{eq-sim-grading}.}
\end{align}

To verify the RHS of \eqref{eq-proj-rob2}, one can go through the same reasoning as Remark \ref{rem-proj-sim}. Regarding Remark \ref{rem-proj-sim-exp}, it must be said that explicit formulae for the maps ${}^\mfD \Pi_i^{j,k}$ occurring in \eqref{eq-proj-sim2} are far more involved than in the $\simalg(n-2)$ case -- see appendix \ref{ref-spinor-descript-proj}.
\newpage
\section{Geometric applications}\label{sec-geomexa}
We shall presently apply the algebraic machinery of sections \ref{sec-algebra} and \ref{sec-curvature} to Lorentzian geometry. Throughout this section,$(\mcM, g)$ will denote an oriented Lorentzian manifold of dimension $n$.  The Levi-Civita will be denoted by $\nabla_a$ with Riemann curvature tensor given by
\begin{align*}
R \ind{_{abd}^c} V \ind{^d} & := 2 \, \nabla \ind{_{[a}} \nabla \ind{_{b]}} V \ind{^c} \, ,
\end{align*}
for any vector field $V^a$. The Riemann tensor decomposes into $\SO(n-1,1)$-irreducibles. When $n>3$, the case we shall mostly be concerned with, this decomposition is given by
\begin{align}\label{eq-Riem-decomp}
R \ind{_{abcd}} & = C \ind{_{abcd}} +\frac{4}{n-2} \Phi \ind{_{[c|[a}}g \ind{_{b]|d]}} + \frac{2}{n(n-1)} R g \ind{_{a[c}} g \ind{_{d]b}} \, ,
\end{align}
where $C \ind{_{abcd}}$ is the Weyl tensor, $\Phi \ind{_{ab}}$ the tracefree part of the Ricci tensor $R \ind{_{ab}} := R \ind{_{acb}^c}$, and $R := R \ind{_a^a}$ the Ricci scalar. There is no Weyl tensor when $n \leq 3$.

The vector bundles on $\mcM$ of interest for us will be the bundles of irreducible curvature tensors. From a representational point of view, these can be defined $\mcD := \mcF \mcM \times_{\SO(n-1,1)} \mfD$, where $\mcF \mcM$ is the frame bundle, and $\mfD$ is an $\SO(n-1,1)$-irreducible representation -- in our case, $\mfD =$ $\mfF$, $\mfA$ or $\mfC$ as defined by \eqref{eq-rep-F}, \eqref{eq-rep-A} and \eqref{eq-rep-C} respectively.

We first briefly review spacetimes endowed with a distinguished null line distribution. Many geometric properties of such manifolds have been investigated, especially within the framework of the null alignment formalism, see \cite{Ortaggio2013} for a recent survey. We then examine the bundle generalisation of Robinson structures.

\subsection{Null line distributions}
Let $\mcK \subset \Tgt \mcM$ be a null line distribution, so that the structure group of the frame bundle $\mcF \mcM$ is reduced to $\Sim(n-2)$. This induces $\Sim(n-2)$-invariant subbundles of vector bundles constructed from the $\simalg(n-2)$-modules presented in section \ref{sec-sim-class} in the usual way. This applies in particular to the bundles of irreducible curvature tensors. These subbundles will simply be given by $\mcD^i = \mcF \mcM \times_{\Sim(n-2)} \mfD^i$, etc... in the obvious way (here $\mfD=$ $\mfF$, $\mfA$ or $\mfC$). We shall therefore draw from the notation of the previous sections and the appendices. In particular, we shall characterise the curvature tensors as elements of $\simalg(n-2)$-invariant subspaces in terms of the kernels of the maps ${}^\mfD \Pi_i^j$ defined by \eqref{eq-proj-sim} and \eqref{eq-proj-sim2} of section \ref{sec-proj-sim}.

Further, since many of the examples we shall consider are equipped with more than one distinguished null line distributions, it will be convenient to consider the bundle $\Gr_1(\Tgt \mcM,g)$ of all unoriented null line distributions on $(\mcM,g)$. This is the $S^{n-2}$-bundle whose fiber over a point $p$ of $\mcM$ is the null Grassmannian $\Gr_1(\Tgt_p \mcM,g)$. We shall refer to a null line distribution $\mcK$ specifically in the maps ${}^\mfD \Pi_i^j$ by writing ${}^\mfD _\mcK \Pi_i^j$.

A natural arena for the study of the geometric properties of $\mcK$ and its orthogonal complement $\mcK^\perp$ is provided by the \emph{intrinsic torsion} of the $\Sim(n-2)$-structure induced by $\mcK$. In broad terms, the intrinsic torsion associated to $\mcK$ splits into classes, which can be identified with $\simalg(n-2)$-invariant decompositions of $\nabla_a k_b \pmod{\alpha_a k_b}$. In general relativity, the resulting classification gives rise to the well-known characterisation of the congruence generated by $\mcK$ in terms of geometric optics: geodesy, shear, twist, dilation and parallelism. The latter is the strongest $\Sim(n-2)$-invariant differential condition on $\mcK$ that one can impose. Equivalently, any generator $k^a$ of $\mcK$ is recurrent with respect to $\nabla_a$, i.e.
$\left( \nabla \ind{_a} k \ind{^{[b}} \right) k \ind{^{c]}} = 0$. The holonomy of $\nabla_a$ is then contained in a subgroup of $\Sim(n-2)$. Local normal forms for such manifolds have been notably given in \cite{Galaev2010}. Here, we recall the integrability condition for the existence of such a vector field, which can be found, for instance, in \cites{Ortaggio2009b,Ortaggio2013a} and references therein. 

\begin{rem}
In the following statements and the rest of the paper, special features of low dimensions can be obtained by excluding those $\simalg(n-2)$-modules that do not occur there -- see section \ref{sec-curvature} and appendix \ref{sec-low-dim} for details.
\end{rem}

\begin{prop}\label{prop-rec-null-vec}
 Suppose that $(\mcM,g)$ admits a parallel null line distribution $\mcK$, and let $k^a$ be any section of $\mcK$ so that $(\nabla_a k^{[b}) k^{c]} = 0$. Then
 \begin{align*}
  R \ind{_{ab\lb{c}}^e} k \ind{_{\rb{d}}} k \ind{_e} & = 0 \, , \\
  {}^\mfC _\mcK \Pi_0^1 (C) & = 0 \, , \\
  {}^\mfF _\mcK \Pi_{-1}^0 (\Phi) & = 0 \, , & \mbox{i.e.} & & \Phi \ind{_{\lb{a}}^e} k \ind{_{\rb{b}}} k \ind{_e} & = 0 \, .
 \end{align*}
Further,
\begin{align*}
 {}^\mfF _\mcK \Pi_0^1 (\Phi) & = 0 & \Longleftrightarrow & & {}^\mfC _\mcK \Pi_0^2 (C) & = 0 \, ,
\end{align*}
and if any two of the following conditions
\begin{align*}
{}^\mfC _\mcK \Pi_0^0 (C) & = 0 \, , &  {}^\mfF _\mcK \Pi_0^0 (\Phi) & = 0 \, , & R & = 0 \, , 
\end{align*}
hold, the remaining one holds too.
\end{prop}

\begin{proof}[Sketch]
 It is a matter of differentiating the recurrent vector field twice and commuting the covariant derivatives. Using the decomposition \eqref{eq-Riem-decomp} then leads to ${}^\mfC _\mcK \Pi_0^1 (C) = 0$, and in addition,
\begin{align*}
{}^\mfC _\mcK \Pi_0^0 (C) \ind{_{ab}} & = \frac{n-4}{n-2} k \ind{_{(a}} \Phi \ind{_{b)c}} k \ind{^c} + \frac{n-2}{n(n-1)} R k \ind{_a} k \ind{_b} \, , & 
{}^\mfC _\mcK \Pi_0^2 (C) & = - \frac{4}{n-2} {}^\mfF _\mcK \Pi_0^1 (\Phi) \, ,
\end{align*}
which completes the proof.
\end{proof}

\vspace{2.5mm}

More restrictive yet are the Lorentzian manifolds, known as \emph{pp-waves}, that admit a \emph{parallel null vector field}. Note that this is no longer a $\Sim(n-2)$-invariant condition. Paralleling Proposition \ref{prop-rec-null-vec}, we have

\begin{prop}\label{prop-par-null-vec}
Let $k^a$ be a parallel null vector field on $(\mcM,g)$, i.e. $\nabla_a k^b = 0$. Then
 \begin{align*}
  R \ind{_{abc}^d} k \ind{_d} & = 0 \, , &
  \Phi \ind{_{ab}} k \ind{^b} & = - \frac{1}{n} R k \ind{_a} \, &
  {}^\mfC _\mcK \Pi_0^1 (C) & = 0 \, .
\end{align*}
Further,
\begin{align*}
 {}^\mfF _\mcK \Pi_0^0 (\Phi) & = 0 & \Longleftrightarrow & & R & = 0 & \Longleftrightarrow & & {}^\mfC _\mcK \Pi_0^0 (C) & = 0 \, , \\
  {}^\mfF _\mcK \Pi_0^1 (\Phi) & = 0 & \Longleftrightarrow & & {}^\mfC _\mcK \Pi_0^2 (C) & = 0 \, , \\
  {}^\mfF _\mcK \Pi_1^0 (\Phi) & = 0 & \Longleftrightarrow & & {}^\mfC _\mcK \Pi_1^0 (C) & = 0 \, .
\end{align*}
\end{prop}

\begin{proof}[Sketch]
Clearly, a parallel null vector field is a special type of recurrent null vector field. We can recycle the proof of Proposition \ref{prop-rec-null-vec}, and we find in addition
\begin{align*}
{}^\mfC _\mcK \Pi_0^0 (C) \ind{_{ab}} & = \frac{1}{(n-1)(n-2)} R k \ind{_a} k \ind{_b} \, , & 
{}^\mfC _\mcK \Pi_1^0 (C) \ind{_{abc}} & = \frac{2}{n-2} {}^\mfF _\mcK \Pi_1^0 (\Phi)_{abc} - \frac{2}{n(n-1)(n-2)} R k \ind{_{[a}} g_{b]c} \, ,
\end{align*}
and the result follows.
\end{proof}

\vspace{2.5mm}

\subsection{Almost Robinson structures}
\subsubsection{Basic definitions and properties}
We shall presently translate the algebraic setting of section \ref{sec-rob-alg} into the language of bundles. As we shall occasionally consider pseudo-Riemannian manifolds of arbitrary signature, we introduce the following definition.

\begin{defn}
Let $(\mcM,g)$ be a pseudo-Riemannian manifold of dimension $2m+\epsilon$ where $\epsilon \in \{0.1\}$. An \emph{almost null structure} is a totally null complex $m$-plane distribution on $\mcM$.
\end{defn}

Henceforth, unless otherwise stated, $(\mcM,g)$ will denote an $n$-dimensional smooth Lorentzian manifold, with $n=2m+\epsilon$, $\epsilon \in \{ 0,1\}$.

\begin{defn}
An \emph{almost Robinson structure} on $(\mcM,g)$ is an almost null structure of real index $1$. We say that $(\mcM,g)$ is an \emph{almost Robinson manifold} if it is equipped with an almost Robinson structure.
\end{defn}
In other words, an almost Robinson structure is a smooth assignment of a Robinson structure $\mcN_p$ on the tangent space $\Tgt_p \mcM$ at every point $p$ of $\mcM$. It defines a real null line distribution $\mcK$ on $\mcM$, i.e. ${}^\C \mcK_p := \mcN_p \cap \bar{\mcN}_p$ for all $p$ in $\mcM$ -- here, ${}^\C \mcK_p$ denotes the complexification of $\mcK_p$. Each fiber of the screenbundle $\mcK^\perp/\mcK$ is equipped with a Hermitian structure. An almost Robinson manifold is characterised by a reduction of the structure group of the frame bundle to $\Sim(m-1,\C)$ stabilising $\mcN$, the subgroup of the stabiliser $\Sim(n-2)$ of $\mcK$. We can then apply the $\simalg(m-1,\C)$-invariant decompositions of the various tensor representations of $\so(n-1,1)$ given in sections \ref{sec-algebra} and \ref{sec-curvature} to the curved setting, and in particular, in Propositions \ref{prop-Ricci-rob}, \ref{prop-CY-rob} and \ref{prop-Weyl-rob}. Again, to characterise curvature tensors as elements of some irreducible $\simalg(m-1,\C)$-invariant subspace, we shall refer to sections \ref{sec-proj-sim} and \ref{sec-proj-rob} and make use of the maps ${}^\mfD _\mcK \Pi_i$, ${}^\mfD _\mcK \Pi_i^j$ and ${}^\mfD _\mcN \Pi_i^{j,k}$ defined by \eqref{eq-proj-sim}, \eqref{eq-proj-sim2} and \eqref{eq-proj-rob2} respectively. Here, we have specified which almost Robinson structure $\mcN$ and its associated null line distribution $\mcK$ the maps refer to.

To make the description somewhat more tractable, it is more convenient to choose a splitting of ${}^\C \Tgt \mcM$ adapted to $\mcN$, and work with representations of the reductive part $\cu(m-1)$ of $\simalg(m-1,\C)$. To this end, we introduce a null line distribution $\mcL$ dual to $\mcK$, so that the almost Robinson structure gives rise to an almost Hermitian structure $J \ind{_a^b}$ on the distribution orthogonal to both $\mcK$ and $\mcL$, and ${}^\C \Tgt \mcM$ splits as
\begin{align}
\begin{aligned}\label{eq-splitting-tangent}
{}^\C \Tgt \mcM & = {}^\C \mcK \oplus \Tgt^{(1,0)} \oplus \Tgt^{(0,1)} \oplus \epsilon \Tgt^{(0,0)} \oplus {}^\C \mcL \, , \\
\mcN & = {}^\C \mcK \oplus \Tgt^{(1,0)} \, , &
\mcN^\perp & = {}^\C \mcK \oplus \Tgt^{(1,0)} \oplus \epsilon \Tgt^{(0,0)} \, ,
\end{aligned}
\end{align}
where $\Tgt^{(1,0)}$ and $\Tgt^{(0,1)}$ are the $\ii$- and $-\ii$-eigenbundle of $J \ind{_a^b}$ respectively, and when $n$ is odd, $\Tgt^{(0,0)}$ is its one-dimensional kernel. This splitting is the curved version of \eqref{eq-splitting-tangent}, and evidently carries through to the bundles of curvature tensors.

\begin{rem}
Unless otherwise specified, we shall not in general distinguish the even- and odd-dimensional cases. When applying any of the following statements to even dimensions, the reader should exclude those $\simalg(m-1,\C)$-modules occuring in odd dimensions only. The same applies to low dimensions, where a number of $\simalg(m-1,\C)$-modules do not occur. See section \ref{sec-curvature} and appendix \ref{sec-low-dim} for details.
\end{rem}

\paragraph{The bundle of almost Robinson structures}
For convenience, we can define the bundle $\Gr_m^1 ({}^\C \Tgt\mcM,g)$ of all almost Robinson structures on $(\mcM,g)$: the fiber over a point $p$ of $\mcM$ is the null Grassmannian $\Gr_m^1 ({}^\C \Tgt_p \mcM,g)$, which has dimension $m(m-1)+ \epsilon(2m-1)$. When $n=2m$, the bundle $\Gr_m^1({}^\C \Tgt\mcM, g)$ consists of the disjoint union of the bundle $\Gr_m^{+,1}({}^\C \Tgt\mcM, g)$ of all almost self-dual Robinson structures, and the bundle $\Gr_m^{-,1}({}^\C \Tgt\mcM, g)$ of all almost anti-self-dual Robinson structures. As usual, when $m$ is even, sections of $\Gr_m^{+,1} ({}^\C \Tgt\mcM, g)$ can be identified with those of $\Gr_m^{-,1} ({}^\C \Tgt\mcM, g)$ by complex conjugation.

\paragraph{Robinson structures}
As in the $\Sim(n-2)$ case above, almost Hermitian geometry \cite{Gray1980}, and complex Riemannian geometry \cites{Taghavi-Chabert2012a,Taghavi-Chabert2013}, the natural framework to study the differential properties of an almost Robinson structure is by classifying its intrinsic torsion. This is beyond the scope of the present article, and will be treated elsewhere. In this article, we shall nonetheless be concerned with integrable Robinson structures, which feature in a number of solutions of Einstein's field equations.

\begin{defn}\label{def-Robinson-manifold}
An almost null structure $\mcN$ on a pseudo-Riemannian manifold $(\mcM,g)$ is said to be
\begin{itemize}
 \item \emph{integrable} if $[\Gamma(\mcN),\Gamma(\mcN)] \subset \Gamma(\mcN)$,
 \item \emph{totally geodetic} if for all $X^a , Y^b \in \Gamma (\mcN)$, $X^a \nabla_a Y^b \in \Gamma (\mcN)$,
 \item \emph{co-integrable} if $[\Gamma(\mcN^\perp),\Gamma(\mcN^\perp)] \subset \Gamma(\mcN^\perp)$,
 \item \emph{totally co-geodetic} if for all $X^a , Y^b \in \Gamma (\mcN^\perp)$, $X^a \nabla_a Y^b \in \Gamma (\mcN^\perp)$.
\end{itemize}
When $(\mcM,g)$ is Lorentzian, the same definitions apply to almost Robinson structures. A (locally) integrable almost Robinson structure is called a \emph{(local) Robinson structure}, and a Lorentzian manifold equipped with a Robinson structure is called a \emph{Robinson manifold}.
\end{defn}

\begin{rem}
The definition of integrable almost Robinson structure was first introduced in \cite{Nurowski2002} in even dimensions, and then extended to the odd-dimensional case in \cite{Taghavi-Chabert2011}, where it is also referred to as an \emph{optical structure}. The above definitions differ from the one given in \cite{Taghavi-Chabert2011}, which did not reflect the subtleties involved in odd dimensions. In particular, what is referred to as an integrable almost Robinson structure or integrable almost optical structure in \cite{Taghavi-Chabert2011} is an almost Robinson structure that is both integrable and co-integrable according to Definition \ref{def-Robinson-manifold}.
\end{rem}

\begin{rem}
It is shown in \cite{Taghavi-Chabert2013} that, for any almost null structure $\mcN$, and thus for any almost Robinson structure,
\begin{itemize}
\item if $\mcN$ is totally geodetic, then it is integrable;
\item if $\mcN$ is both integrable and co-integrable, it is also totally geodetic;
\item if $\mcN$ is totally co-geodetic, then it is both integrable and co-integrable.
\end{itemize}
In even dimensions, since $\mcN = \mcN^\perp$, all definitions reduce to $\mcN$ being integrable.
\end{rem}

A shearfree congruence of null geodesics is equivalent to a Robinson structure in four dimensions, but this is not so in higher dimensions \cite{Trautman2002a}. In higher dimensions, we still wish to retain the geodesy property of the congruence of null curves generated by $\mcK$ associated to a Robinson structure $\mcN$. By the remark above, in even dimensions, this will be automatically satisfied by virtue of the integrability of $\mcN$. However, this property must be added in odd dimensions, and quite naturally follows from a totally geodetic Robinson structure. In fact, as we shall see in section \ref{sec-alg-sp-st}, it appears that \emph{co-integrable} Robinson structures play a bigger r\^{o}le in higher odd dimensions. Let us recall the properties enjoyed by a Robinson manifold.

\begin{prop}
 Let  $\mcN$ be a Robinson structure on $(\mcM,g)$ with associated real null line distribution $\mcK$, i.e. $\mcN \cap \bar{\mcN} = {}^\C \mcK$. Then $\mcK$ is tangent to a congruence of null curves whose leaf space is a CR manifold. Further, if $\mcN$ is totally geodetic, then this congruence is geodetic.
\end{prop}

\begin{proof}[Sketch]
For the \emph{almost} Robinson structure $\mcN$ to descend to an almost CR-structure on the leaf space of $\mcK$, one must require that $[ \bm{k}, \bm{V} ] \in \Gamma (\mcN)$ for any $\bm{k} \in \Gamma(\mcK)$, and all $\bm{V} \in \Gamma (\mcN)$. This is clearly satisfied when $\mcN$ is integrable. The integrability of the almost CR-structure then follows from that of $\mcN$. In even dimensions, this is proved in \cite{Nurowski2002} where the geodesy property of $\mcK$ follows automatically from the integrability of $\mcN$. In odd dimensions, this geodesy property must be an additional requirement.
\end{proof}

\vspace{2.5mm}

Since in dimensions two and three, an almost Robinson structure is equivalent to a real null line distribution, one obtains the following
\begin{lem}\label{lem-Rob-in23}
Suppose $(\mcM,g)$ is a two- or three-dimensional Lorentzian manifold. Then an integrable and co-integrable almost Robinson structure on $(\mcM,g)$ is equivalent to the existence a congruence of null geodesics. Further, there always exists a local Robinson structure, and when $\mcM$ is two-dimensional, there is a unique almost Robinson structure, which is always integrable.
\end{lem}

Using the results of \cites{Taghavi-Chabert2011,Taghavi-Chabert2012,Taghavi-Chabert2012a,Taghavi-Chabert2013} on almost null structures, the curvature condition for the existence of a totally geodetic Robinson structure $\mcN$ is given by
\begin{align}\label{eq-int-cond-Rob}
C_{abcd} X^a Y^b Z^c W^d & = 0 \, , & \mbox{for all $X^a,Y^a,Z^a \in \Gamma (\mcN)$, $W^a \in \Gamma (\mcN^\perp)$.}
\end{align}
This also applies to co-integrable almost Robinson structure \cites{Taghavi-Chabert2011,Taghavi-Chabert2012}. Taking the real span of \eqref{eq-int-cond-Rob} yields the following proposition.
\begin{prop}\label{prop-int-cond-Robinson}
Let $\mcN$ be a totally geodetic or co-integrable Robinson structure on $(\mcM,g)$. Then the Weyl tensor (locally) satisfies
\begin{subequations}\label{eq-int-condition}
\begin{align}
{}^\mfC _\mcN \Pi_{-2}^{0,i} (C) & = 0 \, , & \qquad \mbox{for $i=1,3$} \\
{}^\mfC _\mcN \Pi_{-1}^{1,i} (C) & = 0 \, , & \qquad \mbox{for $i=1,6,9$} \\
{}^\mfC _\mcN \Pi_0^{3,3} (C) & = 0 \, , & \qquad \mbox{for $i=3,10$} 
\end{align}
\end{subequations}
\end{prop}

\begin{proof}
We fix a splitting \eqref{eq-splitting-tangent} adapted to $\mcN$. Then with reference to appendix \ref{ref-spinor-descript-rep}, we have
\begin{align*}
C_{abcd} k^a X^b k^c W^d & = 0 & \Leftrightarrow & & {}^\mfC _\mcN \Pi_{-2}^{0,1} (C) & = 0 \, , \\
C_{abcd} k^a X^b k^c u^d & = 0 & \Leftrightarrow & & {}^\mfC _\mcN \Pi_{-2}^{0,3} (C) & = 0 \, , \\
C_{abcd} k^a X^b Y^c Z^d & = 0 & \Leftrightarrow & & {}^\mfC _\mcN \Pi_{-1}^{1,1} (C) & = 0 \, , \\
C_{abcd} k^a X^b Y^c u^d & = 0 & \Leftrightarrow & & {}^\mfC _\mcN \Pi_{-1}^{1,6} (C) = {}^\mfC _\mcN \Pi_{-1}^{1,9} (C) & = 0 \, , \\
C_{abcd} X^a Y^b Z^c W^d & = 0 & \Leftrightarrow & & {}^\mfC _\mcN \Pi_0^{3,3} (C) & = 0\, , \\
C_{abcd} X^a Y^b Z^c u^d & = 0 & \Leftrightarrow & & {}^\mfC _\mcN \Pi_0^{3,10} (C) & = 0\, ,
\end{align*}
for all local sections $k^a$ of $\mcK$, $X^a, Y^a, Z^a, W^a$ of $\Tgt^{(1,0)}$ and $u^a$ of $\Tgt^{(0,0)}$, hence the result.
\end{proof}

\vspace{2.5mm}

\begin{defn}\label{def-aligned-Rob}
We shall say that an almost Robinson structure $\mcN$ on a Lorentzian manifold is \emph{aligned} with the Weyl tensor if the Weyl tensor tensor satisfies \eqref{eq-int-cond-Rob}, or equivalently, \eqref{eq-int-condition} with respect to $\mcN$.
\end{defn}

\begin{rem}
We note that except for $n=4$, the integrability conditions \eqref{eq-int-condition} for the existence of a local Robinson structure does \emph{not} imply the Petrov type I condition with respect to its associated null line distribution $\mcK$ (or any section $k^a$ thereof):
\begin{align}\label{eq-Petrov-Weyl-I}
 {}^\mfC _\mcK \Pi_{-2}^0 (C) & = 0 \, , & \mbox{i.e.} & & k \ind{_{\lb{a}}} C \ind{_{\rb{b} e f \lb{c}}} k \ind{_{\rb{d}}} k \ind{^e} k \ind{^f} & = 0 \, ,
\end{align}
which is also the defining property for a null direction to be a \emph{Weyl aligned null direction (WAND)}.
\end{rem}

As shown in \cite{Taghavi-Chabert2011}, an example of a vacuum spacetime of Petrov type I that admits co-integrable Robinson structures, for each of which condition \eqref{eq-int-condition} is satisfied, is provided by the black ring solution of \cite{Emparan2002}. The following example is an explicit example of a Lorentzian manifold that does not admit any WAND (i.e. it is of Petrov type G), but that nevertheless admits almost Robinson structures aligned with the Weyl tensor.
\begin{exa}[The five-dimensional static Kaluza-Klein bubble]\label{exa-static-KK5}
The static Kaluza-Klein bubble is the direct product of a Euclidean Schwarzschild black hole and a one-dimensional timelike manifold. Its metric has thus the form
\begin{align*}
\bm{g} & = - \dd t^2 + \left(1-\frac{M}{r}\right)^{-1} \dd r^2 + \left(1-\frac{M}{r}\right) \dd z^2 + r^2 \dd \Omega^2 \, ,
\end{align*}
where $\dd \Omega^2$ is the $2$-sphere metric. Both the self-dual and anti-self-dual parts of the Weyl tensor of the four-dimensional Euclidean Schwarzschild metric are algebraically special. Thus, by the Goldberg-Sachs theorem, it locally admits two distinct Hermitian structures of opposite orientations.

Define the $1$-forms
\begin{align*}
\bm{\kappa} & := - \dd t + \left(1-\frac{M}{r} \right)^{-\frac{1}{2}} \dd r \, , &
\bm{\lambda} & :=  \dd t + \left(1-\frac{M}{r}\right)^{-\frac{1}{2}} \dd r \, , &
\bm{\nu} & := \left(1-\frac{M}{r} \right)^{\frac{1}{2}} \dd z \, ,
\end{align*}
and let $\bm{\mu}$ be a $(1,0)$-form of the Hermitian structure on $S^2$. Then, it is straightforward to check that the two sets of $1$-forms $\{\bm{\kappa},\bm{\mu},\bm{\nu}\}$ and $\{\bm{\lambda},\bm{\mu},\bm{\nu}\}$ define two distinct almost Robinson structures that are both integrable and co-integrable.

Further, since the almost Robinson structures are integrable and co-integrable, the Weyl tensor satisfies the condition \eqref{eq-int-condition} (or equivalently \eqref{eq-int-cond-Rob}). However, it is not algebraically `special' in the sense of \eqref{eq-alg-special}. For otherwise, it would be of Petrov type II (or D) in the null alignment formalism of \cites{Coley2004}. But it is shown in \cite{Godazgar2009} that this spacetime is algebraically general, i.e. of Petrov type G, i.e. there is no null direction with respect to which the Weyl tensor satisfies the weakest algebraic condition \eqref{eq-Petrov-Weyl-I} in this scheme. 

In addition, the two sets of $1$-forms $\{\bm{\kappa}',\bm{\mu}',\bm{\nu}'\}$ and $\{\bm{\lambda}',\bm{\mu}',\bm{\nu}'\}$ where
\begin{align*}
\bm{\kappa}' & := - \dd t + \left(1-\frac{M}{r} \right)^{\frac{1}{2}} \dd z \, , &
\bm{\lambda}' & := \dd t + \left(1-\frac{M}{r}\right)^{\frac{1}{2}} \dd z \, , &
\bm{\nu}' & := \left(1-\frac{M}{r} \right)^{-\frac{1}{2}} \dd r \, ,
\end{align*}
define two distinct almost Robinson structures. Since $\dd \bm{\kappa}' = \frac{M}{4r^2} \left( 1 - \frac{M}{r} \right)^{-\frac{1}{2}} \bm{\nu}' \wedge \left( \bm{\kappa}' + \bm{\lambda}' \right)$, these are clearly non-integrable. Computing the Weyl tensor shows that \eqref{eq-int-cond-Rob} is nonetheless satisfied with respect to any of these.
\end{exa}

\paragraph{The subbundle of almost Robinson structures incident on a null line distribution}
Recall that $(\mcM,g)$ is naturally equipped with the bundles $\Gr_1 (\Tgt \mcM,g)$ and $\Gr_m^1 ({}^\C \Tgt\mcM,g)$ of unoriented null line distributions and almost Robinson structures respectively, the latter splitting into a self-dual component $\Gr_m^{+,1} ({}^\C \Tgt\mcM,g)$ and an anti-self-dual component $\Gr_m^{-,1} ({}^\C \Tgt\mcM,g)$  when $n=2m$.

On an even-dimensional pseudo-Riemannian manifold of any signature, it is a standard result (see \cite{Trautman2002} and references therein), which follows directly from the integrability condition \eqref{eq-int-cond-Rob}, that if at a point $p$, each of the self-dual totally null $m$-planes is tangent to an integrable self-dual totally null $m$-plane distribution, then the Weyl tensor must vanish at $p$ when $m>2$, and its self-dual part when $m=2$. Imposing specific reality conditions on these distributions leads to different interpretation. For instance, in Euclidean signature, a hyper-K\"{a}hler manifold admits a two-sphere of global self-dual Hermitian structures, and must therefore be conformally half-flat in four dimensions. On the other hand, a four-dimensional Lorentzian manifold that admits a two-sphere of local Robinson structures must be locally conformally flat, and similarly in higher dimensions.

Instead, one can pick out a set of co-integrable, or more simply totally geodetic, Robinson structures. A natural way to go about this is to fix a real null line distribution $\mcK$, i.e. a section of $\Gr_1 (\Tgt \mcM,g)$. Then at every point $p$ of $\mcM$, the inverse image of the projection $\pi : \Gr_m^1 ({}^\C \Tgt_p \mcM,g) \rightarrow \Gr_1 (\Tgt_p \mcM,g)$ is simply $\pi^{-1} (\mcK_p) = \{ \mcN_p \in \Gr_m^1 ({}^\C \Tgt_p \mcM,g) : \mcN_p \cap \bar{\mcN}_p = {}^\C \mcK_p \}$, i.e. $\mcK_p$ determines a family of Robinson structures in $\Gr_m^1 ({}^\C \Tgt_p \mcM,g)$, precisely, those intersecting $\mcK_p$. By Proposition \ref{prop-null2Rob}, this family is parametrised by points in the homogeneous space $\OO(n-2)/\U(m-1)$ of dimension $(m-1)(m+2(\epsilon-1))$. When $n=2m$, this family splits into a self-dual family and an anti-self-dual family parametrised by the points of $\SO(n-2)/\U(m-1)$.

Taking the union of all $\pi^{-1} (\mcK_p)$ over all points $p$ forms the \emph{bundle of all almost Robinson structures incident on $\mcK$}, and which we shall denote $\Gr_m^1 ({}^\C \Tgt\mcM , g)_\mcK$. In even dimensions, we shall have an additional two subbundles, the bundle $\Gr_m^{+,1} ({}^\C \Tgt\mcM , g)_\mcK$ of all almost self-dual Robinson structures incident on $\mcK$, and the bundle $\Gr_m^{-,1} ({}^\C \Tgt\mcM , g)_\mcK$ of all almost anti-self-dual Robinson structures incident on $\mcK$.

\begin{prop}\label{prop-multi-Rob}
Let $(\mcM,g)$ be an $n$-dimensional Lorentzian manifold equipped with a null line distribution $\mcK$ with $n=2m+\epsilon$, $\epsilon \in \{0,1\}$. Then the Weyl tensor $C_{abcd}$ satisfies
\begin{align*}
{}^\mfC _\mcK \Pi_{-1}^1 (C) & = 0 \, , & \mbox{when $n \geq 5$,} \\
{}^\mfC _\mcK \Pi_0^3 (C) & = 0 \, , & \mbox{when $n > 5$,}
\end{align*}
at a point $p$ of $\mcM$ if and only if every element of $\Gr_m^1 ({}^\C \Tgt_p \mcM , g)_\mcK$ is aligned with the Weyl tensor at $p$, i.e. $C_{abcd}$ satisfies \eqref{eq-int-cond-Rob} (or equivalently \eqref{eq-int-condition}).

When $n=2m$, the Weyl tensor $C_{abcd}$ satisfies
\begin{align*}
{}^\mfC _\mcK \Pi_0^3 (C) & = 0 \, , & \mbox{when $m>3$,} \\
{}^\mfC _\mcK \Pi_0^{3,+} (C) & = 0 \, , & \mbox{when $m=3$,}
\end{align*}
at a point $p$ of $\mcM$ if and only if if every element of $\Gr_m^{+,1} ({}^\C \Tgt_p \mcM , g)_\mcK$ is aligned with the Weyl tensor at $p$ i.e. $C_{abcd}$ satisfies \eqref{eq-int-cond-Rob} (or equivalently \eqref{eq-int-condition}).
\end{prop}

\begin{proof}
Let $\mcN$ be an almost Robinson structure incident on $\mcK$ such that the Weyl tensor satisfies \eqref{eq-int-cond-Rob}. We then have
\begin{align}
C_{abcd} k^a Y^b k^c W^d & = 0 \, , \label{eq-multi-Rob-2} \\ 
C_{abcd} k^a Y^b Z^c W^d & = 0 \, , \label{eq-multi-Rob-1} \\
C_{abcd} X^a Y^b Z^c W^d & = 0 \, , \label{eq-multi-Rob0}
\end{align}
for all $k^a \in \Gamma ( \mcK)$, $X^a,Y^a \in \Gamma (\mcN^\perp)$, $Z^a, W^a \in \Gamma (\mcN)$, with $X^a,Y^a, Z^a,W^a \notin \Gamma (\mcK)$. We now let vary $\mcN$ while keeping $\mcK$ fixed, we then see that $\eqref{eq-multi-Rob-2} \Longrightarrow {}^\mfC _\mcK \Pi_{-2}^0 (C) = 0$, $\eqref{eq-multi-Rob-1} \Longrightarrow {}^\mfC _\mcK \Pi_{-1}^1 (C) = 0$ and $\eqref{eq-multi-Rob0} \Longrightarrow {}^\mfC _\mcK \Pi_0^3 (C) = 0$ -- this last condition is trivial in dimension five.

In the case, when $n=2m$ with $m>3$ odd and we restrict ourselves to almost self-dual Robinson structures, the above result remains unchanged. In the case $m=3$, however, conditions \eqref{eq-multi-Rob-1} and \eqref{eq-multi-Rob0} split into self-dual and anti-self-dual parts, so that restriction to almost self-dual Robinson structures yields the required condition ${}^\mfC _\mcK \Pi_0^{3,+} (C) = 0$ (which implies in particular ${}^\mfC _\mcK \Pi_{-1}^{1,+} (C) = 0$ and ${}^\mfC _\mcK \Pi_{-2}^0 (C) = 0$).
\end{proof}

Now applying Proposition \ref{prop-int-cond-Robinson} proves
\begin{cor}\label{cor-multi-int-Rob}
Let $(\mcM,g)$ be an $n$-dimensional Lorentzian manifold equipped with a null line distribution $\mcK$ with $n=2m+\epsilon$, $\epsilon \in \{0,1\}$. Suppose that 
at every point $p$ of a region $\mcU$ of $\mcM$ and for any $\mcN_p \in \Gr_m^1 ({}^\C \Tgt_p\mcM , g)_\mcK$, there exists a totally geodetic integrable almost Robinson structure tangent to $\mcN_p$. Then at every point of $\mcU$, the Weyl tensor $C_{abcd}$ satisfies
\begin{align*}
{}^\mfC _\mcK \Pi_{-1}^1 (C) & = 0 \, , & \mbox{when $n \geq 5$,} \\
{}^\mfC _\mcK \Pi_0^3 (C) & = 0 \, , & \mbox{when $n > 5$.}
\end{align*}

When $n=2m$, suppose that at every point $p$ of a region $\mcU$ of $\mcM$ and any $\mcN_p \in \Gr_m^{+,1} ({}^\C \Tgt_p\mcM , g)_\mcK$, there exists an integrable self-dual almost Robinson structure tangent to $\mcN_p$. Then the Weyl tensor $C_{abcd}$ satisfies
\begin{align*}
{}^\mfC _\mcK \Pi_0^3 (C) & = 0 \, , & \mbox{when $m>3$,} \\
{}^\mfC _\mcK \Pi_0^{3,+} (C) & = 0 \, , & \mbox{when $m=3$,}
\end{align*}
at every point of $\mcU$.
\end{cor}

\subsubsection{Spacetimes algebraically special along almost Robinson structures}\label{sec-alg-sp-st}
According to one interpretation, the Goldberg-Sachs theorem in four dimensions \cite{Goldberg2009} gives a relation between (local) Robinson structures and solutions to Einstein's field equations, whose Weyl tensor is \emph{algebraically special}, i.e. of Petrov type II or more degenerate. Among several equivalent approaches, this notion essentially rests on the classification of the multiplicites of the roots of the `Weyl spinor' \cites{Witten1959,Penrose1960}. 

The various notions of algebraic special Weyl tensors, equivalent in four dimensions, are no longer so in higher dimensions. In the case at hand, since we are interested in the geometric properties of an almost Robinson structure $\mcN$, a `special' algebraic degeneracy condition for the Weyl tensor should be defined with respect to $\mcN$.\footnote{Whether this can be made independent of such a choice is a different matter.} In \cites{Taghavi-Chabert2011,Taghavi-Chabert2012}, a simple and natural generalisation of the Petrov type II condition from four to higher dimensions was put forward:
\begin{defn}\label{defn-alg-sp}
A Weyl tensor $C_{abcd}$ on a $(2m+\epsilon)$-dimensional pseudo-Riemannian manifold $(\mcM,g)$, with $\epsilon \in \{ 0 ,1 \}$, is said to be \emph{(locally) algebraically special with respect to an almost null structure $\mcN$} if it satisfies
\begin{align}\label{eq-alg-special}
C_{abcd} X^a Y^b Z^c & = 0 \, , & & \mbox{for all $X^a\in \Gamma(\mcN^\perp)$ and all $Y^a ,Z^a \in \Gamma(\mcN)$,}
\end{align}
and we say that $\mcN$ is a \emph{repeated aligned almost null structure of the Weyl tensor}.

The same definitions apply to almost Robinson structures when $(\mcM,g)$ is Lorentzian.
\end{defn}
What \eqref{eq-alg-special} means for the present classification of the Weyl tensor is given by
\begin{prop}\label{prop-alg-sp}
Let $(\mcM,g)$ be a Lorentzian manifold equipped with an almost Robinson structure $\mcN$. Then condition \eqref{eq-alg-special} is equivalent to
\begin{subequations}\label{eq-GS-condition}
\begin{align}
{}^\mfC _\mcK \Pi_{-1} (C) & = 0 \, , \label{eq-GS-condition_sim} \\
{}^\mfC _\mcN \Pi_0^{1,i} (C) = {}^\mfC _\mcN \Pi_0^{2,i} (C) & = 0 \, , & \mbox{for $i=1,3$} \label{eq-GS-condition_rob1}\\
{}^\mfC _\mcN \Pi_0^{3,i} (C) & = 0 \, , & \mbox{for $i=1,3,6,7,9,10,11,12$} \label{eq-GS-condition_rob2} \\
{}^\mfC _\mcN \Pi_1^{1,i} (C) & = 0 \, , & \mbox{for $i=1,6,9$.} \label{eq-GS-condition_rob3}
\end{align}
\end{subequations}
\end{prop}

\begin{proof}
For clarity, we work in a splitting \eqref{eq-splitting-tangent} adapted to $\mcN$. By imposing reality conditions, condition \eqref{eq-alg-special} implies that $C_{abcd} k^a X^b Y^c Z^d = 0$ for all $k^a \in \Gamma (\mcK)$ and $X^a,Y^a,Z^a \in \Gamma ( \mcK^\perp )$, which is equivalently to \eqref{eq-GS-condition_sim}. The remaining conditions of \eqref{eq-GS-condition} can be obtained from Proposition \eqref{prop-int-cond-Robinson} and the additional conditions
\begin{align*}
C_{abcd} X^a Y^b Z^c \bar{W}^d & = 0 & \Leftrightarrow & & {}^\mfC _\mcN \Pi_0^{3,1} (C) = {}^\mfC _\mcN \Pi_0^{3,6} (C) & = 0 \, , \\
C_{abcd} u^a X^b Y^c \bar{W}^d & = 0 & \Leftrightarrow & & {}^\mfC _\mcN \Pi_0^{3,7} (C) = {}^\mfC _\mcN \Pi_0^{3,11} (C) = {}^\mfC _\mcN \Pi_0^{3,12} (C) & = 0 \, , \\
C_{abcd} X^a u^b Y^c u^d & = 0 & \Leftrightarrow & & {}^\mfC _\mcN \Pi_0^{3,9} (C) & = 0 \, , \\
C_{abcd} k^a X^b Y^c \ell^d & = 0 & \Leftrightarrow & & {}^\mfC _\mcN \Pi_0^{1,1} (C) = {}^\mfC _\mcN \Pi_0^{2,1} (C) & = 0 \, , \\
C_{abcd} k^a u^b X^c \ell^d = C_{abcd} k^a \ell^b u^c X^d & = 0 & \Leftrightarrow & & {}^\mfC _\mcN \Pi_0^{1,3} (C) = {}^\mfC _\mcN \Pi_0^{2,3} (C) & = 0 \, , \\
C_{abcd} X^a Y^b Z^c \ell^d & = 0 & \Leftrightarrow & & {}^\mfC _\mcN \Pi_1^{1,1} (C) & = 0 \, , \\
C_{abcd} u^a X^b Y^c \ell^d & = 0 & \Leftrightarrow & & {}^\mfC _\mcN \Pi_1^{1,6} (C) = {}^\mfC _\mcN \Pi_1^{1,9} (C) & = 0 \, ,
\end{align*}
for all local sections $k^a \in \Gamma (\mcK)$, $\ell^a \in \Gamma (\mcL)$, $X^a,Y^a,Z^a$ of $\Tgt^{(1,0)}$, $\bar{W}^a$ of $\Tgt^{(0,1)}$ and $u^a$ of $\Tgt^{(0,0)}$.
\end{proof}

We can now restate the generalisation of the Goldberg-Sachs theorem given in \cite{Taghavi-Chabert2012}.
\begin{thm}[\cites{Taghavi-Chabert2011,Taghavi-Chabert2012}]\label{thm-GS}
Let $(\mcM,g)$ be an Einstein Lorentzian manifold equipped with an almost Robinson structure $\mcN$. Suppose that the Weyl tensor is (locally) algebraically special along $\mcN$. Assume further that the Weyl tensor is otherwise generic. Then $\mcN$ is integrable and co-integrable.
\end{thm}

In \cites{Taghavi-Chabert2011,Taghavi-Chabert2012}, further degeneracy conditions on the Weyl tensors are also considered, and the Einstein condition is replaced by weaker conditions on the Cotton-York tensor, as in the version of \cites{Kundt1962,Robinson1963}, which reflect the conformal invariance of the theorem.

\begin{rem}
In four dimensions, shearfree congruences of null geodesics (SCNG) are equivalent to Robinson structures. Generalising the former to higher dimensions provides an alternative natural definition of `algebraically special' in line with the null alignment formalism of \cite{Coley2004}:
the Weyl tensor is said to be algebraically special if it admits a null direction $k^a$ with respect to which it satisfies
\begin{align}\tag{\ref{eq-GS-condition_sim}}
{}^\mfC _\mcK \Pi_{-1} (C) & = 0 \, , & \mbox{i.e.} & & C \ind{_{a d e \lb{b}}} k \ind{_{\rb{c}}} k \ind{^d} k \ind{^e} = k \ind{_{\lb{a}}} C \ind{_{b \rb{c} f \lb{d}}} k \ind{_{\rb{e}}} k \ind{^f} & = 0  \, ,
\end{align}
i.e. it is of Petrov type II (or more degenerate) -- see  section \ref{rem-PW-types}. In particular, by Proposition \ref{prop-alg-sp} tells us that \eqref{eq-GS-condition} implies that $C_{abcd}$ satisfies \eqref{eq-GS-condition_sim}, i.e. it is algebraically special in the null alignment formalism. But except in dimension four, the converse is clearly not true.

Also, the algebraic condition \eqref{eq-GS-condition_sim} leads to an `optical matrix' treatment of the Goldberg-Sachs theorem in higher dimensions \cites{Durkee2009,Ortaggio2012a,Ortaggio2013a}, but does not appear to be a sufficient condition for the existence of a SCNG.
\end{rem}

We now give an analogue of Proposition \ref{prop-multi-Rob} in the context of algebraically special spacetimes.
\begin{prop}\label{prop-multi-special-Rob}
Let $(\mcM,g)$ be an $n$-dimensional Lorentzian manifold equipped with a null line distribution $\mcK$, with $n=2m+\epsilon$, $\epsilon \in \{0,1\}$. The Weyl tensor $C_{abcd}$ satisfies
\begin{align*}
{}^\mfC _\mcK \Pi_1^1 (C) & = 0 \, ,
\end{align*}
at a point $p$ of $\mcM$ if and only if the Weyl tensor is algebraically special along all elements of $\Gr_m^1 ({}^\C \Tgt_p \mcM , g)_\mcK$ at $p$, i.e. it satisfies \eqref{eq-alg-special}  (or equivalently \eqref{eq-GS-condition}) with respect to any element of $\Gr_m^1 ({}^\C \Tgt_p \mcM , g)_\mcK$.

When $n=2m$, the Weyl tensor satisfies
\begin{align*}
{}^\mfC _\mcK \Pi_1^1 (C) & = 0 \, , & \mbox{when $m>3$,} \\
{}^\mfC _\mcK \Pi_1^{1,+} (C) & = 0 \, , & \mbox{when $m=3$,}
\end{align*}
at a point $p$ of $\mcM$ if and only if the Weyl tensor is algebraically special along all elements of $\Gr_m^{+,1} ({}^\C \Tgt_p \mcM , g)_\mcK$ at $p$, i.e.
\end{prop}

\begin{proof}
This is analogous to the proof of Proposition \ref{prop-multi-Rob}. Let $\mcN$ be an almost Robinson structure incident on $\mcK$. The algebraically special condition \eqref{eq-alg-special} can then be expressed as
\begin{align}
C_{abcd} k^a Y^b k^c & = 0 \, , \label{eq-multi-special-Rob-1} \\ 
C_{abcd} k^a Y^b Z^c & = 0 \, , \label{eq-multi-special-Rob0} \\
C_{abcd} X^a Y^b Z^c & = 0 \, , \label{eq-multi-special-Rob1}
\end{align}
for all $k^a \in \Gamma ( \mcK)$, $X^a,Y^a \in \Gamma (\mcN^\perp)$, $Z^a \in \Gamma (\mcN)$, with $X^a,Y^a, Z^a \notin \Gamma (\mcK)$. We now let vary $\mcN$ while keeping $\mcK$ fixed, we obtain $\eqref{eq-multi-special-Rob-1} \Longrightarrow {}^\mfC _\mcK \Pi_{-1}^0 (C) = {}^\mfC _\mcK \Pi_{-1}^1 (C) = 0$, $\eqref{eq-multi-special-Rob0} \Longrightarrow {}^\mfC _\mcK \Pi_0^1 (C) = {}^\mfC _\mcK \Pi_0^2 (C) = 0$, and $\eqref{eq-multi-special-Rob1} \Longrightarrow {}^\mfC _\mcK \Pi_1^1 (C) = {}^\mfC _\mcK \Pi_0^3 (C) = 0$.

Finally, this result remains unchanged if we restrict ourselves to almost self-dual Robinson structures when $n=2m>6$. When $n=6$, however, conditions \eqref{eq-multi-special-Rob-1} and \eqref{eq-multi-special-Rob0} split into self-dual and anti-self-dual parts, so that restriction to almost self-dual Robinson structures yields the required condition.
\end{proof}
\vspace{2.5mm}

\subsubsection{Examples of algebraically special spacetimes}
The existence of (local) integrable and co-integrable almost Robinson structures on higher-dimensional solutions to Einstein's field equations and their connection to the special algebraic condition \eqref{eq-alg-special} was first highlighted in \cites{Mason2010,Taghavi-Chabert2011}. Further instances of Robinson structures in higher dimensions were given in \cites{Ortaggio2012a,Ortaggio2013}, but their relation to the degeneracy of the Weyl tensor was not investigated there. We shall review these examples paying more attention to their curvature properties. We also examine the integrability conditions for the existence of a parallel Robinson structure and for the existence of a parallel pure spinor field of real index $1$.

\paragraph{Conformal Killing-Yano $2$-forms}
\begin{thm}[\cite{Mason2010}]\label{thm-CKY2form}
Let $(\mcM,g)$ be a $(2m+\epsilon)$-dimensional pseudo-Riemannian manifold, where $\epsilon \in \{ 0 , 1\}$, equipped with a conformal Killing-Yano $2$-form $\phi_{ab}$, i.e. a $2$-form $\phi \ind{_{ab}}$ that satisfies
\begin{align*}
\nabla \ind{_a} \phi \ind{_{bc}} & = \tau \ind{_{abc}} + 2 g \ind{_{a[b}} K \ind{_{c]}} \, ,
\end{align*}
where $\tau \ind{_{abc}} = \nabla_{[a} \phi_{bc]}$ and $1$-form $K_a = (2m+\epsilon-1) \nabla^c \phi_{ca}$. Assume that $\phi \ind{_a^b}$ has distinct eigenvalues. Let $\mcN$ be the totally null complex $m$-plane distribution associated to some (pure) eigenspinor of $\phi_{ab}$, and let $\mcN^\perp$ denote its orthogonal complement. Then the Weyl tensor satisfies \eqref{eq-alg-special} with respect to $\mcN$.

Suppose further that
\begin{align}\label{eq-tau}
\tau \ind{_{abc}} X^a Y^b Z^b & = 0 \, ,
\end{align}
for all $X^a, Y^b, Z^b \in \Gamma (\mcN^\perp)$. Then $\mcN$ and $\mcN^\perp$ are locally both integrable.
\end{thm}

As stated, the theorem applies to any pseudo-Riemannian manifolds, and while this article deals primarily with Lorentzian geometry, the following Riemannian example is worth mentioning as it sheds light on the r\^{o}le played by both the genericity of the eigenvalues of a CKY $2$-form and the additional condition \eqref{eq-tau}.

\begin{exa}[The Iwasawa manifold]\label{exa-Iwasawa}
The Iwasawa manifold is the quotient of the three-dimensional complex Heisenberg group by a discrete subgroup. The set of all invariant Hermitian structures on this six-dimensional real manifold is known to consist of the union of a point (its bi-invariant Hermitian structure) and a $2$-sphere (\cite{Ketsetzis2004} and references therein). It was also shown in \cite{Barberis2012} that it admits a \emph{Killing-Yano $2$-form}, i.e. a co-closed conformal Killing-Yano $2$-form, which is not closed. We will presently make the connection between these two geometric entities more explicit.

To describe the Iwasawa manifold, we introduce complex coordinates $\{z^1,z^2,z^3\}$ together with complex valued $1$-forms $\bm{\theta}^1 := \dd z^1$, $\bm{\theta}^2 := \dd z^2$ and $\bm{\theta}^3 := - \dd z^3 + z^1 \, \dd z^2$, so that the metric and the Killing-Yano $2$-form take the form
\begin{align*}
\bm{g} & = 2 \, \bm{\theta}^1 \odot \bar{\bm{\theta}}_1 + 2 \, \bm{\theta}^2 \odot \bar{\bm{\theta}}_2 + 2 \, \bm{\theta}^3 \odot \bar{\bm{\theta}}_3 \, , &
\bm{\phi} & = \ii \left( \bm{\theta}^1 \wedge \bar{\bm{\theta}}_1 + \bm{\theta}^2 \wedge \bar{\bm{\theta}}_2 + 3 \, \bm{\theta}^3 \wedge \bar{\bm{\theta}}_3 \right) \, ,
\end{align*}
respectively. As a spinor endomorphism, $\bm{\phi}$ has four complex conjugate pairs of eigenvalues, and each pair of projective \emph{pure} eigenspinors defines a conjugate pair of totally null complex $3$-plane distributions, i.e. an almost Hermitian structure. To be precise, $\bm{\phi}$ admits
\begin{itemize}
\item an eigenvalue $\frac{5\ii}{4}$ of multiplicity $1$ with associated distribution $\mcN^0$ annihilated by $\bm{\theta}^1 \wedge \bm{\theta}^2 \wedge \bm{\theta}^3$,
\item an eigenvalue $\frac{\ii}{4}$ of multiplicity $1$ with associated distribution $\mcN^{12}$ annihilated by $\bar{\bm{\theta}}_1 \wedge \bar{\bm{\theta}}_2 \wedge \bm{\theta}^3$,
\item an eigenvalue $3 \ii$ of multiplicity $2$ with an associated $2$-sphere of distributions $\mcN^{[a,b]}$ annihilated by $\left( a \, \bm{\theta}^1 + b \, \bm{\theta}^2 \right) \wedge \left( b \, \bar{\bm{\theta}}_1 - a \, \bar{\bm{\theta}}_2 \right) \wedge \bm{\theta}^3 $ where $[a,b] \in\CP^1$,
\end{itemize}
and similarly for their respective complex conjugates.

Since $\dd \bm{\theta}^1 = 0$, $\dd \bm{\theta}^2 = 0$ and $\dd \bm{\theta}^3 = \bm{\theta}^1 \wedge \bm{\theta}^2$,
we see at once that $\mcN^0$ and $\mcN^{[a,b]}$ for any $[a,b] \in\CP^1$ are integrable.\footnote{In fact, on inspection of their intrinsic torsion, they define \emph{Hermitian semi-K\"{a}hler} (also known as \emph{special Hermitian}) structures \cite{Gray1976}.} On the other hand, $\mcN^{12}$ is not integrable, but defines a \emph{quasi-K\"{a}hler} structure \cite{Gray1976}. Computing
\begin{align*}
\bm{\tau} & := \dd \bm{\phi} = 3 \ii \left( \bm{\theta}^1 \wedge \bm{\theta}^2 \wedge \bar{\bm{\theta}}_3 - \bar{\bm{\theta}}_1 \wedge \bar{\bm{\theta}}_2 \wedge \bm{\theta}^3 \right) \, ,
\end{align*}
we see that $\bm{\tau}$ is degenerate on $\mcN^0$ and $\mcN^{[a,b]}$ for any $[a,b] \in\CP^1$, but not on $\mcN^{12}$. Thus, as expected from the additional condition \eqref{eq-tau} of Theorem \ref{thm-CKY2form}, $\bm{\tau}$ is the obstruction to the integrability of $\mcN^{12}$.

As for the curvature, we first define
\begin{align*}
\bm{\mu} & = 2 \ii \left( \bm{\theta}^1 \wedge \bar{\bm{\theta}}_1 + \bm{\theta}^2 \wedge \bar{\bm{\theta}}_2 \right) \, , &
\bm{\nu} & = 2 \, \bm{\theta}^1 \odot \bar{\bm{\theta}}_1 + 2 \, \bm{\theta}^2 \odot \bar{\bm{\theta}}_2 \, , \\
\bm{\alpha} & = 2 \ii \, \bm{\theta}^3 \wedge \bar{\bm{\theta}}_3 \, , &
\bm{\beta} & = 2 \, \bm{\theta}^3 \odot \bar{\bm{\theta}}_3 \, .
\end{align*}
Then, in abstract index notation, the Weyl tensor, the tracefree Ricc tensor and the Ricci scalar are given by
\begin{align*}
C_{abcd} & = - \frac{1}{2} \left( \mu_{ab} \mu_{cd} - \mu_{a[c} \mu_{d]b} \right) + \frac{1}{2} \left( \mu_{ab} \alpha_{cd} + \alpha_{ab} \mu_{cd} -  2 \mu_{[a|[c} \alpha _{d]|b]} \right) + \frac{7}{10} \nu_{a[c} \nu_{d]b} + \frac{3}{5} \alpha_{ab} \alpha_{cd} - \frac{3}{5} \nu_{[a|[c} \beta_{d]|b]} \, ,\\
\Phi _{ab} & = \frac{2}{3} \nu_{ab} - \frac{4}{3} \beta_{ab} \, , \\
R & = 2 \, .
\end{align*}
In particular, the Weyl tensor is algebraically special with respect to any of $\mcN^0$, $\mcN^{12}$ and $\mcN^{[a,b]}$ for any $[a,b] \in\CP^1$.  The Iwasawa manifold is not Einstein, but the Cotton-York tensor is given by
\begin{align*}
\bm{A} & = \frac{2}{3} \left( \bar{\bm{\theta}}_3 \odot \left( \bm{\theta}^1 \wedge \bm{\theta}^2 \right) + \bm{\theta}^3 \odot \left( \bar{\bm{\theta}}_1 \wedge \bar{\bm{\theta}}_2 \right) \right) \, ,
\end{align*}
i.e. it is degenerate with respect to of $\mcN^0$ and $\mcN^{[a,b]}$ for any $[a,b] \in\CP^1$, but not with respect to $\mcN^{12}$. The Cotton-York tensor is therefore the obstruction to the integrability of $\mcN^{12}$ corroborating \cite{Taghavi-Chabert2011}.
\end{exa}

\begin{exa}[The Kerr-NUT-AdS metric]
The higher-dimensional Euclidean Kerr-NUT-AdS metric \cite{Chen2006} is a (partial) generalisation of the Plebanski-Demianski metric \cite{Plebanski1976}, and was shown to admit a conformal Killing-Yano tensor that satisfies the requirement of Theorem \ref{thm-CKY2form}  (in fact, with $\tau \ind{_{abc}}=0$) \cites{Kubizvn'ak2007,Kubizvn'ak2007a,Krtouvs2007,Krtouvs2008}. In even dimensions, it then follows that all $2^m$ projective eigenspinors of the CKY $2$-forms define Hermitian structures \cite{Mason2010}. These metrics also come in Lorentzian flavour, in which case they locally admit a discrete set of $2^m$ Robinson structures. Similar considerations apply in odd dimensions.
\end{exa}

Other metrics admitting a special CKY $2$-form as in Theorem \ref{thm-CKY2form} include those obtained by `switching off' the mass or NUT parameters, or cosmological constants of the Kerr-NUT-(A)dS metric, e.g. the Myers-Perry black hole \cites{Myers1986,Gibbons2005}. While one cannot apply Theorem \ref{thm-CKY2form} when the eigenvalues of the CKY $2$-form become degenerate, one can still deduce the existence of Robinson structures as a limiting case. This is precisely the case when all the rotation coefficients of the Kerr-NUT-(A)dS metric are set to zero, and in particular, for the Schwarzschild metric. which is discussed in Example \ref{exa-Schwarzschild}. 

\paragraph{Kerr-Schild metrics}
A remarkable property of the Kerr metric \cite{Gibbons2005} is that they can be put in \emph{Kerr-Schild form},  i.e. \emph{exact} first order perturbations of Minkowski space\footnote{Generalisation of this ansatz can also be obtained by considering non-flat background metrics.}, with a preferred null line distribution $\mcK$, as given by \eqref{eq-Kerr-Schild}. This makes the study of the differential properties of $\mcK$ particularly tractable, as already discussed in \cite{Ortaggio2009a}. There are further interesting curvature properties that these metrics enjoy when considering Robinson structures incident on $\mcK$.

\begin{prop}\label{prop-KS-Robinson}
Let $(\mcM,g)$ be a Lorentzian manifold equipped with a co-integrable Robinson structure $\mcN$. Let $\mcK$ be the associated null line distribution of $\mcN$, i.e. $\mcN \cap \bar{\mcN} = {}^\C \mcK$. Suppose that the metric on $\mcM$ is of the Kerr-Schild form
\begin{align}\label{eq-Kerr-Schild}
g \ind{_{ab}} & = \eta \ind{_{ab}} + H k \ind{_a} k \ind{_b} \, ,
\end{align}
where $\eta \ind{_{ab}}$ is a flat metric, $H$ is a smooth function, and $k^a$ is a section of $\mcK$. Then the Weyl tensor is algebraically special with respect to $\mcN$. In addition, the Ricci tensor and its tracefree part satisfy $R_{ab} X^a Y^b = 0$ and $\Phi_{ab} X^a Y^b = 0$ for all $X^a \in \Gamma(\mcN)$ and $Y^a \in \Gamma(\mcN^\perp)$, i.e. ${}^\mfF_\mcN \Pi_0^{1,1} (\Phi) = {}^\mfF_\mcN \Pi_0^{1,3} (\Phi) =0$.
\end{prop}

\begin{proof}
In \cite{Taghavi-Chabert2015}, it is shown that for a complex Riemannian manifold equipped with an almost null structure $\mcN$ and a metric of the form $g_{ab} = \eta_{ab} + H_{ab}$, where $\eta_{ab}$ is the flat complex Euclidean metric and $H_{ab}$ is a section of $\odot^2 \mcN$, if both $\mcN$ and $\mcN^\perp$ are integrable then the Weyl tensor satisfies \eqref{eq-alg-special}. This result can also be applied to (complexified) Lorentzian metric, in which case the additional structure forces $H_{ab}$ to be of the form $H k_a k_b$ for some function $H$ and real null vector field $k^a$. The condition on the Ricci tensor is given in the same reference.
\end{proof}
\vspace{2.5mm}

\begin{exa}[Myers-Perry black holes in Kerr-Schild form]\label{exa-Kerr-MP}
We introduce the standard flat coordinates $\{x ^a\}=\{t,x_i,y_i,\epsilon z\}$ on Minkowski space. The higher-dimensional Kerr-Schild ansatz \cites{Myers1986,Gibbons2005} describing a rotating black hole in higher dimensions admits a Kerr-Schild form \eqref{eq-Kerr-Schild} with
\begin{align*}
k \ind*{_a} \dd x ^a & :=  \dd t + \sum_{i=1} ^{m-1+\epsilon} \frac{r(x_i \dd x_i + y_i \dd y_i ) + a_i (x_i
   \dd y_i - y_i \dd x_i)}{r^2 +
   a_i^2} + (1-\epsilon) \frac{z \dd z}{r} \, , \\
U & := \frac{1}{r^{\epsilon+1}}\left(1-\sum_{i=1}^{m-1}\frac{a_i^2(x_i^2+y_i^2)}{(r^2+a_i^2)^2}\right) \prod_{j=1}^{n-1} (r^2+a_j^2) \, , & H & := \frac{2M}{U} \, .
\end{align*}
Here, the constants $a_i$ are the rotation coefficients and $M$ the mass of the black hole, and $r$ satisfies the constraint
\begin{align*}
\sum_{i=1}^{m-1}\frac{x_i^2+y_i^2}{r^2+a_i^2} + (\epsilon-1) \frac{z^2}{r^2} & = 1 \, .
\end{align*}
which arises as the determinant of the CKY $2$-form $\phi \ind{_a^b}$ for the metric. Solutions $r$ to this equation are the eigenvalues of $\phi \ind{_a^b}$. In fact, this metric admits a second Kerr-Schild ansatz with a null direction $\ell^a$, dual to $k^a$, obtained by sending $r$ to $-r$. The vector fields $k \ind{^a}$ and $\ell^a$ are two distinct real eigenvectors for the real eigenvalues $\pm r$ of $\phi_{ab}$ on $(\mcM,g)$. The remaining eigenvalues are purely imaginary with complex eigenvectors. Thus, by Theorem \ref{thm-CKY2form} one can associate a discrete set of $2^{m-1}$ local Robinson structures to each of $k ^a$ and $\ell^a$ \cites{Mason2010}. Finally, being a special case of the Kerr-NUT-(A)dS metric, the Weyl tensor is algebraically special with respect to any of the $2^m$ Robinson structures of the Kerr-Myers-Perry metric.
\end{exa}

\begin{exa}[Schwarzschild metric]\label{exa-Schwarzschild}
When all rotation coefficients $a_i$ are zero, the Kerr-Schild/Myers-Perry black hole of Example \ref{exa-Kerr-MP} degenerates to the Kerr-Schild form of the Schwarzschild black hole. Both the vector fields $k^a$ and $\ell^a$ become hypersurface-orthogonal, with
\begin{align*}
k _a \dd x ^a & :=  \dd t + \dd r \, , & \ell _a \dd x ^a & :=  \dd t - \dd r \, , & H & := \frac{2M}{r^{2n-3+\epsilon}} \, .
\end{align*}
We recast the flat metric of \eqref{eq-Kerr-Schild} into the form
\begin{align}\label{eq-Schwarzschild}
\eta_{ab} & = 2 \, k_{(a} \ell_{b)} + h_{ab} \, ,
\end{align}
where $h_{ab}$ is the round metric on the $(n-2)$-dimensional sphere $S^{n-2}$ (up to a factor $r^2$) -- the hypersurfaces orthogonal to both $k^a$ and $\ell^a$ are $(n-2)$-dimensional spheres $S^{n-2}$. 

The (co-integrable) Robinson structures of the Kerr metric persist as the parameters $a_i$ are set to zero in the Scharzschild metric \cite{Ortaggio2012a}. In fact, the rank of the CKY $2$-form of the Kerr-NUT-(A)dS metric will drop to $2$, which implies in particular that as a spinor endomorphism, the CKY $2$-form will have only a pair of eigenvalues $\pm r$, each of multiplicity $2^{m-1}$. This is not unlike Example \ref{exa-Iwasawa}, which admits infinitely many Hermitian structures. An eigenspinor for $\bm{\phi}$ need not be pure, but a pure eigenspinor will determine a (not necessarily integrable) almost Robinson structure incident on either $\mcK$ or $\mcL$. We shall presently, describe these pure eigenspinors in terms of their associated almost Robinson structures.

For our purpose, however, we can use the more standard form of the Schwarschild metric \eqref{eq-Schwarzschild}. For specificity, we assume $n=2m$. An arbitrary almost Robinson structure incident on $\mcK := \langle k^a \rangle$ must be spanned by $k^a$ and $m-1$ null complex vector fields $\{ \bm{\zeta}^A \}$, say, which define a local almost Hermitian structure on $S^{n-2}$. Since $k^a$ commutes with any of $\{\bm{\zeta}^A \}$, the question then boils down to seek \emph{integrable} almost Hermitian structures on $S^{n-2}$.

Except on $S^2$ there are no global Hermitian structures on $S^{n-2}$, but this is not so locally. We introduce the standard complex coordinates $\{ z^A , \bar{z}_A \}$ on an affine subset $\R^{n-2}$ of $S^{n-2}$ with flat metric $2 \, \dd z^A \odot \dd \bar{z}_A$. Then an arbitrary totally null complex $(m-1)$-plane distribution is spanned by vector fields of the form $\bm{\zeta}^A = \parderv{}{\bar{z}_A} + \phi ^{AB} \parderv{}{z^B}$ for some $\frac{1}{2}(m-1)(m-2)$ functions $\phi ^{AB}=\phi^{[AB]}$ on $\mcM$. For $\{ \bm{\zeta}^A \}$ to be integrable, these functions must satisfy $\bm{\zeta}^C \phi ^{AB} = 0$. The general analytic solution to this system of PDEs can be obtained from the prescription of a set of $\frac{1}{2}(m-1)(m-2)$ holomorphic functions in $\frac{1}{2}m(m-1)$ complex variables:\footnote{This defines an $(m-1)$-dimensional complex submanifold of \emph{twistor space} \cite{Hughston1988}.} this is the content of the Kerr theorem \cites{Cox1976,Penrose1986,Hughston1988}. Thus, there exist infinitely many local analytic Hermitian structures on $S^{n-2}$. Consequently, at every point $p$ of the Schwarzschild black hole $(\mcM,g)$, and for every element $\mcN_p$ of $\Gr_m^1({}^\C \Tgt_p \mcM , g)_\mcK$ and $\Gr_m^1({}^\C \Tgt_p \mcM , g )_\mcL$, we can find an integrable almost Robinson structure tangent to $\mcN_p$.

A similar argument applies in odd dimensions too: locally, on $S^{2m-1}$, there are infinitely many analytic metric-compatible CR structures whose orthogonal complements are integrable too. Their description is given in \cite{Taghavi-Chabert2015} where an odd-dimensional version of the Kerr theorem is also presented. In turn, these give rise to infinitely many Robinson  structures on the Schwarschild spacetime. 

All these (integrable or not) almost Robinson structures will subsist when considering the full Kerr-Schild metric. By Proposition \ref{prop-KS-Robinson}, locally the Weyl tensor of the Schwarzschild metric is \emph{algebraically special along any local integrable and co-integrable section of $\Gr_m^1 ({}^\C \Tgt\mcM,g)_\mcK$ and $\Gr_m^1 ({}^\C \Tgt\mcM,g)_\mcL$.} In particular, we can now apply Proposition \ref{prop-multi-special-Rob} to conclude that
\begin{align}\label{eq-Schwarzschild-alg-sp}
 {}^\mfC _\mcK \Pi_1^1 ( C ) & = 0 \, , & {}^\mfC _\mcL \Pi_1^1 ( C ) & = 0 \, ,
\end{align}
which corroborates already well-known results from the literature: the Schwarzschild metric is of type D(bcd) in the null alignment formalism \cites{Pravda2007b,Ortaggio2013} -- of course, one need not to have recourse to the above argument to work out the algebraic type of the Weyl tensor of the Schwarzschild metric, which follows directly from the fact that the Schwarschild metric is a warped metric conformal to a Lorentzian metric on $\mcM^{(1)} \times S^{n-2}$ where $\mcM^{(1)}$ is a two-dimensional Lorentzian manifold.

What is more, from Proposition \ref{prop-multi-special-Rob}, we discover that \emph{any} local, \emph{not necessarily integrable}, section of $\Gr_m^1 ({}^\C \Tgt\mcM,g)_\mcK$ and $\Gr_m^1 ({}^\C \Tgt\mcM,g)_\mcL$ will satisfy \eqref{eq-alg-special}. Thus the Schwarzschild metric admits infinitely many Robinson structures, and its Weyl tensor is algebraically special to infinitely many (not necessarily integrable) almost Robinson structures.
\end{exa}

In sum, the above example shows
\begin{prop}\label{prop-alg-sp-NInt}
On an Einstein Lorentzian manifold, not every almost Robinson structure with respect to which the Weyl tensor satisfies the algebraically special conditions \eqref{eq-alg-special} (i.e. \eqref{eq-GS-condition}) is integrable and co-integrable.
\end{prop}

\begin{rem}
An analogous statement in the null alignment formalism is that not every WAND is geodetic \cite{Godazgar2009}. In this case, any of the almost Robinson structures associated to a non-geodetic WAND is necessarily non-integrable.

Further examples of non-integrable totally null distributions on complex Riemannian or split signature manifolds, which satisfy \eqref{eq-alg-special}, were already pointed out in \cites{Taghavi-Chabert2012a,Taghavi-Chabert2013} and references therein.
\end{rem}

\paragraph{Robinson-Trautman}
The Schwarzschild metric belongs to the class of Robinson-Trautman metrics \cite{Robinson1961/1962} that have been generalised to higher dimensions in \cites{Podolsky2006a,Ortaggio2008}. These are spacetimes admiting an expanding, twistfree and shearfree congruence $\mcK$ of null geodesics. In other words, the orthogonal complement $\mcK^\perp$ of $\mcK$ is integrable, i.e. $[ \Gamma (\mcK^\perp) , \Gamma (\mcK^\perp) ] \subset \Gamma (\mcK^\perp)$, so that $\mcM$ is foliated by null hypersurfaces, and the conformal Riemannian metric on the screenspace $\mcK^\perp/\mcK$ is preserved along the flow of $\mcK$.

The local description of these metrics is given in \cite{Podolsky2006a}. There exist local coordinates $\{u,r,x^i \}$, where $u$ parametrises the field of hyperplanes tangent to $\mcK^\perp$, and $r$ is an affine parameter along a null geodesic of $\mcK$, so that $\bm{k} := \parderv{}{r}$ spans $\mcK$, and $\{ x^i \}$ are coordinates on the $(n-2)$-dimensional submanifolds $\mcH_u$ tangent to $\mcK^\perp$ and complementary to $\mcK$. Each of $\mcH_u$ is equipped with a metric $h_{ij} = h_{ij} (u;x)$. For simplicity, we assume that $(\mcM,g)$ is Ricci-flat. In this case, the Robinson-Trautman metric takes the local form
\begin{align}\label{eq-Robinson-Trautman}
\bm{g} & = 2 \, \dd u \odot \dd r + H \dd u \otimes \dd u + 2 \, r^2 h_{ij} \, \dd x^i \odot \dd x^j \, ,
\end{align}
where $H$ is a smooth function on $\mcM$, and $h_{ij} = h_{ij} (u;x)$ is a family of Einstein metrics on $\mcH_u$. Denote by $\mcL$ the null line distribution dual to $\mcK$, spanned by $\bm{\ell} := \parderv{}{u} - H \parderv{}{r}$. It turns out \cite{Podolsky2006a} that when $n>4$, the metric \eqref{eq-Robinson-Trautman} can only be of Petrov type D, with Weyl aligned null directions $\mcK$ and $\mcL$, or conformally flat. More specifically, in our notation, the Weyl curvature is given by
\begin{align}
\begin{aligned}\label{eq-curvature-RT}
 {}^\mfC _{\mcK} \Pi_0^0 (C) = {}^\mfC _{\mcL} \Pi_0^0 (C) & = \mu (u)  \, , \\
{}^\mfC _{\mcK} \Pi_0^1 (C) = {}^\mfC _{\mcK} \Pi_0^2 (C) = {}^\mfC _{\mcL} \Pi_0^1 (C) = {}^\mfC _{\mcL} \Pi_0^2 (C) & = 0 \, , \\
 {}^\mfC _{\mcK} \Pi_0^3 (C) = {}^\mfC _{\mcL} \Pi_0^3 (C) & =
 \begin{cases}
0 \, , & n = 5 \, ,  \\
  r^2 C_{ijk\ell} (u;x) \, , & n > 5 \, ,
\end{cases}
\end{aligned}
\end{align}
where $\mu (u)$ is some function, and for each $u$, $C_{ijk\ell} (u;x)$ is the Weyl tensor of the Einstein metric $h_{ij} (u;x)$ on $\mcH_u$. For the Schwarzschild metric, conditions \eqref{eq-curvature-RT} reduce further to \eqref{eq-Schwarzschild-alg-sp} since $C_{ijk\ell} (u;x) = 0$ for any $u$, and since this condition is always satisfied when $n=5$, the vacuum Robinson-Trautman metric coincides with the Schwarzschild metric then.

From conditions \eqref{eq-curvature-RT}, it is clear that any further properties of the conformal curvature of $(\mcM,g)$ will depend exclusively on the family of screenspace metrics $h_{ij} (u;x)$, and thus on any additional structure compatible with $h_{ij} (u;x)$. A natural candidate for such a structure is evidently an almost Hermitian structure.\footnote{It is noteworthy to mention \cite{Ortaggio2008} that the screenspace metrics of the non-vacuum Robinson-Trautman solutions must be \emph{almost K\"{a}hler}.} In even dimensions, each $C_{ijk\ell} (u;x)$ can be characterised in terms of the classification given in \cites{Tricerri1981,Falcitelli1994}. In odd dimensions, such a classification can be derived with no difficulty from the article \cite{Taghavi-Chabert2013}.

The case when $(\mcM,g)$ is a six-dimensional Robinson-Trautman spacetime is of particular interest: each manifold $\mcH_u$ on $\mcM$ is a four-dimensional Riemannian manifold, which means that one can apply the Petrov-Penrose classification of the Weyl tensor of $h_{ij} (u;x)$. The Weyl tensor associated to the screenspace Riemannian metric $h_{ij} (u;x)$ splits into a self-dual part $\Psi^+_{ijk\ell}$ and anti-self-dual part $\Psi^-_{ijk\ell}$, each of a particular Petrov type: $\Psi^+_{ijk\ell}$ is said to be of Petrov type D or algebraically special if there is an almost self-dual Hermitian structure $J \ind{_i^j}$ with respect to which $\Psi^+_{ijk\ell} X^i Y^j Z^k =0$ for any vector fields $X^i, Y^j, Z^k$ in the $\ii$-eigenbundle of $J \ind{_i^j}$, and similarly for $\Psi^-_{ijk\ell}$. Algebraically general and vanishing (anti)-self-dual Weyl tensors are called Petrov type G and O respectively -- see \cite{Gover2010} and references therein for details.

We can adopt our orientation convention to be such that
\begin{align*}
 {}^\mfC _{\mcK} \Pi_0^{3,\pm} (C) = {}^\mfC _{\mcL} \Pi_0^{3,\pm} (C) & =  r^2 \Psi^{\pm}_{ijk\ell} (u;x) \, .
\end{align*}
Now, suppose that $\mcN$ is an almost self-dual Robinson structure incident on $\mcK$. The following observations are easy to make:
\begin{itemize}
\item if ${}^\mfC _\mcN \Pi_0^{3,i}(C)=0$ for $i=1,3$ then $\Psi^+$ is of Petrov type D;
\item if ${}^\mfC _\mcN \Pi_0^{3,i}(C)=0$ for $i=0,1,3$ then $\Psi^+$ is of Petrov type O;
\item if ${}^\mfC _\mcN \Pi_0^{3,5}(C)=0$ then $\Psi^-$ is of Petrov type O.
\end{itemize}
For the converse, we note that generically $\Psi^+$ is of type G, which singles out a field self-dual $2$-plane, i.e. a principal almost Hermitian structure. This almost Hermitian structure can naturally be paired with any of the null line distributions $\mcK$ and $\mcL$. In particular, there always exists a pair of almost self-dual Robinson structures on $\mcM$, and the Weyl tensor must satisfy ${}^\mfC _\mcN \Pi_0^{3,3}(C)=0$. If $\Psi^+$ is of type D, respectively of type O, there always exists an almost Robinson structure $\mcN$ for which ${}^\mfC _\mcN \Pi_0^{3,i}(C)=0$ for $i=1,3$, respectively, ${}^\mfC _\mcN \Pi_0^{3,i}(C)=0$ for $i=0,1,3$. Finally, if $\Psi^-$ is of Petrov type O, we can find an almost Robinson structure such that ${}^\mfC _\mcN \Pi_0^{3,5}(C)=0$.

Since $h_{ij} (u;x)$ is Einstein, we can go a step further by applying the four-dimensional Riemannian Goldberg-Sachs theorem \cites{Przanowski1983,Nurowski1993,Apostolov1997,Gover2010}. If $h_{ij} (u;x)$ is algebraically special, then $\mcH_u$ admits a Hermitian structure, which lifts to a Robinson structure on $\mcM$. Conversely, suppose that $\mcN$ is a (locally) integrable almost Robinson structure on $\mcM$ incident on $\mcK$. Then $\mcN$ restricts to a Hermitian structure on each $\mcH_u$. In particular, by the Goldberg-Sachs theorem, the Weyl tensor is algebraically degenerate.

Summing up these observations, we have
\begin{prop}\label{prop-Robinson-Trautman-6}
Let $(\mcM,g)$ be a Ricci-flat six-dimensional Robinson-Trautman spacetime, i.e. a Ricci-flat Lorentzian manifold equipped with a null line distribution $\mcK$ that generates a twistfree, shearfree and expanding congruence of null geodesics. Let $\mcN$ be an almost Robinson structure incident on $\mcK$, i.e. $\mcN$ is a (local) section of $\Gr_m^1 ({}^\C \Tgt\mcM,g)_\mcK$. Then locally the Weyl tensor is algebraically special with respect to $\mcN$ if and only if $\mcN$ is integrable.
\end{prop}

In a similar vein, one could extend this analysis to Kundt spacetimes and warped product of manifolds, which share some of the properties of the Robinson-Trautman metrics. There is however too little space for this purpose here.

\paragraph{Higher-dimensional Taub-NUT-(A)dS metrics}
It was pointed out in \cite{Ortaggio2013a} that in six dimensions these metrics admit local Robinson structures. Here, we extend this to higher dimensions and examine the algebraic properties of their Weyl tensor.

\begin{exa}\label{exa-Taub-NUT-(A)dS}
These metrics generalise the four-dimensional Taub-NUT-(A)dS of \cites{Newman1963,Taub2004a,Hawking1999a}, and are constructed as metrics on a spacetime whose boundary at infinity is a $\U(1)$-fibration over a $(2m-2)$-dimensional K\"{a}hler-Einstein manifold $(\mcB,H,J)$. In fact, they admit an odd-dimensional version by constructing a $\U(1)$-fibration over the first factor of the direct product $\mcB \times \mcY$ where $\mcY$ is an odd-dimensional Riemannian manifold. Throughout this example, we suspend the Einstein summation convention. We shall also specialise to the case where $\mcY$ is simply $\R$. To describe the construction, we introduce a (local) unitary basis $\{ \bm{\mu}^A ,\bar{\bm{\mu}}_A \}_{A=2,\ldots,m}$ on $\mcB$ adapted to the K\"{a}hler structure, i.e.
\begin{align*}
\dd \bm{\mu}^A & = \sum_{B=2}^m \bm{\mu}^B \wedge \bm{\alpha} \ind{_B^A} \, ,
\end{align*}
where $\bm{\alpha} \ind{_B^A}$ is the connection $1$-form for the K\"{a}hler metric on $\mcB$. In odd dimensions, we take $\{\bm{\mu}^0\}$ to span $\Tgt^* \mcY$ with $\dd \bm{\mu}^0 = 0$. The Taub-NUT-(A)dS metric of \cites{Awad2002,Mann2004,Awad2006} is then given by
\begin{align}\label{eq-TaubNUT}
\bm{g} & = F^{-1} \dd r^2 - F \, \left( \dd t + \bm{A} \right)^2 + 2 \, \sum_{A=2}^m (r^2 + N_A^2 ) \bm{\mu}^A \odot \bar{\bm{\mu}}_A \,  + \, \epsilon r^2 \bm{\mu}^0 \otimes \bm{\mu}^0 \, ,
\end{align}
where $t$ is the coordinate on the $\U(1)$-fiber, the constants $N_A$ are the NUT parameters, $F=F(r)$ is a smooth function of $r$. The $1$-form $\bm{A}$ can be thought of as a K\"{a}hler potential in the sense that it satisfies
\begin{align*}
\dd \bm{A} & = 2 \, \ii \sum_{C=2}^m N_C^2 \bm{\mu}^C \wedge \bar{\bm{\mu}}_C \, .
\end{align*}
In what follows, we shall make either of the following two assumptions:
\begin{itemize}
\item either we take $\mcB = \mcB^1 \times \ldots \times \mcB^m$ where each $\mcB^i$ is a $2$-dimensional Riemannian manifold of constant curvature, i.e. spheres, tori, or hyperboloids;
\item or, $\mcB$ is a more general K\"{a}hler manifold, but we impose $N_A = N_B$ for all $A,B$.
\end{itemize}
Define the $1$-forms
\begin{align*}
\bm{\theta}^1 & = \bm{\kappa} := 2^{-\frac{1}{2}} \left( F^{-\frac{1}{2}} \dd r + F^{\frac{1}{2}} \, \left( \dd t + \bm{A} \right) \right) \, ,
& \tilde{\bm{\theta}}_1 & = \bm{\lambda} := 2^{-\frac{1}{2}} \left( F^{-\frac{1}{2}} \dd r - F^{\frac{1}{2}} \, \left( \dd t + \bm{A} \right) \right) \, , \\
\bm{\theta}^A & := (r^2 + N_A^2 )^{\frac{1}{2}} \bm{\mu}^A \, , & \bm{\theta}^0 & := r \bm{\mu}^0 \, .
\end{align*}
Then the connection $1$-form for the metric \eqref{eq-TaubNUT} takes the form
\begin{align*}
\bm{\Gamma} \ind{_1^1} & = 2^{-\frac{3}{2}} F^{-\frac{1}{2}} F' \left( \bm{\kappa} - \bm{\lambda} \right) \, , \\
\bm{\Gamma} \ind{_A^1} & = 2^{-\frac{1}{2}} \frac{- r + \ii N_A} {r^2 + N_A^2} F^{\frac{1}{2}} \bar{\bm{\theta}}_A \, , & \bm{\Gamma} \ind{_A_1} & = - 2^{-\frac{1}{2}} \frac{r + \ii N_A}{r^2 + N_A^2} F^{\frac{1}{2}} \bar{\bm{\theta}}_A \, , & \bm{\Gamma} \ind{_1^0} & = 2^{-\frac{1}{2}} \frac{F^{\frac{1}{2}}}{r} \bm{\theta}^0 = \bm{\Gamma} \ind{^{10}} \\
\bm{\Gamma} \ind{_A^B} & = \bm{\alpha} \ind{_A^B} - 2^{-\frac{1}{2}} \ii \frac{N_A}{r^2+N_A^2} F^{\frac{1}{2}} \left( \bm{\kappa} - \bm{\lambda} \right) \delta \ind*{_A^B} \, ,
\end{align*}
from which we immediately see that $(\mcM,g)$ is endowed with $2^m$ local co-integrable Robinson structures annihilated by the set of $(m+\epsilon)$ $1$-forms
\begin{align*}
& \{ \bm{\kappa} \, , \bm{\theta}^A \, , \bar{\bm{\theta}}_B \, , \epsilon \bm{\theta}^0 \}_{A \neq B} \, , &
& \{ \bm{\lambda} \, , \bm{\theta}^A \, , \bar{\bm{\theta}}_B \, , \epsilon \bm{\theta}^0 \}_{A \neq B} \, ,
\end{align*}
Computing the curvature $2$-form yields
\begin{align*}
\bm{R} \ind{_1^1} & = - 2^{-1} F'' \bm{\kappa} \wedge \bm{\lambda} + \ii \sum_C \frac{N_C}{ ( F' (r^2+N_C^2) - 2 r F)}{(r^2+N_C^2)^2} \bm{\theta}^C \wedge \bar{\bm{\theta}}_C \, , \\
\bm{R} \ind{_A^1} & = \frac{- \left( F' r (r^2+N_A^2) + 2 F N_A^2 \right)  +  \ii N_A \left( F' (r^2+N_A^2) - 2 F r \right)}{2(r^2+N_A^2)^2} \bm{\kappa} \wedge \bar{\bm{\theta}}_A \, , \\
\bm{R} \ind{_1^0} & = \frac{F'}{2r} \bm{\lambda} \wedge \bm{\theta}^0 \, , \\
\bm{R} \ind{_A^B} & = \underline{\bm{R}} \ind{_A^B} +  \left( \ii \frac{N_A ( F' (r^2+N_A^2) - 2 r F)}{(r^2+N_A^2)^2} \bm{\kappa} \wedge \bm{\lambda} + 2 \frac{N_A F}{r^2+N_A^2} \sum_C \frac{N_C}{r^2+N_C^2} \bm{\theta}^C \wedge \bar{\bm{\theta}}_C \right) \delta \ind*{_A^B} \\
& \qquad \qquad \qquad \qquad \qquad \qquad + \frac{F (N_A N_B - r^2)}{(r^2+N_A^2)(r^2+N_B^2)} \bm{\theta}^B \wedge \bar{\bm{\theta}}_A \, , \\
\bm{R} \ind{_A^0} & = \frac{F}{r^2 + N_A^2} \bar{\bm{\theta}}_A \wedge \bm{\theta}^0 \, .
\end{align*}
Einstein metrics can be obtained for suitable choice of function $F=F(r)$ as given in \cites{Mann2004,Awad2002,Awad2006}, in which case it is immediate that the Weyl tensor satisfies \eqref{eq-alg-special}.
\end{exa}

\paragraph{Parallel Robinson structures}
The holonomy of the Levi-Civita of a spacetime admitting a parallel null line distribution is contained in $\Sim(n-2)$. When the holonomy of the Levi-Civita connection is reduced to a subgroup of $\Sim(m-1,\C)$, the manifold admits a parallel Robinson structure, i.e.
\begin{align}\label{eq-par-Robinson}
(\nabla_a X^b ) Y_b & = 0 \, , & & \mbox{for all $X^a \in \Gamma (\mcN), Y^a \in \Gamma (\mcN^\perp)$}
\end{align}
Equivalently, the manifold admits a recurrent pure spinor of real index $1$.\footnote{As explained in appendix \ref{sec-spinors}, a pure spinor of real index $1$ is equivalent to the existence of $\simalg(m-1,\C)$-invariant $3$-form $\varrho_{abc}$ and $2$-form $\mu_{ab}$. Recurrence of the spinor is equivalent to recurrence of these forms.} This is one of the cases considered in \cite{Galaev2013}.

We now examine the integrability condition for the existence of such a parallel Robinson structure.

\begin{prop}\label{prop-int-parallel_Rob}
Let $(\mcM,g)$ be a Lorentzian manifold equipped with a parallel Robinson structure $\mcN$ with associated null line distribution $\mcK$. Then the Weyl tensor satisfies conditions \eqref{eq-GS-condition} together with the additional conditions
\begin{subequations}
\begin{align}
{}^\mfC _\mcK \Pi_0^1 (C) & = 0 \, , \label{eq-trivial} \\
{}^\mfC _\mcN \Pi_0^{3,4} (C) & = 0 \, , \label{eq-easy-cond0} \\
{}^\mfC _\mcN \Pi_1^{1,2} (C) & = 0 \, . \label{eq-easy-cond1}
\end{align}
\end{subequations}
Further, the tracefree Ricci tensor satisfies
\begin{align}
{}^\mfF _\mcN \Pi_0^{1,i} (\Phi) & = 0 \, , & & \mbox{for $i=1,3$.} \label{eq-Rob-par-RicTF}
\end{align}

Finally:
\begin{enumerate}
\item \label{item-cond1} If any two of the following conditions
\begin{align*}
 {}^\mfC _\mcN \Pi^{0,0}_0 (C) & = 0 \, , & {}^\mfC _\mcN \Pi^{3,0}_0 (C) & = 0 \, , & {}^\mfF _\mcN \Pi^{0,0}_0  (\Phi) & = 0 \, , & R = 0 \, ,
\end{align*}
hold, then the remaining two hold too. In addition, in odd dimensions,
\begin{align*}
{}^\mfC _\mcN \Pi^{2,2}_0 (C) & = 0 & \Longleftrightarrow & & {}^\mfC _\mcN \Pi^{3,0}_0 (C) & = 0 & \Longleftrightarrow & & {}^\mfF _\mcN \Pi^{1,2}_0 (\Phi) & = 0 \, .
\end{align*}
\item \label{item-cond2} We have
\begin{align*}
 {}^\mfC _\mcN \Pi^{2,0}_0 (C) & = 0 & \Longleftrightarrow & & {}^\mfC _\mcN \Pi^{3,2}_0 (C) & = 0 & \Longleftrightarrow & & {}^\mfF _\mcN \Pi^{1,0}_0 (\Phi) & = 0 \, ,
\end{align*}
and in addition, in odd dimensions,
\begin{align*}
{}^\mfC _\mcN \Pi^{3,8}_0 (C) & = 0 & \Longleftrightarrow & & {}^\mfF _\mcN \Pi^{2,0}_0 (\Phi) & = 0 \, .
\end{align*}
\item \label{item-cond3}
If any two of the following conditions
\begin{align*}
 {}^\mfC _\mcN \Pi^{0,0}_1 (C) & = 0 \, , & {}^\mfC _\mcN \Pi^{1,0}_1 (C) & = 0 \, , & {}^\mfF _\mcN \Pi^{0,0}_1 (\Phi) & = 0 \, ,
\end{align*}
hold, then the remaining one holds too.
In addition, in odd dimensions,
\begin{align*}
{}^\mfC _\mcN \Pi^{1,0}_1 (C) & = 0 & \Longleftrightarrow & & {}^\mfC _\mcN \Pi^{1,5}_1 (C) & = 0 \, .
\end{align*}
\end{enumerate}
\end{prop}

\begin{proof}
 Taking another covariant derivative of \eqref{eq-par-Robinson} to compute the Riemann curvature tensor, we find
\begin{align}\label{eq-Riem-int_cond_rob_par}
0 & = R \ind{_{abcd}} X^c Y^d = C \ind{_{abcd}} X^c Y^d +\frac{2}{n-2} \left( X^c \Phi \ind{_{c[a}} Y \ind{_{b]}} - Y^c \Phi \ind{_{c[a}} X \ind{_{b]}} \right) + \frac{2}{n(n-1)} R X \ind{_{[a}} Y \ind{_{b]}} \, ,
\end{align}
for all $X^a \in \Gamma (\mcN)$, $Y^a \in \Gamma (\mcN^\perp)$. By the results of \cites{Taghavi-Chabert2012a,Taghavi-Chabert2013}, this implies $C \ind{_{abcd}} X^c Y^d Z^b = 0$ and $\Phi_{ab} X^a Y^b = 0$ for all $X^a, Z^b \in \Gamma (\mcN)$, $Y^a \in \Gamma (\mcN^\perp)$. 
The first of these conditions yields the `complex Goldberg-Sachs' conditions of Proposition \ref{prop-alg-sp}, while the second of these gives the additional algebraic constraints
\eqref{eq-Rob-par-RicTF}. With reference to the proof of Proposition \ref{prop-rec-null-vec}, we reobtain ${}^\mfC _\mcN \Pi_0^{2,i} ( C ) = 0$ for $i=1,3$, as follows from Proposition \ref{prop-alg-sp}.

On the other hand, since the associated real null distribution is also parallel, Proposition \ref{prop-rec-null-vec} applies, i.e. \eqref{eq-trivial} holds.

Introducing a splitting \eqref{eq-splitting-tangent} of the tangent bundle adapted to $\mcN$, we now consider the remaining conditions implied by \eqref{eq-Riem-int_cond_rob_par}, more precisely,
\begin{align}
\begin{aligned}\label{eq-Riem_par-Rob}
R_{abcd} X^a Y^b \bar{V}^c \bar{W}^d = R_{abcd} k^a Y^b \ell^c \bar{V}^d = R_{abcd} X^a Y^b \ell^c \bar{V}^d & = 0 \\
R_{abcd} X^a u^b \bar{V}^c u^d = R_{abcd} X^a u^b \bar{W}^c u^d = R_{abcd} k^a u^b \ell^c u^d  & = 0 \, , & \qquad \mbox{($\epsilon=1$ only),}
\end{aligned}
\end{align}
for all $k^a \in \Gamma(\mcK)$, $X^a,Y^a,V^a,W^a \in \Gamma(\Tgt^{(1,0)})$, $u^a \in\Gamma (\Tgt^{(0,0)})$ and $\ell^a \in \Gamma(\mcL)$.
These conditions impose algebraic relations between irreducible components of the Weyl tensor and the tracefree Ricci tensor, and the Ricci scalar, which lie in isotypic irreducible $\cu(m-1)$-modules. We proceed by dimensions.
Since the modules $\breve{\mfC}_0^{3,4}$ and $\breve{\mfC}_1^{1,2}$ of respective dimensions $\frac{1}{4}m(m-1)^2(m-4)$ and $m(m-1)(m-3)$ are not isotypic to any other, we can safely conclude \eqref{eq-easy-cond0} and \eqref{eq-easy-cond1}. Now, with reference to equations \eqref{eq-Riem_par-Rob} and appendix \ref{sec-spinor-rep}, we prove the additional conditions \ref{item-cond1}, \ref{item-cond2} and \ref{item-cond3} by noting:
\begin{enumerate}
\item first,
\begin{align*}
0 & = 6 \, \Psi^{0,0}_0 + \frac{n-4}{n-2} \Phi^{0,0}_0 + \frac{n-2}{n(n-1)} R \, , &
0 & = 6 \frac{n-2}{n-3} \Psi^{0,0}_0 - \frac{m}{2(m-\epsilon)-1} \Psi^{3,0}_0 + \frac{n}{n-2} \Phi^{0,0}_0 \, ,
\end{align*}
and additionally, in odd dimensions,
\begin{align*}
\frac{2(2m-1)}{(2m-2)(n-4)} \Phi^{1,2}_0 & = - \frac{2(2m-1)(n-2)}{(n-4)(2m-2)} \Psi^{2,2}_0 = \frac{2m}{2m-3} \Psi^{3,0}_0 \, ,
\end{align*}
where $\Psi_0^{0,0}$, $\Psi_0^{2,2}$, $\Psi_0^{3,0}$, $\Phi_0^{0,0}$ and $\Phi_0^{1,2}$ are the components for the $1$-dimensional modules $\breve{\mfC}_0^{0,0}$, $\breve{\mfC}_0^{2,2}$, $\breve{\mfC}_0^{3,0}$, $\breve{\mfF}_0^{0,0}$, $\breve{\mfF}_0^{1,2}$ respectively;
\item further,
\begin{align*}
0 & = (\Psi^{2,0}_0) \ind{_B^A} + \frac{2}{n-2} (\Phi^{1,0}_0) \ind{_B^A} \, , &
0 & = - \left(\frac{n-2}{n-4} \right) \left(\frac{2m-3}{2m-4} \right) (\Psi^{2,0}_0) \ind{_B^A} + \frac{2m-10}{2m-4} (\Psi^{3,2}_0) \ind{_B^A} \, , 
\end{align*}
and additionally, in odd dimensions,
\begin{align*}
0 & = \frac{2}{n-2} (\Psi^{2,0}_0) \ind{_B^A} + (\Psi^{3,8}_0) \ind{_B^A} \, ,
\end{align*}
where $(\Psi_0^{3,2}) \ind{_B^A}$, $(\Psi_0^{3,8}) \ind{_B^A}$, $(\Psi_0^{2,0}) \ind{_B^A}$ and $(\Phi_0^{1,0}) \ind{_B^A}$ are the components for $m(m-2)$-dimensional modules $\breve{\mfC}_0^{3,2}$, $\breve{\mfC}_0^{3,8}$, $\breve{\mfC}_0^{2,0}$ and $\breve{\mfF}_0^{1,0}$ respectively;
\item finally,
\begin{align*}
0 & = \frac{1}{n-3} (\Psi^{0,0}_1)^A + \frac{2m-6}{2(2m-4)} (\Psi^{1,0}_1)^A - \frac{1}{n-2} (\Phi^{0,0}_0)^A \, ,
\end{align*}
and additionally, in odd dimensions
\begin{align*}
0 & = \frac{2m-6}{2(2m-4)} (\Psi^{1,0}_1)^A - (\Psi^{1,5}_1)^A \, ,
\end{align*}
where $(\Psi_1^{0,0})^A$, $(\Psi_1^{1,0})^A$, $(\Psi_1^{1,5})^A$ and $(\Phi_1^{0,0})^A$ are the components for $(2m-2)$-dimensional modules $\breve{\mfC}_1^{0,0}$, $\breve{\mfC}_1^{1,0}$, $\breve{\mfC}_1^{1,5}$ and $\breve{\mfF}_1^{0,0}$ respectively.
\end{enumerate}
This ends the proof of Proposition \ref{prop-int-parallel_Rob}.
\end{proof}

\paragraph{Parallel pure spinor fields of real index $1$}
From its defining properties, a pure spinor field of real index $1$ gives rise to an almost Robinson structure $\mcN$. When such a spinor is parallel, one obtains a counterpart of Proposition \ref{prop-par-null-vec} in the context of almost Robinson geometry. For an account of metrics admitting a parallel pure spinor, see for instance \cites{Bryant2000,Leistner2002}.
\begin{prop}\label{prop-parallel-pure-spinor}
Let $(\mcM,g)$ be a Lorentzian manifold admitting a parallel pure spinor field $\bm{\xi}$ of real index $1$. Let $\mcN$ be the almost Robinson structure annihilating $\bm{\xi}$ with associated null line distribution $\mcK$. Then
\begin{subequations}
\begin{align}
{}^\mfC _\mcK \Pi_1^0 (C) & = 0 \, , & & \mbox{i.e. \quad $C \ind{_{abcd}} k \ind{^d} = 0$} \, , \label{eq-par-spin-sim} \\
{}^\mfC _\mcN \Pi_0^i (C) & = 0 \, , & & \mbox{for $i=0,1,2$,} \label{eq-par-spin-rob} \\
{}^\mfC _\mcN \Pi_0^{3,i} (C) & = 0 \, , & & \mbox{for all $i\neq5$,} \label{eq-par-spin-rob2} \\
{}^\mfC _\mcN \Pi_1^{1,i} (C) & = 0 \, , & & \mbox{for all $i\neq3,7,8$,} \label{eq-par-spin-rob3} \\
{}^\mfF _\mcK \Pi_1^0 (\Phi) & = 0 \, , & & \mbox{i.e. \quad $\Phi \ind{_{a[b}} k \ind{_{c]}} = 0$} \, , \label{eq-par-spin-sim2} \\
 R & = 0 \, . \label{eq-par-spin-sim3} 
\end{align}
\end{subequations}
\end{prop}

\begin{proof}
For simplicity, we treat the case $n=2m+1$ only, the even-dimensional case being similar. We first note that since the pure spinor field is parallel, so must be its conjugate and its associated null vector $k^a$. In particular, Proposition \ref{prop-par-null-vec} holds. Using the notation of appendix \ref{sec-spinors}, we know (see \cites{Taghavi-Chabert2012a,Taghavi-Chabert2013} and references therein) that the existence of a parallel pure spinor $\xi^A$ implies $R = 0$, i.e. \eqref{eq-par-spin-sim3}, and $\Phi_{ab} \gamma \ind{^a_A^B} \xi^A = 0$. The latter condition leads to \eqref{eq-par-spin-sim2}, so that by Proposition \ref{prop-par-null-vec}, we have that \eqref{eq-par-spin-sim}.

Moreover, since $\mcN$ is parallel, Proposition \ref{prop-int-parallel_Rob} holds with our additional constraints \eqref{eq-par-spin-sim}, \eqref{eq-par-spin-sim2}, and \eqref{eq-par-spin-sim2}, i.e. $C_{abcd} X^c Y^d = 0$ for all $X^a \in \Gamma (\mcN)$ and $Y^a \in \Gamma (\mcN)$. The extra constraints can be found to be ${}^\mfC_\mcN \Pi_0^{3,i} (C) = 0$ for $i=0,2,8$ and ${}^\mfC_\mcN \Pi_1^{1,i} (C) = 0$ for $i=0,5$.

Finally, the additional integrability condition $C_{abcd} \gamma \ind{^{cd}_A^B} \xi^A = 0$ can be re-expressed as $C_{abcd} \varrho \ind{^{cd}_e} = 0$ and $C_{abcd} \mu \ind{^{cd}} = 0$ , where $\varrho_{abc} := 3 \, k_{[a} \omega_{bc]}$ and $\mu_{ab} = 2 k_{[a} u_{b]}$, $\omega_{ab}$ is any Hermitian structure and $u^a$ a unit vector on the screenspace of $\mcK$. Thus, those components of $C_{abcd}$ which have non-vanishing trace with respect to any choice $\omega_{ab}$ will vanish. The only additional condition on the Weyl tensor is then ${}^\mfC _\mcN \Pi_1^{1,4}(C) = 0$, which can happen only in odd dimensions. This completes the proof.
\end{proof}

\paragraph{Acknowledgements} I would like to thank Thomas Leistner for interesting discussions and for his hospitality during my stay at the University of Adelaide in November 2012. I am also grateful to Dimitri Alekseevsky and Anton Galaev for helpful discussions. Finally, I extend my thanks to Lionel Mason, whose ideas came to influence the content of this article.

This work was funded by a SoMoPro (South Moravian Programme) Fellowship: it has received a financial contribution from the European
Union within the Seventh Framework Programme (FP/2007-2013) under Grant Agreement No. 229603, and is also co-financed by the South Moravian Region.

The author has also benefited from an Eduard \v{C}ech Institute postdoctoral fellowship GPB201/12/G028, and a GA\v{C}R (Czech Science Foundation) postdoctoral grant GP14-27885P.

\appendix
\section{Pure spinors and Robinson structures}\label{sec-spinors}
The aim of this appendix is to give an explicit description of a Robinson structure building on the spinor calculus of references \cites{Taghavi-Chabert2012a,Taghavi-Chabert2013}. Spinor representations for a real inner product space $(\mfV,g)$ are complex vector spaces equipped with additional reality or quaternionic structures depending on the dimension of $\mfV$, the signature $g$, and the convention adopted. There is a wealth of literature on the details and subtleties involved in their description \cites{Cartan1981,Penrose1986,Budinich1989,Kopczy'nski1992,Kopczy'nski1997,Lawson1989}.

Our treatment will however be overwhelmingly dimension-independent, and, more in the spirit of Cartan's theory of spinors \cite{Cartan1981}, we shall exploit the \emph{geometric} properties of \emph{pure} spinors to construct our calculus. In fact, many of the algebraic properties of Clifford and spinor representations can be derived from such geometric considerations. In this way, our spinor calculus will merely boil down to a tensor calculus adapted to a given Robinson structure.  Alternative spinorial approaches to five-dimensional spacetimes can be found in \cites{De2005,Godazgar2010} and \cite{Garc'ia-ParradoG'omez-Lobo2009}.

Throughout, $(\mfV,g)$ will denote $n$-dimensional Minkowski space, and as before, the signature  of $g$ will be taken to be $(n-1,1)$. We shall work in the complexification $({}^\C \mfV , g)$ of $(\mfV,g)$, in which case the metric tensor will be assumed to be complex-valued. The spinor representations for $(\mfV,g)$ will then be obtained by additional algebraic restrictions. We first treat the even- and odd-dimensional cases separately for clarity. Our setup and notation will allow us to dispense with such a distinction later.

\subsection{Background algebra}
\subsubsection{Even dimensions}
Assume $n=2m$. The spinor representation $\mfS$ of $({}^\C \mfV, g)$ splits into a direct sum of two $2^{m-1}$-dimensional irreducible chiral spin spaces $\mfS^+$ and $\mfS^-$, the spaces of positive and negative spinors respectively. Elements of $\mfS^+$ and $\mfS^-$ will be adorned with upstairs primed and unprimed upper case Roman indices respectively, e.g. $\alpha^{A'} \in \mfS^+$ and $\beta^A \in \mfS^-$, and similarly for the dual spinor spaces $(\mfS^\pm)^*$ with downstairs indices. The relation between $({}^\C \mfV,g)$ and $\mfS$ is given by the Clifford algebra $\Cl({}^\C \mfV,g)$ of $(\mfV,g)$, which establishes an isomorphism between the exterior algebra $\wedge^\bullet \mfV$ and the space $\End(\mfS)$ of endomorphisms of $\mfS$. We introduce the van der Waerden $\gamma$-matrices $\gamma \ind{_a_A^{B'}}$ and $\gamma \ind{_a_{A'}^B}$ which satisfy\footnote{Our convention differs from \cites{Taghavi-Chabert2012a} by a sign.}
\begin{align*}
 \gamma \ind{_{(a}_{A'}^C} \gamma \ind{_{b)}_C^{B'}} & = g \ind{_{ab}} \delta \ind*{_{A'}^{B'}} \, , & \gamma \ind{_{(a}_A^{C'}} \gamma \ind{_{b)}_{C'}^B} & = g \ind{_{ab}} \delta \ind*{_A^B} \, .
\end{align*}
We emphasise that we shall use the van der Waerden $\gamma$-matrices in a \emph{purely abstract way}, i.e. in the sense that we shall think of $\gamma \ind{_a_{A'}^B}$ and $\gamma \ind{_a_A^{B'}}$ as injective maps from $\mfV$ to the space of homomorphisms from $\mfS^\mp$ to $\mfS^\pm$, or as a projection from $(\mfS^\mp)^*\otimes \mfS^\pm$ to $\mfV$. 

The real structure $\bar{}$ on ${}^\C \mfV$ preserving $\mfV$ induces a real structure, when $m=1,2 \pmod 4$, and a quaternionic structure when $m=0,3 \pmod 4$ on $\mfS$. Thus, the spinor representation $\mfS$ of $(\mfV,g)$ will be that of $({}^\C \mfV,g)$ together with this additional structure. Further, this complex conjugation interchanges the chirality of spinors when $m$ is even, and preserve them when $m$ is odd. One can then use the isomorphisms $\mfS^\pm \cong (\mfS^\pm)^*$ when $m$ is even, and $\mfS^\pm \cong (\mfS^\pm)^*$ when $m$ is odd, to define a complex conjugation $\bar{}: \mfS^{\pm} \rightarrow (\mfS^{\mp})^*$, i.e. $\bar{}:\xi \ind{^{A'}} \mapsto \bar{\xi} \ind{_A}$, and $\bar{}: \eta \ind{^A} \mapsto \bar{\eta} \ind{_{A'}}$.

\paragraph{The pure spinors associated to a Robinson structure}
By definition, a Robinson structure on $(\mfV,g)$ is a totally null complex $m$-plane $\mfN$ in $(\mfV,g)$, which, for definiteness, we shall assume to be self-dual. Then $\mfN$ is the annihilator of a \emph{pure} positive spinor $\xi \ind{^{A'}}$ in the sense that $\mfN = \ker \xi_a^A : {}^\C \mfV \rightarrow \mfS^-$,
where we have written $\xi_a^A := \xi \ind{^{B'}} \gamma \ind{_a_{B'}^A}$, and $\xi^{A'}$ satisfies the algebraic condition $\xi \ind{^a^A} \xi_a^B = 0$. Similarly, the conjugate $m$-plane $\bar{\mfN}$ annihilates the conjugate pure spinor $\bar{\xi} \ind{_A}$. We note that the chirality of $\bar{\xi} \ind{_A}$ is consistent with the orientation of $\bar{\mfN}$, i.e. a $\beta$-plane when $m$ is even, and an $\alpha$-plane when $m$ is odd.

Any element $V^a$ of $\mfN$ is of the form $V \ind{^a} = \xi \ind{^a^A} v \ind{_A}$
for some $v_A \in (\mfS^-)^*/\{\ker \xi_a^A : {}^\C \mfV \leftarrow (\mfS^-)^*\}$. The conjugate spinor $\bar{\xi}_A$ is the distinguished element of $(\mfS^-)^*$ for which the real null vector
\begin{align}\label{eq-real-null-vec-spinor}
 k^a & = \xi^{aB} \bar{\xi}_B \, ,
\end{align}
spans $\mfK$, the real span of the intersection of $\mfN$ and $\bar{\mfN}$.

\paragraph{A dual Robinson structure and its pure spinors}
We now introduce a Robinson structure $\mfN^*$ dual to $\mfN$. Here, the totally null $m$-plane $\mfN^*$, its complex conjugate $\bar{\mfN}^*$ and their real intersection $\mfL$ are dual to $\mfN$, $\bar{\mfN}$ and $\mfK$ respectively. The $m$-plane $\mfN^*$ annihilates a pure spinor $\bar{\eta}_{A'}$ dual to $\xi^{A'}$, i.e. $\mfN^* = \ker \bar{\eta} \ind{_a_A} : {}^\C \mfV \rightarrow (\mfS^-)^*$, where $\bar{\eta} \ind{_a_A} := \gamma \ind{_a_A^{B'}} \bar{\eta} \ind{_{B'}}$. Similarly, to $\bar{\mfN}^*$ corresponds a pure spinor $\eta^A$ dual to $\bar{\xi}_A$ and conjugate of $\bar{\eta}_{A'}$. We can choose to normalise these spinors as $\xi^{A'} \bar{\eta}_{A'} = \frac{1}{2}$ and $\bar{\xi}_A \eta^A = - \frac{1}{2}$. 
Any generator of $\mfL$ is a real multiple of 
\begin{align}\label{eq-real-null-vec-spinor-dual}
\ell^a := 2\,  \eta  \ind{^A} \bar{\eta}^a_A \, .
\end{align}
 
We can now set $\mfV_1 = \mfK$, $\mfV_{-1} = \mfL$ and $\mfV_0 = \mfK^\perp \cap \mfL^\perp$, and refer to the splitting \eqref{eq-sim-grading} of the filtration \eqref{eq-sim-filtration}. The spinor $I \ind*{_B^A} := \xi \ind{^a^A} \bar{\eta} \ind{_a_B} : \mfS^- \rightarrow  \mfS^-$ is the identity map on $\im \xi \ind{^a^A}$, or dually on $\im \bar{\eta} \ind{_a_B}$. Now, since $\mfV_0^{(1,0)} = \mfN \cap \bar{\mfN}^*$, we have $\xi \ind{^a^A} \eta \ind*{_a^{B'}} = 2 \, \xi^{B'} \eta^A$ so that $\bar{\xi}_B I \ind*{_A^B} = \bar{\xi}_A$ and $\eta^B I \ind*{_B^A} = \eta^A$. It follows that $\im \xi \ind{^a^A}$ and $\im \bar{\eta}  \ind*{^a_A}$ contain $\langle \bar{\xi}_A \rangle$ and $\langle \eta^A \rangle$ respectively. Thus, the map defined by $\iota \ind*{_B^A} := I \ind*{_B^A} - 2 \, \bar{\xi}_B \eta^A$ can be identified with the identity map on the $(m-1)$-dimensional subspaces
\begin{align*}
 \mfS_{\frac{1}{2}}^{(0,1)} & := \{ \im \xi \ind*{_a^A} : {}^\C \mfV \rightarrow \mfS^- \} \cap \{ \ker \bar{\xi} \ind{_A} : \C \leftarrow \mfS^- \} \, , \\
 \mfS_{-\frac{1}{2}}^{(1,0)} & := \{ \im \bar{\eta} \ind{_a_A} : {}^\C \mfV \rightarrow (\mfS^-)^* \} \cap \{ \ker \eta \ind{^A} : \C \leftarrow (\mfS^-)^*  \} \, .
\end{align*}
At this stage, we set $\mfS_{\frac{1}{2}} := \langle \xi^{A'} \rangle$ and $\bar{\mfS}_{\frac{1}{2}} := \langle \bar{\xi}_A \rangle$. This notation is consistent with the fact that $\mfV_1 :=  \mfK$ injects into the tensor product $\mfS_{\frac{1}{2}} \otimes \bar{\mfS}_{\frac{1}{2}}$. In fact, both $\xi^{A'}$ and $\bar{\xi}_A$ are eigenspinors of the grading element of the Lie algebra \eqref{eq-grading-sim} for the eigenvalue $\frac{1}{2}$. Putting things together, we have the identifications
\begin{align}\label{eq-vector10-spinor10-even}
\begin{aligned}
 \mfN & = {}^\C \mfV_1 \oplus \mfV_0^{(1,0)} \cong \mfS_{\frac{1}{2}} \otimes \left( \bar{\mfS}_{\frac{1}{2}} \oplus \mfS_{-\frac{1}{2}} ^{(1,0)} \right) \, ,  \\
 \mfN^* & = {}^\C \mfV_{-1} \oplus \mfV_0^{(0,1)} \cong \bar{\mfS}_{-\frac{1}{2}} \otimes \left( \mfS_{-\frac{1}{2}} \oplus \mfS_{\frac{1}{2}} ^{(0,1)} \right) \, ,
 \end{aligned}
\end{align}
so that any elements of $\mfV_0^{(1,0)}$ and $\mfV_0^{(0,1)}$ admit spinorial representatives in $\mfS_{\frac{1}{2}}^{(1,0)}$ and $\mfS_{-\frac{1}{2}}^{(0,1)}$ respectively.

The complex conjugation on ${}^\C \mfV_0$ interchanges $\mfV_0^{(1,0)}$ and $\mfV_0^{(0,1)}$. It must therefore extends to an involution on $\mfS$ which interchanges $\mfS_{\frac{1}{2}}^{(1,0)}$ and $\mfS_{-\frac{1}{2}}^{(0,1)}$. This new complex conjugation, which we shall denote $\check{}$, must depend on the complex conjugation $\bar{}$ on $\mfS$ and a choice of splitting.\footnote{Recall that $\mfV_0$ is equipped with a positive definite metric tensor $h \ind{_{ab}}$, and we can identify $\mfS^{(1,0)}_{\frac{1}{2}}$ and $\mfS^{(0,1)}_{-\frac{1}{2}}$ as subspaces of a chiral spinor representation for $(\mfV_0,h)$ and its dual respectively. The complex conjugation $\hat{}$ can then be identified with the canonical complex conjugation on the spinor representation of $(\mfV_0,h)$}

\subsubsection{Odd dimensions}
Now assume $n=2m+1$. The irreducible spinor representation $\mfS$ of $(\mfV,g)$ is a $2^m$-dimensional complex vector space. Elements of $\mfS$ will be adorned with upstairs upper case Roman indices, e.g. $\alpha^A \in \mfS$, and similarly for the dual spinor space $\mfS^*$ with downstairs indices. The relation between $({}^\C \mfV,g)$ and $\mfS$ is given by the Clifford algebra $\Cl({}^\C \mfV,g)$ of $(\mfV,g)$, which establishes an isomorphism between the exterior algebra $\wedge^\bullet \mfV$ and the space $\End(\mfS)$ of endomorphisms of $\mfS$. We introduce the van der Waerden $\gamma$-matrices $\gamma \ind{_a_A^B}$ which satisfy\footnote{Again, our convention differs from \cites{Taghavi-Chabert2013} by a sign.}
\begin{align*}
 \gamma \ind{_{(a}_A^C} \gamma \ind{_{b)}_C^B} & = g \ind{_{ab}} \delta \ind*{_A^B} \, ,
\end{align*}
As in the even-dimensional case, these $\gamma$-matrices will be thought of as abstract maps.

The real structure $\bar{}$ on ${}^\C \mfV$ preserving $\mfV$ induces a real structure, when $m=0,1 \pmod 4$, and a quaternionic structure, when $m=2,3 \pmod 4$, on $\mfS$. Using the isomorphism $\mfS \cong \mfS^*$, we can define a complex conjugation, denoted by $\bar{}$, that sends elements of $\mfS$ to elements of $\mfS^*$, i.e. $\bar{}:\xi \ind{^A} \mapsto \bar{\xi} \ind{_A}$.

\paragraph{The pure spinors associated to a Robinson structure}
We associate to a Robinson structure $\mfN$ a pure spinors $\xi^{A}$, which is annihilated by the complex null $m$-plane $\mcN$, i.e. $ \mfN = \ker \xi \ind*{_a^A} : {}^\C \mfV \rightarrow \mfS$, where $\xi \ind*{_a^A} := \xi \ind{^{B'}} \gamma \ind{_a_{B'}^A}$. Here, the pure spinor $\xi^A$ satisfies the algebraic constraint $\xi \ind{^a^A} \xi \ind*{_a^B} = \xi \ind{^A} \xi \ind{^B}$. Similarly, $\bar{\mfN}$ annihilates a pure spinor $\bar{\xi}_A$ conjugate to $\xi^A$. Moreover, the real index of $\mfN$ imposes an additional algebraic constraints on these pure spinors, i.e.
\begin{align*}
\xi^A \bar{\xi}_A & = 0 \, ,
\end{align*}
which is the algebraic condition for $\mfN$ and $\bar{\mfN}$ to intersect in ${}^\C \mfK$. The real span $\mfK$ of this intersection is in fact generator by the element
\begin{align}\label{eq-real-null-vec-spinor-odd}
 k^a & = \xi \ind{^a^B} \bar{\xi}_B \, ,
\end{align}
as in even dimensions.

More generally, any element $V^a$ of $\mfN^\perp$ is of the form $V \ind{^a} = \xi \ind{^a^A} v \ind{_A}$
for some $v_A \in (\mfS^-)^*/\{\ker \xi \ind*{_a^A} : {}^\C \mfV \leftarrow (\mfS^-)^*\}$. We then have that $V^a$ belongs to $\mfN$ if and only if $v_A \xi^A = 0$.

\paragraph{A dual Robinson structure and its pure spinor}
Again, we choose a Robinson structure $\mfN^*$, dual to $\mfN$, to which we associate a pure spinor $\bar{\eta}_A$ dual to $\xi^A$, i.e. $\mfN^* = \ker \bar{\eta} \ind{_a_A} : {}^\C \mfV \rightarrow \mfS^*$, where $\bar{\eta} \ind{_a_A} := \gamma \ind{_a_A^B} \bar{\eta} \ind{_B}$. Similarly, $\bar{\mfN}^*$ annilates the complex conjugate pure spinor $\eta^A$ dual to $\bar{\xi}_A$. The spinors $\bar{\eta}_A$ and $\eta^A$ are of real index $1$, and as such satisfy $\bar{\eta}_A \eta^A = 0$, and the real intersection of $\mfN^*$ and $\bar{\mfN}^*$ is spanned by the real element
\begin{align}\label{eq-real-null-vec-spinor-dual-odd}
\ell \ind{^a} & = 2 \, \eta \ind{^B} \bar{\eta} \ind{^a_B}  \, .
\end{align}
We choose the normalisation $\xi^A \bar{\eta}_A = - \frac{1}{2}$ and $\bar{\xi}_A \eta^A = \frac{1}{2}$. Referring to \cites{Taghavi-Chabert2013}, we know that since $\mfN$ and $\bar{\mfN}^*$ intersect in a totally null $(m-1)$-plane, $\xi^A$ and $\eta^B$ must satisfy $\xi \ind{^a^A} \eta \ind*{_a^B} = - \xi^A \eta^B + 2 \eta^A \xi^B$. We also know that $I \ind*{_B^A} := \xi \ind{^a^A} \bar{\eta} \ind{_a_B} + \bar{\eta}_B \xi^A : \mfS \rightarrow \mfS$ is the identity map on $\im \xi \ind*{_a^A} \cap \ker \bar{\eta}_A$, or dually, $\im \bar{\eta} \ind{_a_A} \cap \ker \xi^A$. We therefore have $\bar{\xi}_B I \ind*{_A^B} = \bar{\xi}_A$ and $\bar{\xi}_B \bar{I} \ind*{_A^B} = 0$, i.e. $\im \xi \ind{^a^A}$ and $\im \bar{\eta}  \ind*{^a_A}$ contain $\langle \bar{\xi}_A \rangle$ and $\langle \eta^A \rangle$ respectively. Hence the map $\iota \ind*{_B^A} := I \ind*{_B^A} - 2 \bar{\xi}_B \eta^A$ is the identity on the $(m-1)$-dimensional vector subspaces
\begin{align*}
 \mfS_{\frac{1}{2}}^{(0,1)} & := \{ \im \xi \ind*{_a^A} : {}^\C \mfV \rightarrow \mfS \} \cap \{ \ker \bar{\eta} \ind{_A} : \C \leftarrow \mfS \} \cap \{ \ker \bar{\xi} \ind{_A} : \C \leftarrow \mfS \} \, , \\
 \mfS_{-\frac{1}{2}}^{(1,0)} & := \{ \im \bar{\eta} \ind{_a_A} : {}^\C \mfV \rightarrow \mfS^* \} \cap \{ \ker \xi \ind{^A} : \C \leftarrow \mfS^* \} \cap \{ \ker \eta \ind{^A} : \C \leftarrow \mfS^* \} \, .
\end{align*}
As in the even-dimensional, one obtains isomorphisms \eqref{eq-vector10-spinor10-even}. In addition, we have
\begin{align}
\begin{aligned}\label{eq-vector10-spinor10-odd}
 \mfN^\perp & = {}^\C \mfV_1 \oplus \mfV_0^{(1,0)} \oplus \mfV_0^{(0,0)} \cong \mfS_{\frac{1}{2}} \otimes \left( \bar{\mfS}_{\frac{1}{2}} \oplus \mfS_{-\frac{1}{2}} ^{(1,0)} \oplus \mfS_{-\frac{1}{2}} \right) \, ,  \\
 (\mfN^\perp)^* & = {}^\C \mfV_{-1} \oplus \mfV_0^{(0,1)} \oplus \mfV_0^{(0,0)} \cong \bar{\mfS}_{-\frac{1}{2}} \otimes \left( \mfS_{-\frac{1}{2}} \oplus \mfS_{\frac{1}{2}} ^{(0,1)} \oplus \mfS_{\frac{1}{2}} \right) \, .
 \end{aligned}
\end{align}
where $\mfS_{\frac{1}{2}} := \langle \xi^A \rangle$ and $\bar{\mfS}_{\frac{1}{2}} := \langle \bar{\xi}_A \rangle$.
In particular, the vector space $[\mfV_0^{(0,0)}]$ is spanned by the real unit vector
\begin{align}\label{eq-spinor-odd-unit}
u^a & := \xi \ind{^a ^B} \bar{\eta}_B + \eta \ind{^a ^B} \bar{\xi}_B \, .
\end{align}

The complex conjugation on ${}^\C \mfV_0$ interchanges $\mfV_0^{(1,0)}$ and $\mfV_0^{(0,1)}$, and must therefore induce an involution\footnote{As in the even-dimensional case, we note that $\mfV_0$ is equipped with a positive definite metric tensor $h \ind{_{ab}}$, and we can identify $\mfS^{(1,0)}_{\frac{1}{2}}$ and $\mfS^{(0,1)}_{-\frac{1}{2}}$ as being subspaces of a chiral spinor representation for $(\mfV_0,h)$ and its dual respectively.} that interchanges $\mfS_{\frac{1}{2}}^{(1,0)}$ and $\mfS_{-\frac{1}{2}}^{(0,1)}$, we shall denote this involution by $\check{}$.

\subsubsection{Robinson $2$-forms and Robinson $3$-forms}\label{sec-Robinson-forms}
The pairing of the conjugate pair of pure spinors $\xi^{A'}$ (or $\xi^A$) and $\bar{\xi}_A$ associated to a Robinson structure $\mfN$ and its complex conjugate $\bar{\mfN}$ yields $\simalg(m-1,\C)$-invariant $k$-forms as is shown in \cite{Kopczy'nski1992}. In other words, the tensor product $\mfS_{\frac{1}{2}} \otimes \bar{\mfS}_{\frac{1}{2}}$ projects into $\wedge^k \mfV$ for some $k$. When $\epsilon=0$, $k$ has to be odd, but is otherwise unrestricted. As we have already seen with \eqref{eq-real-null-vec-spinor} and \eqref{eq-real-null-vec-spinor-odd}, when $k=1$, this pairing is simply the vector $k^a$ spanning the intersection $\mfK$ of $\mfN$ and $\bar{\mfN}$. When $k=2,3$, we obtain the real $3$-form and $2$-form
\begin{align*}
\varrho_{abc} & := \ii \, \xi \ind*{_{abc} ^B} \bar{\xi}_B \, , & \mbox{when $\epsilon = 0 ,1$,}  \\
   \mu_{ab} & := \xi \ind*{_{ab} ^B} \bar{\xi}_B \, , & \mbox{when $\epsilon = 1$ only.} 
\end{align*}
Here $\xi \ind*{_{abc} ^B} = \xi^{A'} \gamma \ind{_{abc}_{A'} ^B}$ when $\epsilon=0$ and $\xi \ind*{_{abc} ^B} = \xi^A \gamma \ind{_{abc}_A ^B}$ when $\epsilon=1$. Using the purity condition, it is straightforward to check that these forms satisfy
\begin{align*}
 \varrho \ind{_e ^{ab}} \varrho \ind{^e _{cd}} & = 4 \, k \ind{^{\lb{a}}} k \ind{_{\lb{c}}} g \ind*{^{\rb{b}}_{\rb{d}}} - \epsilon \mu \ind{^{ab}} \mu \ind{_{cd}} \, ,   & \mbox{when $\epsilon = 0 ,1$,}  \\
  \mu \ind{_{ac}} \mu \ind{_b^c} & = k_a k_b \, , & \mbox{when $\epsilon = 1$ only.} 
\end{align*}
These conditions are equivalent to writing
\begin{align*}
 \varrho \ind{_{abc}} & = 3 \, k \ind{_{[a}} \omega \ind{_{bc]}} \, , & \mbox{when $\epsilon = 0 ,1$,} \\
   \mu \ind{_{ab}} & = 2 \, k \ind{_{[a}} u \ind{_{b]}}  \, , & \mbox{when $\epsilon = 1$ only,}
\end{align*}
for some Hermitian $2$-form $\omega \ind{_{ab}}$ on $\mfV_0$, and when $\epsilon=1$, some unit vector $u^a$ orthogonal to $\omega \ind{_{ab}}$. Since $\omega \ind{_{ab}} := \ell^c \varrho \ind{_{abc}}$ and $u_a := \ell^b \mu{_{ba}}$ for some choice of $\ell_a$ dual to $k^a$ with $k^a \ell_a =1$, both $\omega_{ab}$ and $u^a$ depend on the choice of splitting of $\mfK^\perp/\mfK$. On the other hand, both $\varrho_{abc}$ and $\mu_{ab}$ are independent of such a choice.

\begin{rem}
In four dimensions, one has $k \ind{_a} = \frac{1}{3!} \varepsilon \ind{_a^{bcd}} \varrho \ind{_{bcd}}$, while in five dimensions, $\mu \ind{_{ab}} = \frac{1}{2} \varepsilon \ind{_{ab}^{cde}} \varrho \ind{_{cde}}$, which result in a number of simplifications in these dimensions -- see appendix \ref{sec-low-dim}.
\end{rem}

\subsection{Tensor-spinor translation}
As we have set up our notation and convention, primed spinor indices will not appear in even dimensions, so we can effectively combine the even- and odd-dimensional cases. This is due to the fact that our primary objects here are $\xi \ind{^a^A}$ and $\bar{\eta} \ind{^a_A}$, which we shall use to convert spinorial quantities into $\cu(m-1)$-invariant tensorial ones and vice and versa with reference to \eqref{eq-vector10-spinor10-even} and \eqref{eq-vector10-spinor10-odd}. In particular, we shall always use the real vectors $k^a$, and $\ell^a$, and in odd dimensions,  $u_a$ as defined by \eqref{eq-real-null-vec-spinor}, \eqref{eq-real-null-vec-spinor-odd}, \eqref{eq-real-null-vec-spinor-dual}, \eqref{eq-real-null-vec-spinor-dual-odd} and \eqref{eq-spinor-odd-unit}. The remaining isomorphisms $\mfV_0^{(1,0)} \cong \mfS_\frac{1}{2} \otimes \mfS_{-\frac{1}{2}}^{(1,0)}$ and $\mfV_0^{(0,1)} \cong \mfS_{-\frac{1}{2}} \otimes \mfS_{\frac{1}{2}}^{(0,1)}$ allow us to realise the embedding of $\uu(m-1)$-modules into $\so(n-2)$-modules. For this reason, unprimed spinor indices will always denote memberships to the standard complex representation $\mfS_{\frac{1}{2}}^{(1,0)}$ of $\uu(m-1)$ and its dual $\mfS_{-\frac{1}{2}}^{(0,1)}$. In general, if $\phi \ind{^{a_1 \ldots a_p b_1 \ldots b_q}}$ is in the $\cu(m-1)$-module $\mfV_0^{(q,p)} := \left( {\bigotimes}^p \mfV_0^{(0,1)} \right) \otimes \left( {\bigotimes}^q \mfV_0^{(1,0)} \right)$, then we can write
\begin{align*}
\phi \ind{^{a_1 \ldots a_p b_1 \ldots b_q}} = \xi \ind{^{a_1}^{A_1}} \ldots \xi \ind{^{a_p}^{A_p}} \bar{\eta} \ind*{^{b_1}_{B_1}} \ldots \bar{\eta} \ind*{^{b_q}_{B_q}} \phi \ind{_{A_1 \ldots A_p}^{B_1 \ldots B_q}} \, ,
\end{align*}
for some $\phi \ind{_{A_1 \ldots A_p}^{B_1 \ldots B_q}}$ in $\mfS_0^{(q,p)} := \left( {\bigotimes}^p \mfS_{-\frac{1}{2}}^{(0,1)} \right) \otimes \left( {\bigotimes}^q \mfS_{\frac{1}{2}}^{(1,0)} \right)$.

The complex conjugate of $\phi \ind{^{a_1 \ldots a_p b_1 \ldots b_q}}$ is then simply given by
\begin{align*}
\bar{\phi} \ind{^{a_1 \ldots a_p b_1 \ldots b_q}} & = \bar{\eta} \ind*{^{a_1}_{A_1}} \ldots \bar{\eta} \ind*{^{a_p}_{A_p}} \xi \ind{^{a_1}^{B_1}} \ldots \xi \ind{^{a_q}^{B_q}} \check{\phi} \ind{^{A_1 \ldots A_p}_{B_1 \ldots B_q}} \, ,
\end{align*}
where $\check{\phi} \ind{^{A_1 \ldots A_p}_{B_1 \ldots B_q}}$ belongs to $\mfS_0^{(q,p)}$.

An element of some irreducible $\uu(m-1)$-submodule of $\mfV_0^{(q,p)}$ will lie in the (invariant) kernel of some projections from $\mfV_0^{(q,p)}$ to this irreducible submodule. Of particular interest is the identity element $\iota \ind*{_A^B}$ on $\mfS_{\frac{1}{2}}^{(1,0)}$ which has images
\begin{align*}
\omega_{ab} & = \ii \left( \xi \ind{_{[a}^A} \bar{\eta} \ind{_{b]}_A} - \eta \ind{_{[a}^A} \bar{\xi} \ind{_{b]}_A} \right) \, , &
H_{ab} & = \xi \ind{_{(a}^A} \bar{\eta} \ind{_{b)}_A} - \eta \ind{_{(a}^A} \bar{\xi} \ind{_{b)}_A} - 2 \, k \ind{_{(a}} \ell \ind{_{b)}} - \epsilon u \ind{_a} u \ind{_b} \, ,
\end{align*}
in $\wedge^2 \mfV_0$ and $\odot^2 \mfV_0$ respectively. It is then not too difficult to show that a tensor lying in $[\mfV_0^{(0,1)} \otimes \mfV_0^{(1,0)}]$ is tracefree with respect to either $\omega_{ab}$ and $H_{ab}$ if and only if its spinor representative in $\mfS_{-\frac{1}{2}}^{(0,1)} \otimes \mfS_{\frac{1}{2}}^{(1,0)}$ is tracefree with respect to $\iota \ind*{_A^B}$, i.e. $\phi \ind{_A^B} \iota \ind*{_B^A} = 0$.

\begin{rem}
An alternative description of the above spinorial approach of a Robinson structure is to identify the spinor module $\mfS$ with the vector space $\bigwedge^\bullet \mfN$. Then using the decomposition $\mfN = {}^\C \mfV_1 \oplus \mfV_0^{(1,0)}$, we have
\begin{align*}
\mfS & \cong {}^\C \mfV_1 \otimes \left( \bigoplus_{k=0}^{m-1} \wedge^k \mfV_0^{(1,0)} \right) \oplus \left( \bigoplus_{k=0}^{m-1} \wedge^k \mfV_0^{(1,0)} \right) \, ,
\end{align*}
and the spinors $\xi^{A'}$ and $\bar{\xi}_{A}$ can be identified with generators of ${}^\C \mfV_1 \otimes \left( \wedge^{m-1} \mfV_0^{(1,0)} \right)$ and ${}^\C \mfV_1$ respectively.

Similarly, having fixed a dual $\mfN^*$, we can identify the spinors $\bar{\eta}_{A'}$ and $\eta^{A'}$ with ${}^\C \mfV_{-1} \otimes \left( \wedge^{m-1} \mfV_0^{(0,1)} \right)$ and ${}^\C \mfV_{-1}$ respectively, so that
\begin{align*}
\mfS_{\frac{1}{2}}^{(1,0)} & \cong {}^\C \mfV_1 \otimes \left( \wedge^{m-2} \mfV_0^{(1,0)} \right) \, , &
\mfS_{\frac{1}{2}}^{(0,1)} & \cong \mfV_0^{(1,0)} \, .
\end{align*}
This applies to both even and odd dimensions.
\end{rem}
\section{Tensorial and spinorial descriptions of irreducible $\simalg(n-2)$-modules and $\simalg(m-1,\C)$-modules}\label{sec-spinor-descript}
This appendix complements sections \ref{sec-algebra} and \ref{sec-curvature} by providing a number of explicit formulae to describe the various irreducible $\simalg(n-2)$- and $\co(n-2)$-modules arising from the choice of a null line, and $\simalg(m-1,\C)$- and $\cu(m-1)$-modules arising from a Robinson structure. Throughout, the notation will follow that of sections \ref{sec-algebra} and \ref{sec-curvature}.

\subsection{Projection maps}\label{ref-spinor-descript-proj}
\subsubsection{$\simalg(n-2)$-invariant projection maps}
Here, we fix a null line $\mfK$ of an $n$-dimensional Minkowski space $(\mfV,g)$, so that $\mfV$ admits the $\simalg(n-2)$-invariant filtration \eqref{eq-sim-filtration}. Now choose an element $k^a$ of $\mfK$. The $\simalg(n-2)$-invariant submodules of $\g$, $\mfF$, $\mfA$ and $\mfC$ can then be described in terms of the kernels of the maps defined below. These project into the modules $\g_i^j$, $\mfF_i^j$, $\mfA_i^j$ and $\mfC_i^j$ on restriction to the associated graded $\simalg(n-2)$-modules $\gr(\g)$, $\gr(\mfF)$, $\gr(\mfA)$ and $\gr(\mfC)$ respectively. In the case of the Weyl tensor, these maps can be used to extend the Bel-Debever criterion of \cite{Ortaggio2009b}.
\paragraph{The Lie algebra $\so(n-1,1)$}
For $\phi \ind{_{ab}} \in \g$, we define
\begin{align*}
 {}^\g \Pi_{-1}^0 ( \phi ) & := k^c \phi \ind{_{c \lb{a}}} k \ind{_{\rb{b}}} \, , &
 {}^\g \Pi_0^0 ( \phi ) & := \phi \ind{_a^b} k^a \, , &
 {}^\g \Pi_0^1 ( \phi ) & := \phi \ind{_{\lb{a}b}} k \ind{_{\rb{c}}} \, .
\end{align*}
and when $n=6$, ${}^\g \Pi_0^{1,\pm} (\phi) := k_{[d} \phi_{ef]} \pm \frac{1}{3!} k_a \phi_{bc} \varepsilon \ind{^{abc}_{def}}$, where $\varepsilon_{abcdef}$ is a normalised volume $6$-form.

\paragraph{The tracefree Ricci tensor}
For $\Phi \ind{_{ab}} \in \mfF$, we define
\begin{align*}
 {}^\mfF \Pi_{-2}^0 (\Phi) & := k \ind{^a} k \ind{^b} \Phi \ind{_{ab}} \, , &
 {}^\mfF \Pi_{-1}^0 (\Phi) & := k \ind{^c}  \Phi \ind{_{c[a}} k \ind{_{b]}} \, , \\
 {}^\mfF \Pi_0^1 (\Phi) & := k \ind{_{[a}}  \Phi \ind{_{b][c}} k \ind{_{d]}} + \frac{1}{n-2} \left( k \ind{_{[a}} g \ind{_{b][c}} \Phi \ind{_{d]e}} k \ind{^e} + k \ind{_{[c}} g \ind{_{d][a}} \Phi \ind{_{b]e}} k \ind{^e} \right) \, , &
 {}^\mfF \Pi_0^0 (\Phi) & := k \ind{^b} \Phi \ind{_{ab}}  \, , \\
 {}^\mfF \Pi_1^0 (\Phi) & := k \ind{_{[a}} \Phi \ind{_{b]c}}  \, .
\end{align*}

\paragraph{The Cotton-York tensor}
For $A \ind{_{abc}} \in \mfA$, we define
\begin{align*}
 {}^\mfA \Pi_{-2}^0 (A) & :=  k \ind{^c} k \ind{^d} A \ind{_{cd[a}} k \ind{_{b]}} \, , \\
 {}^\mfA \Pi_{-1}^0 (A) & := k \ind{^b} k \ind{^c} A \ind{_{bca}}  \, , \\
 {}^\mfA \Pi_{-1}^1 (A) & := k \ind{_{[a}} A \ind{_{b]e[c}} k \ind{_{d]}} k \ind{^e} - k \ind{_{[c}} A \ind{_{d]e[a}} k \ind{_{b]}} k \ind{^e} + \frac{1}{n-2} \left( k \ind{_{[a}} g \ind{_{b][c|}} k \ind{^e} k \ind{^f} A \ind{_{ef|d]}} - k \ind{_{[c}} g \ind{_{d][a|}} k \ind{^e} k \ind{^f} A \ind{_{ef|b]}} \right) \, , \\
 {}^\mfA \Pi_{-1}^2 (A) & :=  k \ind{_{[a}} A \ind{_{b]e[c}} k \ind{_{d]}} k \ind{^e} + k \ind{_{[c}} A \ind{_{d]e[a}} k \ind{_{b]}} k \ind{^e} + \frac{1}{n-2} \left( k \ind{_{[a}} g \ind{_{b][c|}} k \ind{^e} k \ind{^f} A \ind{_{ef|d]}} + k \ind{_{[c}} g \ind{_{d][a|}} k \ind{^e} k \ind{^f} A \ind{_{ef|b]}} \right) \, , \\
 {}^\mfA \Pi_0^0 (A) & := 2 k \ind{^d} A \ind{_{ad[b}} k \ind{_{c]}} - k \ind{_a} k \ind{^d} A \ind{_{dbc}} \, , \\
 {}^\mfA \Pi_0^1 (A) & := 2 k \ind{^d} A \ind{_{ad[b}} k \ind{_{c]}} + k \ind{_a} k \ind{^d} A \ind{_{dbc}}   \, , \\
 {}^\mfA \Pi_0^2 (A) & := k \ind{_{[a}} A \ind{_{b][cd}} k \ind{_{e]}} - \frac{1}{n-3} g \ind{_{[c|[a}} \left( 2 k \ind{^f} A \ind{_{b]f|d}} k \ind{_{e]}} - k \ind{_{b]}} A \ind{_{f|de]}} k \ind{^f} \right) - \frac{2}{(n-2)(n-3)} g \ind{_{[c|[a}} g \ind{_{b]|d}} k \ind{^f} k \ind{^g} A \ind{_{fg|e]}} \, , \\
 {}^\mfA \Pi_1^0 (A) & := A \ind{_{abc}} k \ind{^c} \, , \\
 {}^\mfA \Pi_1^1 (A) & := k \ind{_{[a}} A \ind{_{b]cd}} - k \ind{_{[c}} A \ind{_{d]ab}} + \frac{2}{n-2} \left( g \ind{_{[a|[c}} A \ind{_{d]|b]e}} k \ind{^e} - g \ind{_{[c|[a}} A \ind{_{b]|d]e}} k \ind{^e} \right) \, , \\
 {}^\mfA \Pi_1^2 (A) & := k \ind{_{[a}} A \ind{_{b]cd}} + k \ind{_{[c}} A \ind{_{d]ab}} + \frac{2}{n-2} \left( g \ind{_{[a|[c}} A \ind{_{d]|b]e}} k \ind{^e} + g \ind{_{[c|[a}} A \ind{_{b]|d]e}} k \ind{^e} \right) \, , 
\end{align*}

\paragraph{The Weyl tensor}
For $C \ind{_{abcd}} \in \mfC$, we define
\begin{align*}
 {}^\mfC \Pi_{-2}^0 (C) & := k \ind{_{\lb{a}}} C \ind{_{\rb{b} e f \lb{c}}} k \ind{_{\rb{d}}} k \ind{^e} k \ind{^f} \, , \\
 {}^\mfC \Pi_{-1}^0 (C) & := C \ind{_{a d e \lb{b}}} k \ind{_{\rb{c}}} k \ind{^d} k \ind{^e} \, , \\
 {}^\mfC \Pi_{-1}^1 (C) & := k \ind{_{\lb{a}}} C \ind{_{b \rb{c} f \lb{d}}} k \ind{_{\rb{e}}} k \ind{^f} + \frac{2}{n-3} g \ind{_{\lb{a} | \lb{d}}} C \ind{_{\rb{e} f g |b}} k \ind{_{\rb{c}}} k \ind{^f} k \ind{^g} \, , \\
 {}^\mfC \Pi_0^0 (C) & := C \ind{_{a c d b}} k \ind{^c} k \ind{^d} \, , \\
 {}^\mfC \Pi_0^1 (C) & := k \ind{_{\lb{a}}} C \ind{_{b \rb{c} d e}} k \ind{^e} \, , \\
 {}^\mfC \Pi_0^2 (C) & := C \ind{_{a b e \lb{c}}} k \ind{_{\rb{d}}} k \ind{^e} + C \ind{_{c d e \lb{a}}} k \ind{_{\rb{b}}} k \ind{^e} - \frac{4}{n-2} g \ind{_{\lb{c}| \lb{a}}} C \ind{_{\rb{b} e f | \rb{d}}} k \ind{^e} k \ind{^f} \, , \\
 {}^\mfC \Pi_0^3 (C) & := k \ind{_{\lb{a}}} C \ind{_{b \rb{c} \lb{d} e}} k \ind{_{\rb{f}}} - \frac{2}{n-4} \left( g \ind{_{\lb{a}|\lb{d}}} C \ind{_{e\rb{f}|g|b}} k \ind{_{\rb{c}}} k^g + g \ind{_{\lb{d}|\lb{a}}} C \ind{_{b\rb{c}|g|e}} k \ind{_{\rb{f}}} k^g \right) \\
& \qquad \qquad \qquad + \frac{4}{9(n-3)(n-4)} \left( g \ind{_{d \lb{a}}} C \ind{_{\rb{b} g h \lb{e}}} g \ind{_{\rb{f}c}} + g \ind{_{d \lb{b}}} C \ind{_{\rb{c} g h \lb{e}}} g \ind{_{\rb{f}a}} + g \ind{_{d \lb{c}}} C \ind{_{\rb{a} g h \lb{e}}} g \ind{_{\rb{f}b}} \right. \\
& \qquad \qquad \qquad \qquad + \left. g \ind{_{e \lb{a}}} C \ind{_{\rb{b} g h \lb{f}}} g \ind{_{\rb{d}c}} + g \ind{_{e \lb{b}}} C \ind{_{\rb{c} g h \lb{f}}} g \ind{_{\rb{d}a}} + g \ind{_{e \lb{c}}} C \ind{_{\rb{a} g h \lb{f}}} g \ind{_{\rb{d}b}} \right. \\
& \qquad \qquad \qquad \qquad \qquad + \left. g \ind{_{f \lb{a}}} C \ind{_{\rb{b} g h \lb{d}}} g \ind{_{\rb{e}c}} + g \ind{_{f \lb{b}}} C \ind{_{\rb{c} g h \lb{d}}} g \ind{_{\rb{e}a}} + g \ind{_{f \lb{c}}} C \ind{_{\rb{a} g h \lb{d}}} g \ind{_{\rb{e}b}} \right) k \ind{^g} k \ind{^h} \, , \\
 {}^\mfC \Pi_1^0 (C) & := C \ind{_{a b c d}} k \ind{^d} \, , \\
 {}^\mfC \Pi_1^1 (C) & := k \ind{_{\lb{a}}} C \ind{_{b \rb{c} d e}} + \frac{2}{n-3} g \ind{_{\lb{a} | \lb{d}}} C \ind{_{\rb{e}|f| b \rb{c}}} k \ind{^f} \, , \\
 {}^\mfC \Pi_2^0 (C) & := C \ind{_{a b c d}} \, .
\end{align*}

\subsubsection{$\simalg(m-1,\C)$-invariant projection maps}
We now fix a Robinson structure $\mfN$ on $(\mfV,g)$ with associated real null line $\mfK$. In principle, given an irreducible $\so(n-1,1)$-module $\mfR$, say, one could define its $\simalg(m-1,\C)$-submodules with respect to $\mfN$ as kernels of projection maps by means of certain algebraic operations with the $\simalg(m-1,\C)$-invariant tensors $k^a$, $\varrho \ind{_{abc}}$ and $\mu \ind{_{ab}}$ defined in appendix \ref{sec-Robinson-forms}. If these maps are `saturated with symmetries', they will correspond to some irreducible $\simalg(m-1,\C)$-modules of an associated graded module.
In practice, the fact that we are dealing with a $3$-form $\varrho_{abc}$ makes this approach rather cumbersome. To illustrate the point, we consider the $\simalg(m-1,\C)$-invariant graph \eqref{eq-Penrose-diag-g} of the Lie algebra $\g$. We define, for $\phi_{ab} \in \g$, 
\begin{align*}
 {}^\g \Pi_0^{1,1} ( \phi ) & := \varrho \ind{_{ab}^f} \phi \ind{_{f \lb{c}}} k \ind{_{\rb{d}}} k \ind{_e} + \varrho \ind{_{cd}^e} \phi \ind{_{e \lb{a}}} k \ind{_{\rb{b}}} k \ind{_e} + \frac{2}{n-3} \left( k \ind{_{\lb{c}}} g \ind{_{\rb{d} \lb{a}}} \varrho \ind{_{\rb{b}}^{fg}} \phi \ind{_{fg}} + k \ind{_{\lb{a}}} g \ind{_{\rb{b} \lb{c}}} \varrho \ind{_{\rb{d}}^{fg}} \phi \ind{_{fg}} \right) k \ind{_e} \\
& \qquad \qquad  \qquad \qquad  \qquad \qquad  + \epsilon \left( \varrho \ind{_{ab}^f} \phi \ind{_{fg}} \mu \ind{^g_{[c}} \mu \ind{_{d]e}} + \varrho \ind{_{cd}^f} \phi \ind{_{fg}} \mu \ind{^g_{[a}} \mu \ind{_{b]e}} \right) \, , \\
 {}^\g \Pi_0^{1,2} ( \phi ) & := \varrho \ind{_{ab}^e} \phi \ind{_{ef}} \varrho \ind{^f_{cd}} + 4 \, k \ind{_{\lb{a}}} \phi \ind{_{\rb{b} \lb{c}}} k \ind{_{\rb{d}}}  - 2 \, \epsilon \left( \, \mu \ind{_{ab}} \mu \ind{_{[c}^e} \phi \ind{_{d]e}} - \, \mu \ind{_{cd}} \mu \ind{_{[a}^e} \phi \ind{_{b]e}} \right) \, ,
\end{align*}
and in addition, in odd dimensions,
\begin{align*}
 {}^\g \Pi_{-1}^{0,0} ( \phi ) & := 4 k^c \phi \ind{_{c \lb{a}}} k \ind{_{\rb{b}}} + \mu \ind{_{ab}} \phi \ind{_{cd}} \mu \ind{^{cd}} \, , &
 {}^\g \Pi_{-1}^{0,1} ( \phi ) & := \mu \ind{^{ab}} \phi \ind{_{ab}} \, , \\
 {}^\g \Pi_1^{1,3} ( \phi ) & := k \ind{_{[a}} \phi \ind{_{b]d}} \mu \ind{^d_c} \, ,  &
 {}^\g \Pi_1^{0,1} ( \phi ) & := \phi \ind{_{ac}} \mu \ind{^c_b} \, , &
 {}^\g \Pi_1^{0,0} ( \phi ) & := \phi \ind{_{[ab}} \mu \ind{_{c]d}} \, .
\end{align*}

To give another example, one can give explicit expressions for the maps ${}^\mfC \Pi_i^{j,k}$ of Proposition \eqref{prop-int-cond-Robinson}, by using the spinorial expression for the integrability condition given in \cites{Hughston1988,Taghavi-Chabert2012a,Taghavi-Chabert2013}. When $\epsilon=0$, we can recast the integrability condition \eqref{eq-int-cond-Rob} as \cite{Taghavi-Chabert2012a}
\begin{align}\label{eq-int-spinor-even}
C \ind{_{abcd}} \xi \ind{^a^A} \xi \ind{^b^B} \xi \ind{^c^C} \xi \ind{^d^D} & = 0 \, ,
\end{align}
where $\xi^{A'}$ is the spinor field annihilating the Robinson structure -- we refer the reader to appendix \ref{sec-spinors} for the notation. Contracting each free spinor index with $\bar{\xi} \ind{_{ab}_A}$ in equation \eqref{eq-int-spinor-even}, and using
the identity
\begin{align*}
 \xi \ind{_a^C} \bar{\xi} \ind{_{bc}_C} & = - \ii \varrho \ind{_{abc}} - 2 g \ind{_{a[b}} k \ind{_{c]}} \, ,
\end{align*}
we find\footnote{Equation \eqref{eq-int-Robinson-even} can be compared with its Hermitian analogue given in \cites{Gray1976,Tricerri1981,Falcitelli1994}. If $J \ind{_a^b}$ is a metric-compatible complex structure, then its integrability condition is given by
\begin{align*}
J\ind{_e^a} J \ind{_f^b} J \ind{_g^c} J \ind{_h^d}  C \ind{_{a b c d}}
- J \ind{_e^a} J \ind{_f^b} C \ind{_{a b g h}} - C \ind{_{e f c d}} J \ind{_g^c} J \ind{_h^d} + C \ind{_{e f g h}} 
- 4 \, J \ind{_{\lb{e}}^a} C \ind{_{\rb{f} a c \lb{g}}} J \ind{_{\rb{h}}^c}  = 0 \, ,
\end{align*}
which can also be derived from the spinorial expression \eqref{eq-int-spinor-even} with suitable reality conditions.}
\begin{multline}\label{eq-int-Robinson-even}
{}^\mfC \Pi_0^{3,3} (C) := \rho \ind{_{ab}^i} \rho \ind{_{cd}^j} \rho \ind{_{ef}^k} \rho \ind{_{gh}^l}  C \ind{_{i j k l}}
+ 4 \, \rho \ind{_{ab}^i} \rho \ind{_{cd}^j} C \ind{_{i j \lb{e}| \lb{g}}} k \ind{_{\rb{h}}} k \ind{_{|\rb{f}}} + 4 \, k \ind{_{\lb{a}|}} k \ind{_{\lb{c}}} C \ind{_{\rb{d}|\rb{b} k l}} \rho \ind{_{ef}^k} \rho \ind{_{gh}^l} \\
+ 16 \, k \ind{_{\lb{a}|}} k \ind{_{\lb{c}}} C \ind{_{\rb{d}|\rb{b} \lb{e}| \lb{g}}} k \ind{_{\rb{h}}}  k \ind{_{|\rb{f}}} 
- 4 \, \rho \ind{_{ab}^j} k \ind{_{\lb{c}}} C \ind{_{\rb{d} j k \lb{e}}} k \ind{_{\rb{f}}} \rho \ind{_{gh}^k} + 4 \, \rho \ind{_{cd}^j} k \ind{_{\lb{a}}} C \ind{_{\rb{b} j k \lb{e}}} k \ind{_{\rb{f}}} \rho \ind{_{gh}^k} \\
+ 4 \, \rho \ind{_{ab}^j} k \ind{_{\lb{c}}} C \ind{_{\rb{d} j k \lb{g}}} k \ind{_{\rb{h}}} \rho \ind{_{ef}^k} - 4 \, \rho \ind{_{cd}^j} k \ind{_{\lb{a}}} C \ind{_{\rb{b} j k \lb{g}}} k \ind{_{\rb{h}}} \rho \ind{_{ef}^k} = 0 \, .
\end{multline}
The other cases are similar.

\subsection{Representatives}\label{ref-spinor-descript-rep}
\subsubsection{$\co(n-2)$-representatives}
Assume the existence of a null line $\mfK$ on $(\mfV,g)$, with generator $k^a$, together with the choice of a vector $\ell^a$ dual to $k^a$ with $k^a \ell_a=1$. In other words, $\mfV$ is equipped with a $\simalg(n-2)$-invariant filtration \eqref{eq-sim-filtration} and a $\co(n-2)$-invariant grading \eqref{eq-sim-grading} with $k^a \in \mfV_1$ and $\ell^a \in \mfV_{-1}$. The irreducible $\simalg(n-2)$-modules $\g_i^j$, $\mfF_i^j$, $\mfA_i^j$ and $\mfC_i^j$ are linearly isomorphic to irreducible $\co(n-2)$-modules $\breve{\g}_i^j$, $\breve{\mfF}_i^j$, $\breve{\mfA}_i^j$ and $\breve{\mfC}_i^j$ respectively. To describe elements of these modules, we introduce the $\co(n-2)$-invariant tensors $E_{ab} := -2 k _{[a} \ell _{b]}$, $S_{ab} := 2 k_{(a} \ell_{b)}$, and $h_{ab} := g_{ab} - S_{ab}$.
\paragraph{The Lie algebra $\so(n-1,1)$}
 Let $\phi \ind{_{ab}} \in \g$. Then
\begin{itemize}
 \item $\phi \ind{_{ab}} \in \breve{\g}_1^0$ if and only if $\phi \ind{_{ab}} = 2 \, k \ind{_{[a}} \phi \ind{_{b]}}$ for some $\phi \ind{_a}$ such that $\phi \ind{_a} k \ind{^a} = \phi \ind{_a} \ell \ind{^a} =0$;
 \item $\phi \ind{_{ab}} \in \breve{\g}_0^0$ if and only if $\phi \ind{_{ab}} = \phi \, E \ind{_{ab}}$ for some real $\phi$;
 \item $\phi \ind{_{ab}} \in \breve{\g}_0^1$ if and only if $\phi \ind{_{ab}} k \ind{^a} = \phi \ind{_{ab}} \ell \ind{^a} = 0$.
\end{itemize}
Elements of $\breve{\g}_{-1}^0$ can be obtained from $\breve{\g}_1^0$ by interchanging $k^a$ and $\ell^a$.

\paragraph{The tracefree Ricci tensor}
 Let $\Phi_{ab} \in \mfF$. Then
\begin{itemize}
 \item $\Phi_{ab} \in \breve{\mfF}_2^0$ if and only if $\Phi \ind{_{ab}} = \Phi \, k \ind{_a} k \ind{_b}$ for some real $\Phi$;
 \item $\Phi_{ab} \in \breve{\mfF}_1^0$ if and only if $\Phi \ind{_{ab}} = 2 \, k \ind{_{(a}} \Phi \ind{_{b)}}$ for some $\Phi \ind{_a} \in \mfV_0$;
 \item $\Phi_{ab} \in \breve{\mfF}_0^0$ if and only if $\Phi \ind{_{ab}} = \Phi \left( S \ind{_{ab}} - \frac{2}{n-2} h \ind{_{ab}} \right)$ for some real $\Phi$;
 \item $\Phi_{ab} \in \breve{\mfF}_0^1$ if and only if $\Phi \ind{_{ab}} k \ind{^a} = \Phi \ind{_{ab}} \ell \ind{^a}=0$;
\end{itemize}
Elements of $\breve{\mfF}_{-i}^j$ can be obtained from those of $\breve{\mfF}_i^j$ by interchanging $k^a$ and $\ell^a$.

\paragraph{The Cotton-York tensor}
 Let $A_{abc} \in \mfA$. Then
\begin{itemize}
  \item $A_{abc} \in \breve{\mfA}_2^0$ if and only if $A \ind{_{abc}} = 2 \, k \ind{_a} k \ind{_{[b}} A \ind{_{c]}}$ for some $A \ind{_c} \in \mfV_0$;
  \item $A_{abc} \in \breve{\mfA}_1^0$ if and only if $A \ind{_{abc}} = a \left( k \ind{_a} E \ind{_{bc}} - \frac{2}{n-2} h \ind{_{a[b}} k \ind{_{c]}} \right)$ for some real $a$;
  \item $A_{abc} \in \breve{\mfA}_1^1$ if and only if $A \ind{_{abc}} = k \ind{_a} A \ind{_{bc}} - A \ind{_{a[b}} k \ind{_{c]}}$ for some $A_{ab}=A_{[ab]} \in \breve{\g}_0^1$;
  \item $A_{abc} \in \breve{\mfA}_1^2$ if and only if $A \ind{_{abc}} = 2 \, A \ind{_{a[b}} k \ind{_{c]}} $ for some $A \ind{_{ab}}=A \ind{_{(ab)}} \in \breve{\mfF}_0^1$;
  \item $A_{abc} \in \breve{\mfA}_0^0$ if and only if $A \ind{_{abc}} = A \ind{_a} E \ind{_{bc}} - E \ind{_{a[b}} A \ind{_{c]}}$ for some $A \ind{_c} \in \mfV_0$;
  \item $A_{abc} \in \breve{\mfA}_0^1$ if and only if $A \ind{_{abc}} = S \ind{_{a[b}} A \ind{_{c]}} - \frac{2}{n-3} h \ind{_{a[b}} A \ind{_{c]}} $ for some $A \ind{_c} \in \mfV_0$;
  \item $A_{abc} \in \breve{\mfA}_0^2$ if and only if $A \ind{_{abc}} k \ind{^a} = A \ind{_{abc}} \ell \ind{^a} = 0$;
\end{itemize}
Elements of $\breve{\mfA}_{-i}^j$ can be obtained from those of $\breve{\mfA}_i^j$ by  interchanging $k^a$ and $\ell^a$.

\paragraph{The Weyl tensor}
 Let $C_{abcd} \in \mfC$. Then
\begin{itemize}
 \item $C_{abcd} \in \breve{\mfC}_0^0$ if and only if
\begin{align*}
 C \ind{_{abcd}} & = \Psi \left( 4 \, E \ind{_{ab}} E \ind{_{cd}} - 4 \, E \ind{_{\lb{a}|\lb{c}}} E \ind{_{\rb{d}|\rb{b}}} 
+ \frac{6}{n-2} \left( h \ind{_{\lb{a} | \lb{c}}} S \ind{_{\rb{d}|\rb{b}}} \right) - \frac{12}{(n-2)(n-3)} \left( h \ind{_{a \lb{c}}} h \ind{_{\rb{d}b}} \right) \right) \, ,
\end{align*}
for some real $\Psi$;
 \item $C_{abcd} \in \breve{\mfC}_0^1$ if and only if $C_{abcd} = 2 \, E \ind{_{ab}} \Psi \ind{_{cd}} + 2 \, E \ind{_{cd}} \Psi \ind{_{ab}} - 4 \, E \ind{_{\lb{a}|\lb{c}}} \Psi \ind{_{\rb{d}|\rb{b}}}$ for some $\Psi \ind{_{ab}}= \Psi \ind{_{[ab]}} \in \breve{\g}_0^1$;
 \item $C_{abcd} \in \breve{\mfC}_0^2$ if and only if $C_{abcd} = 2 \, S \ind{_{\lb{a}|[c}} \Psi_{d]|\rb{b}} - \frac{4}{n-4} \left( h \ind{_{\lb{a} | \lb{c}}} \Psi \ind{_{\rb{d}|\rb{b}}}\right)$ for some $\Psi \ind{_{ab}}= \Psi \ind{_{(ab)}} \in \breve{\mfF}_0^1$;
 \item $C_{abcd} \in \breve{\mfC}_0^3$ if and only if $C_{abcd}k^a=C_{abcd}\ell^a=0$;
 \item $C_{abcd} \in \breve{\mfC}_1^0$ if and only if
\begin{align*} 
 C_{abcd} & = 2 \, k \ind{_{\lb{a}}} \Psi_{\rb{b}} E_{cd} + 2 \, k \ind{_{\lb{c}}} \Psi_{\rb{d}} E_{ab} - \frac{4}{n-3} \left( h \ind{_{\lb{a}|\lb{c}}} \Psi \ind{_{\rb{d}}} k \ind{_{|\rb{b}}} + h \ind{_{\lb{c}|\lb{a}}} \Psi \ind{_{\rb{b}}} k \ind{_{|\rb{d}}} \right)
\end{align*}
for some $\Psi_{a} \in \mfV_0$ satisfying $\Psi_ak^a=\Psi_a\ell^a=0$;
 \item $C_{abcd} \in \breve{\mfC}_1^1$ if and only if $C_{abcd} = k \ind{_{\lb{a}}} \Psi_{\rb{b}cd} + k \ind{_{\lb{c}}} \Psi_{\rb{d}ab}$
for some $\Psi_{abc}=\Psi_{a[bc]} \in \breve{\mfA}_0^2$;
 \item $C_{abcd} \in \breve{\mfC}_2^0$ if and only if $C_{abcd} = k \ind{_{\lb{a}}} \Psi_{\rb{b}\lb{c}} k \ind{_{\rb{d}}}$ for some $\Psi \ind{_{ab}}= \Psi \ind{_{(ab)}} \in \breve{\mfF}_0^1$.
\end{itemize}
Elements of $\breve{\mfC}_{-i}^j$ can be obtained from those of $\breve{\mfC}_i^j$ by interchanging $k^a$ and $\ell^a$.

\subsubsection{$\cu(m-1)$-representatives}\label{sec-spinor-rep}
Here, we assume the existence of a Robinson structure $\mfN$ with associated real null line $\mfK$ on $(\mfV,g)$ and a dual Robinson structure $\mfN^*$ with associated real null line $\mfL$ dual to $\mfK$, as described in appendix \ref{sec-spinors}. This fixes linear isomorphisms from the irreducible $\simalg(m-1,\C)$-modules $\g_i^{j,k}$, $\mfF_i^{j,k}$, $\mfA_i^{j,k}$, and $\mfC_i^{j,k}$ to irreducible $\cu(m-1)$-modules $\breve{\g}_i^{j,k}$, $\breve{\mfF}_i^{j,k}$, $\breve{\mfA}_i^{j,k}$, and $\breve{\mfC}_i^{j,k}$. To describe elements of these $\cu(m-1)$-modules, one simply decompose the elements of the irreducible $\co(n-2)$-modules given in the previous section further using the spinor calculus of appendix \ref{sec-spinors}.

In fact, it is enough to consider the decomposition of the modules $\breve{\g}_0^1$, $\breve{\mfF}_0^1$, $\breve{\mfA}_0^2$ and $\breve{\mfC}_0^3$, which may be seen as the buildling blocks of the other irreducible $\cu(m-1)$-modules appearing in this article -- see the tables of section \ref{sec-curvature}.
More specifically, we have
\begin{align*}
\breve{\g}_{\pm1}^{0,i} & \cong \mfV_{\pm1} \otimes \breve{\mfV}_0^{0,i} \, , \\
\breve{\mfF}_{\pm1}^{0,i} & \cong \mfV_{\pm1} \otimes \breve{\mfV}_0^{0,i} \, , \\
\breve{\mfA}_{\pm2}^{0,i} & \cong \mfV_{\pm1} \otimes \mfV_{\pm1} \otimes \breve{\mfV}_0^{1,i} \, , &
\breve{\mfA}_{\pm1}^{1,i} & \cong \mfV_{\pm1} \otimes \breve{\g}_0^{1,i} \, , &
\breve{\mfA}_{\pm1}^{2,i} & \cong \mfV_{\pm1} \otimes \breve{\mfF}_0^{1,i} \, , \\
\breve{\mfC}_{\pm1}^{0,i} & \cong \mfV_{\pm1} \otimes \breve{\mfV}_0^{0,i} \, , &
\breve{\mfC}_{\pm1}^{1,i} & \cong \mfV_{\pm1} \otimes \breve{\mfA}_0^{2,i} \, , &
\breve{\mfC}_0^{1,i} & \cong \breve{\g}_0^{1,i} \, , &
\breve{\mfC}_0^{2,i} & \cong \breve{\mfF}_0^{1,i} \, .
\end{align*}
For instance, from Proposition \ref{prop-CY-rob}, we know that the $\cu(m-1)$-module $\breve{\mfA}_1^{1,1}$ is isomorphic to $\mfV_1 \otimes \breve{\g}_0^{1,2}$. Thus, an element $A \ind{_{abc}}$ of $\breve{\mfA}_1^{1,1}$ is simply given by the representative $A \ind{_{abc}} = k \ind{_a} A \ind{_{bc}} - A \ind{_{a[b}} k \ind{_{c]}}$, since $A \ind{_{abc}}$ belongs to $\breve{\mfA}_1^1$, where $A_{ab}=A_{[ab]}$ belongs to the irreducible $\cu(m-1)$-module  $\breve{\g}_0^{1,2}$.

We shall use the spinor notation introduced in appendix \ref{sec-spinors} to which the reader is referred. In particular, upstairs and downstairs upper case Roman indices will always refer to the $(m-1)$-dimensional subspaces $\mfS_{\frac{1}{2}} ^{(0,1)}$ and $\mfS_{-\frac{1}{2}} ^{(1,0)}$ of the spinor space respectively. We shall also use the short-hand $c.c.$ to refer to the complex conjugate of a (complex) tensor in the usual sense -- what this means `spinorially' is explained in appendix \ref{sec-spinors}. We shall use the $\uu(m-1)$-invariant tensors $J \ind{_a^b}$, $H_{ab}$, $\omega_{ab}=J \ind{_a^c} H_{cb}$ and $u_a$ that were defined in section \ref{sec-algebra}: $J \ind{_a^c} J \ind{_c^b} = - H _a^b$, $H_{ab} k^a = H_{ab} \ell^a = H_{ab} u^a = 0$, $u_a u^a = 1$. They can also be constructed from the choice of a pure spinor and its dual.

\paragraph{Spinorial representatives for $\breve{\g}_0^1$}
 Let $\phi_{ab} \in \breve{\g}_0^1$. Then
\begin{itemize}
 \item $\phi_{ab} \in \breve{\g}_0^{1,0}$ if and only if $\phi_{ab} = \phi \, \omega_{ab}$ for some real $\phi$;
 \item $\phi_{ab} \in \breve{\g}_0^{1,1}$ if and only if $\phi_{ab} = 2 \ii \, \xi \ind*{_a^B} \xi \ind*{_b^C} \phi \ind{_{BC}} + c.c.$
for some $\phi \ind{_{BC}} = \phi \ind{_{[BC]}}$;
 \item $\phi_{ab} \in \breve{\g}_0^{1,2}$ if and only if $\phi_{ab} = 2 \ii \, \xi \ind*{_{\lb{a}}^B} \bar{\eta} \ind{_{\rb{b} D}} \phi \ind{_B^D} + c.c.$ for some tracefree $\phi \ind{_B^D}$;
\item $\phi_{ab} \in \breve{\g}_0^{1,3}$ if and only if $\phi_{ab} = 2 u \ind{_{[a}} \phi \ind{_{b]}}$
for some $\phi \ind{_a} \in \mfV_0^{0,0}$;
\end{itemize}

\paragraph{Spinorial representatives for $\breve{\mfF}_0^1$}
Let $\Phi \ind{_{ab}} \in \breve{\mfF}_0^1$. Then
\begin{itemize}
 \item $\Phi_{ab} \in \breve{\mfF}_0^{1,0}$ if and only if $\Phi_{ab} = 2 \, \xi \ind*{_{(a}^A} \bar{\eta} \ind{_{b)}_B} \Phi \ind{_A^B} + c.c.$
for some tracefree $\Phi \ind{_A^B}$;
 \item $\Phi_{ab} \in \breve{\mfF}_0^{1,1}$ if and only if $\Phi_{ab} = \xi \ind*{_a^A} \xi \ind*{_b^B} \Phi \ind{_{A B}} + c.c.$ for some $\Phi \ind{_{AB}} = \Phi \ind{_{(AB)}}$;
 \item $\Phi_{ab} \in \breve{\mfF}_0^{1,2}$ if and only if $\Phi_{ab} = \Phi \left( u_a u_b - \frac{1}{2m-2} H_{ab} \right)$ for some real $\Phi$;
\item $\Phi_{ab} \in \breve{\mfF}_0^{1,3}$ if and only if $\Phi_{ab} = 2 \, u_{(a} \Phi_{b)}$ for some $\Phi_a \in \mfV_0^{0,0}$.
\end{itemize}

\paragraph{Spinorial representatives for $\breve{\mfA}_0^2$}
Let $A \ind{_{abc}} \in \breve{\mfA}_0^2$. Then
\begin{itemize}
 \item $A_{abc} \in \breve{\mfA}_0^{2,0}$ if and only if $A_{abc} = A_a \omega_{bc} - A_{[b} \omega_{c]a} + \frac{3}{2m-3} H_{a[b} J \ind{_{c]}^d} A_d$ for some $A_a \in \mfV_0^{0,0}$;

\item $A_{abc} \in \breve{\mfA}_0^{2,1}$ if and only if $A_{abc} = \xi \ind*{_a^A} \xi \ind*{_b^B} \xi \ind*{_b^C} A_{ABC} + c.c.$ for some $A_{ABC}=A_{A[BC]}$ such that $A_{[ABC]} = 0$;

\item $A_{abc} \in \breve{\mfA}_0^{2,2}$ if and only if $A_{abc} = \bar{\eta} \ind{_a_A} \xi \ind*{_b^B} \xi \ind*{_c^C}  A \ind{^A_{BC}} - \bar{\eta} \ind{_{[b}_A} \xi \ind*{_{c]}^B} \xi \ind*{_a^C}  A \ind{^A_{BC}} + c.c.$
for some tracefree $A \ind{^A_{BC}} = A \ind{^A_{[BC]}}$;

\item $A_{abc} \in \breve{\mfA}_0^{2,3}$ if and only if $A_{abc} = 2 \, \xi \ind*{_a^A} \xi \ind*{_{[b}^B} \bar{\eta} \ind{_{c]}_C}  A \ind{_{AB}^C} + c.c.$
for some tracefree $A \ind{_{AB}^C} = A \ind{_{(AB)}^C}$;

\item $A_{abc} \in \breve{\mfA}_0^{2,4}$ if and only if $A_{abc} = a \left( u_a \omega_{bc} - u_{[b} \omega_{c]a} \right)$ for some real $a$;

\item $A_{abc} \in \breve{\mfA}_0^{2,5}$ if and only if $A_{abc} = 2 \, u_a u_{[b} A_{c]} - \frac{2}{2m-3} H_{a[b} A_{c]}$ for some $A_a \in \mfV_0^{0,0}$;

\item $A_{abc} \in \breve{\mfA}_0^{2,6}, \, \breve{\mfA}_0^{2,7}$ if and only if $A_{abc} = u \ind{_a} A_{bc} - u \ind{_{[b}} A \ind{_{c]a}}$ for some $A_{ab} \in \breve{\g}_0^{1,1}, \, \breve{\g}_0^{1,2}$ respectively;

\item $A_{abc} \in \breve{\mfA}_0^{2,8}, \, \breve{\mfA}_0^{2,9}$ if and only if $A_{abc} = 2 \, A_{a[b} u \ind{_{c]}}$ for some $A_{ab} \in \breve{\mfF}_0^{1,0}, \, \breve{\mfF}_0^{1,1}$ respectively;

\end{itemize}

\paragraph{Spinorial representatives for $\breve{\mfC}_0^3$}
 Let $C_{abcd} \in \breve{\mfC}_0^3$. Then
\begin{itemize}
 \item $C_{abcd} \in \breve{\mfC}_0^{3,0}$ if and only if
$C \ind{_{abcd}} = \Psi \left( 2 \, \omega \ind{_{ab}} \omega \ind{_{cd}} - 2 \, \omega \ind{_{a\lb{c}}} \omega \ind{_{\rb{d}b}} 
- \frac{6}{2m-3} H \ind{_{a \lb{c}}} H \ind{_{\rb{d}b}} \right)$ for some real $\Psi$;

 \item $C_{abcd} \in \breve{\mfC}_0^{3,1},\breve{\mfC}_0^{3,2}$ if and only if
\begin{align*}
 C \ind{_{abcd}} & = \omega \ind{_{ab}} \Psi \ind{_{cd}} + \Psi \ind{_{ab}} \omega \ind{_{cd}} - 2 \, \omega \ind{_{\lb{a}|\lb{c}}} \Psi \ind{_{\rb{d}|\rb{b}}}
- \frac{6}{2m-4} \left( H \ind{_{\lb{a} | \lb{c}}} J \ind{_{\rb{d}}^e} \Psi \ind{_{|\rb{b}e}} + H \ind{_{\lb{c} | \lb{a}}} J \ind{_{\rb{b}}^e} \Psi \ind{_{|\rb{d}e}} \right)
\end{align*}
for some $\Psi \ind{_{cd}} = \Psi \ind{_{[cd]}} \in \breve{\g}_0^{1,1},  \breve{\g}_0^{1,2}$ respectively;

\item $C_{abcd} \in \breve{\mfC}_0^{3,3}$ if and only if $C_{abcd} = \xi \ind*{_a ^A} \xi \ind*{_b ^B} \xi \ind*{_c ^C} \xi \ind*{_d ^D} \Psi_{A B C D} + c.c.$ for some $\Psi _{ABCD} = \Psi _{[AB][CD]}$ satisfying $\Psi _{[ABC]D} = 0$;

\item $C_{abcd} \in \breve{\mfC}_0^{3,4}$ if and only if
\begin{align*}
  C_{abcd} = \xi \ind*{_a ^A} \xi \ind*{_b ^B} \bar{\eta} \ind{_c _C} \bar{\eta} \ind{_d _D} \Psi \ind{_{A B}^{C D}} + \xi \ind*{_c ^A} \xi \ind*{_d ^B} \bar{\eta} \ind{_a _C} \bar{\eta} \ind{_b _D} \Psi \ind{_{A B}^{C D}} - 2 \, \xi \ind*{_{\lb{a}|} ^A} \xi \ind*{_{\lb{c}} ^C} \bar{\eta} \ind{_{\rb{d} |} _D} \bar{\eta} \ind{_{\rb{b}} _B} \Psi \ind{_{A C}^{D B}} + c.c.
\end{align*}
for some tracefree $\Psi \ind{_{A C}^{D B}} = \Psi \ind{_{[A C]}^{[D B]}}$;

 \item $C_{abcd} \in \breve{\mfC}_0^{3,5}$ if and only if $C_{abcd} =  \xi \ind*{_{\lb{a}|} ^A} \xi \ind*{_{\lb{c}} ^C} \bar{\eta} \ind{_{\rb{d}|} _D} \bar{\eta} \ind{_{\rb{b}} _B} \Psi \ind{_{A C}^{D B}} + c.c.$ for some tracefree $\Psi \ind{_{A C}^{D B}} = \Psi \ind{_{(A C)}^{(D B)}}$;

\item $C_{abcd} \in \breve{\mfC}_0^{3,6}$ if and only if $C_{abcd} = \xi \ind*{_a ^A} \xi \ind*{_b ^B} \xi \ind*{_{\lb{c}} ^C} \bar{\eta} \ind{_{\rb{d}}_D} \Psi \ind{_{A B C}^D} + \xi \ind*{_c ^A} \xi \ind*{_d ^B} \xi \ind*{_{\lb{a}} ^C} \bar{\eta} \ind{_{\rb{b}}_D} \Psi \ind{_{A B C}^D} + c.c.$ for some tracefree $\Psi \ind{_{ABC}^D} = \Psi \ind{_{[AB]C}^D}$ satisfying $\Psi \ind{_{[ABC]}^D} = 0$;

\item $C_{abcd} \in \breve{\mfC}_0^{3,7}$ if and only if 
\begin{multline*}
 C_{abcd} = \omega \ind{_{ab}} \Psi \ind{_{\lb{c}}} u \ind{_{\rb{d}}} + \omega \ind{_{cd}} \Psi \ind{_{\lb{a}}} u \ind{_{\rb{b}}} - \omega \ind{_{\lb{a}|\lb{c}}} \Psi \ind{_{\rb{d}}} u \ind{_{|\rb{b}}} - \omega \ind{_{\lb{c}|\lb{a}}} \Psi \ind{_{\rb{b}}} u \ind{_{|\rb{d}}}  \\
+ \frac{3}{2m-3} \left( H \ind{_{\lb{a} | \lb{c}}} u \ind{_{\rb{d}}} J \ind{_{|\rb{b}}^e} \Psi \ind{_e}  + H \ind{_{\lb{c} | \lb{a}}} u \ind{_{\rb{b}}} J \ind{_{|\rb{d}}^e} \Psi \ind{_e}\right)
\end{multline*}
for some $\Psi \ind{_a} \in \breve{\mfV}_0^{0,0}$;

 \item $C_{abcd} \in \breve{\mfC}_0^{3,8},  \breve{\mfC}_0^{3,9}$ if and only if  $C_{abcd} = u \ind{_{\lb{a}}} \Psi \ind{_{\rb{b}\lb{c}}} u \ind{_{\rb{d}}} 
+ \frac{1}{2m-4} H \ind{_{\lb{a} | \lb{c}}} \Psi \ind{_{\rb{d}|\rb{b}}}$ for some $\Psi \ind{_{cd}} = \Psi \ind{_{(cd)}} \in \breve{\mfF}_0^{1,0}, \breve{\mfF}_0^{1,1}$ respectively;

\item $C_{abcd} \in \breve{\mfC}_0^{3,10}, \breve{\mfC}_0^{3,11}, \breve{\mfC}_0^{3,12}$ if and only if $C_{abcd} = u \ind{_{\lb{a}}} \Psi \ind{_{\rb{b}cd}} + u \ind{_{\lb{c}}} \Psi \ind{_{\rb{d}ab}}$ for some $\Psi \ind{_{abc}} = \Psi \ind{_{a[bc]}} \in \breve{\mfA}_0^{2,1}, \breve{\mfA}_0^{2,2}, \breve{\mfA}_0^{2,3}$ respectively.
\end{itemize}
\section{Low dimensions}\label{sec-low-dim}
There are a number of simplifications that can be made in the classification of curvature tensors in low dimensions, notably up to dimension six. In the following, we give a brief description of the special features of dimensions four, five and six.

\subsection{Four dimensions}
In four dimensions, a Robinson structure is equivalent to a choice of a null line $\mfK$, which is a reflection of the isomorphisms of Lie algebras $\simalg(2) \cong \simalg(1,\C)$ and $\so(2) \cong \uu(1)$. As a consequence, the $\simalg(1,\C)$-invariant classification of the Weyl tensor reduces to its $\simalg(2)$-classification. The corresponding $\simalg(m-1,\C)$-invariant graph then reads
\begin{align*}
 \xymatrix@R=.1em{
 & & \mfC_0^1 \ar[dr] & & \\
 \mfC_2^0 \ar[r] & \mfC_1^0 \ar[dr] \ar[ur] & & \mfC_{-1}^0 \ar[r] & \mfC_{-2}^0 \, . \\
 & & \mfC_0^0 \ar[ur] & & }
\end{align*}
This graph differs markedly from the usual Penrose flow diagram of the Petrov classification of the Weyl tensor. Similar remarks on the classification of the Cotton-York tensor apply.

\subsection{Five dimensions}
In five dimensions, $\mfV_0$ is three-dimensional, so we can identify the Hermitian $2$-form on $\mfV_0$ with the unit vector orthogonal to it. Thus, a Robinson structure can be determined by a choice of a null line $k^a$ and an equivalence class of spacelike subspaces of $\mfV^0/\mfV^1$. For clarity, the diagram below is the five-dimensional specialisation of diagram \ref{diagram-Penrose-C-rob}:
\begin{align*}
\xy
(-80,20)*+{\mfC_2^{0,1}}="s51",
(-80,0)*+{\mfC_2^{0,3}}="s53",
(-80,-20)*+{\mfC_2^{0,2}}="s52",
(-50,40)*+{\mfC_1^{1,9}}="s49",
(-50,20)*+{\mfC_1^{1,5}}="s45",
(-50,0)*+{\mfC_1^{1,4}}="s44",
(-50,-20)*+{\mfC_1^{0,0}}="s38",
(-50,-40)*+{\mfC_1^{0,1}}="s39",
(0,50)*+{\mfC_0^{2,1}}="s22",
(0,30)*+{\mfC_0^{2,3}}="s24",
(0,10)*+{\mfC_0^{2,2}}="s23",
(0,-10)*+{\mfC_0^{1,3}}="s20",
(0,-30)*+{\mfC_0^{1,0}}="s17",
(0,-50)*+{\mfC_0^{0,0}}="s16",
(50,40)*+{\mfC_{-1}^{1,9}}="s15",
(50,20)*+{\mfC_{-1}^{1,5}}="s11",
(50,0)*+{\mfC_{-1}^{1,4}}="s10",
(50,-20)*+{\mfC_{-1}^{0,0}}="s4",
(50,-40)*+{\mfC_{-1}^{0,1}}="s5",
(80,20)*+{\mfC_{-2}^{0,1}}="s1",
(80,0)*+{\mfC_{-2}^{0,3}}="s3",
(80,-20)*+{\mfC_{-2}^{0,2}}="s2",
"s1"; "s4" ; **@{-} ?>*\dir{>}; "s1"; "s15" ; **@{-} ?>*\dir{>};
"s2"; "s5" ; **@{-} ?>*\dir{>}; "s2"; "s11" ; **@{-} ?>*\dir{>};
"s3"; "s4" ; **@{-} ?>*\dir{>}; "s3"; "s5" ; **@{-} ?>*\dir{>}; "s3"; "s10" ; **@{-} ?>*\dir{>}; "s3"; "s11" ; **@{-} ?>*\dir{>}; "s3"; "s15" ; **@{-} ?>*\dir{>};
"s4"; "s16" ; **@{-} ?>*\dir{>}; "s4"; "s17" ; **@{-} ?>*\dir{>}; "s4"; "s22" ; **@{-} ?>*\dir{>}; "s4"; "s20" ; **@{-} ?>*\dir{>}; "s4"; "s24" ; **@{-} ?>*\dir{>};
"s5"; "s16" ; **@{-} ?>*\dir{>}; "s5"; "s20" ; **@{-} ?>*\dir{>}; "s5"; "s23" ; **@{-} ?>*\dir{>}; "s5"; "s24" ; **@{-} ?>*\dir{>}; 
"s10"; "s17" ; **@{-} ?>*\dir{>}; "s10"; "s24" ; **@{-} ?>*\dir{>};
"s11"; "s20" ; **@{-} ?>*\dir{>}; "s11"; "s23" ; **@{-} ?>*\dir{>}; "s11"; "s24" ; **@{-} ?>*\dir{>};
"s15"; "s20" ; **@{-} ?>*\dir{>}; "s15"; "s22" ; **@{-} ?>*\dir{>}; "s15"; "s24" ; **@{-} ?>*\dir{>};
"s16"; "s38" ; **@{-} ?>*\dir{>}; "s16"; "s39" ; **@{-} ?>*\dir{>};
"s17"; "s38" ; **@{-} ?>*\dir{>}; "s17"; "s44" ; **@{-} ?>*\dir{>};
"s20"; "s38" ; **@{-} ?>*\dir{>}; "s20"; "s39" ; **@{-} ?>*\dir{>}; "s20"; "s45" ; **@{-} ?>*\dir{>}; "s20"; "s49" ; **@{-} ?>*\dir{>};
"s22"; "s38" ; **@{-} ?>*\dir{>}; "s22"; "s49" ; **@{-} ?>*\dir{>};
"s23"; "s39" ; **@{-} ?>*\dir{>}; "s23"; "s45" ; **@{-} ?>*\dir{>};
"s24"; "s38" ; **@{-} ?>*\dir{>}; "s24"; "s49" ; **@{-} ?>*\dir{>};
"s24"; "s39" ; **@{-} ?>*\dir{>}; "s24"; "s44" ; **@{-} ?>*\dir{>}; "s24"; "s45" ; **@{-} ?>*\dir{>};
"s38"; "s51" ; **@{-} ?>*\dir{>}; "s38"; "s53" ; **@{-} ?>*\dir{>};
"s39"; "s52" ; **@{-} ?>*\dir{>}; "s39"; "s53" ; **@{-} ?>*\dir{>};
"s44"; "s53" ; **@{-} ?>*\dir{>};
"s45"; "s52" ; **@{-} ?>*\dir{>}; "s45"; "s53" ; **@{-} ?>*\dir{>};
"s49"; "s51" ; **@{-} ?>*\dir{>}; "s49"; "s53" ; **@{-} ?>*\dir{>};
\endxy  
\end{align*}

\subsection{Six dimensions}
In six dimensions, the semi-simple part $\so(4,\R)$ in the Levi decomposition of $\simalg(4)$ splits further as $\so(4,\R) \cong \su^+(2) \oplus \su^-(2)$, where $\su^\pm(2)$ are two copies of the simple Lie algebra $\su(2)$, which can be identified with the self-dual part and anti-self-dual part of $\so(4,\R)$. Here, the Hodge duality operator on $\so(4,\R)$ is given by $\varepsilon \ind{_{abcd}^{ef}} E \ind{_{ef}}$ for any choice of grading element $E \ind{_{ab}}$.

Such a splitting of $\simalg(4)$ is automatically taken care in the $\simalg(2,\C)$ decomposition of $\g$. To be precise, we can make the identifications $\su^+(2) = \dbl \g_0^{(2,0)} \dbr \oplus \g_0^\omega$ and $\su^-(2) = [ \g_0^{(1,1)_\circ} ]$. We can thus `regroup' some of the $\simalg(2,\C)$-modules into self-dual and anti-self-dual $\simalg(4)$-modules occuring in the classifications of the Cotton-York tensor and of the Weyl tensor as follows:
 \begin{align*}
   \mfA_{\pm1}^{1,+} & \cong \mfA_{\pm1}^{1,0} \oplus \mfA_{\pm1}^{1,2}  \, , &
  \mfA_{\pm1}^{1,-} & \cong \mfA_{\pm1}^{1,1} \, , \\
  \mfA_0^{2,+} & \cong \mfA_0^{2,0} \oplus \mfA_0^{2,1}  \, , &
  \mfA_0^{2,-} & \cong \mfA_0^{2,3} \, , \\
  \mfC_{\pm1}^{1,+} & \cong \mfC_{\pm1}^{1,0} \oplus \mfC_{\pm1}^{1,1}  \, , &
  \mfC_{\pm1}^{1,-} & \cong \mfC_{\pm1}^{1,3} \, , \\
  \mfC_0^{1,+} & \cong \mfC_0^{1,0} \oplus \mfC_0^{1,1}  \, , &
  \mfC_0^{1,-} & \cong \mfC_0^{1,2} \, , \\
  \mfC_0^{3,+} & \cong \mfC_0^{3,0} \oplus \mfC_0^{3,1} \oplus \mfC_0^{3,3} \, , &
  \mfC_0^{3,-} & \cong \mfC_0^{3,5} \, .
 \end{align*}

\bibliography{biblio}

\end{document}